\documentclass[onecolumn,final,a4paper,sort&compress]{elsarticle}

\makeatletter
\def\ps@pprintTitle{%
\let\@oddhead\@empty
\let\@evenhead\@empty
\def\@oddfoot{}%
\let\@evenfoot\@oddfoot}
\makeatother

\usepackage{comment}
\usepackage{amsmath}       
\usepackage{amssymb}       
\usepackage{chapterbib}    
\usepackage{color}
\usepackage{comment}
\usepackage{bm,bbm}
\usepackage{latexsym}
\usepackage{overpic}
\usepackage{times}
\usepackage{epsfig}
\usepackage{graphicx,graphics,rotating}
\usepackage{subfigure}
\usepackage{multicol}
\usepackage{multirow}

\newtheorem{theorem}{Theorem}[section]
\newtheorem{lemma}[theorem]{Lemma}

\newtheorem{proposition}[theorem]{Proposition}

\newtheorem{remark}[theorem]{Remark}

\newtheorem{assumption}[theorem]{Assumption}

\definecolor{MyDarkGreen}{rgb}{0,0.45,0}

\newcommand{\BLUE}[1]{{#1}}
\newcommand{\RED} [1]{{#1}}

\def\trait #1 #2 #3 {\vrule width #1pt height #2pt depth #3pt}
\def\fin{\hfill
        \trait .3 5 0
        \trait 5 .3 0
        \kern-5pt
        \trait 5 5 -4.7
        \trait 0.3 5 0
\medskip}

\newenvironment{proof}{ \textit{Proof.} }{\fin}

\newcommand{\TERM} [1]{\mbox{$\mathbf{(#1)}$}}

\newcommand{\DOFS}[1]{($\mathbf{#1}$)}

\newcommand{\INTP}{\footnotesize{I}}

\newcommand{\REAL}{\mathbbm{R}}

\DeclareMathOperator{\dims}{\text{dim}\hspace{0.25mm}}
\DeclareMathOperator{\eqdef}{:=}

\newcommand{\restrict}[2]{{#1}_{|{#2}}}

\newcommand{\PGRAPH}[1]{\medskip\noindent\textbf{#1}.}
\newcommand{\EOD}{\end{document}}

\newcommand{\TLD}[1]{\widetilde{#1}}
\newcommand{\BAR}[1]{\overline{#1}}

\newcommand{\mv}{\mathbf{m}}

\newcommand{\uv}{\mathbf{u}}
\newcommand{\vv}{\mathbf{v}}

\newcommand{\xv}{\mathbf{x}}
\newcommand{\yv}{\mathbf{y}}

\newcommand{\Cv}{\mathbf{C}}

\newcommand{\Fv}{\mathbf{F}}

\newcommand{\Hv}{\mathbf{H}}

\newcommand{\Vv}{\mathbf{V}}

\newcommand{\as}{a}

\newcommand{\fs}{f}
\newcommand{\gs}{g}
\newcommand{\hs}{h}

\newcommand{\ks}{k}

\newcommand{\ms}{m}

\newcommand{\ps}{p}
\newcommand{\qs}{q}
\newcommand{\rs}{r}

\newcommand{\us}{u}
\newcommand{\vs}{v}
\newcommand{\ws}{w}
\newcommand{\xs}{x}
\newcommand{\ys}{y}

\newcommand{\Bs}{B}
\newcommand{\Cs}{C}
\newcommand{\Ds}{D}

\newcommand{\Fs}{F}

\newcommand{\Ss}{S}

\newcommand{\Vs}{V}

\newcommand{\astwo}{\calA_2}
\newcommand{\asone}{\calA_1}
\newcommand{\aszer}{\calA_0}
\newcommand{\astwoP}{\calA_2^{\P}}
\newcommand{\asoneP}{\calA_1^{\P}}
\newcommand{\aszerP}{\calA_0^{\P}} 
\newcommand{\Vszr}{V_{0}}

\newcommand{\matD}{\mathsf{D}}

\newcommand{\matI}{\mathsf{I}}

\newcommand{\calA}{\mathcal{A}}
\newcommand{\calB}{\mathcal{B}}

\newcommand{\calD}{\mathcal{D}}
\newcommand{\calE}{\mathcal{E}}
\newcommand{\calF}{\mathcal{F}}

\newcommand{\calL}{\mathcal{L}}

\newcommand{\calV}{\mathcal{V}}

\newcommand{\calDt}{\widetilde{\mathcal{D}}}
\newcommand{\calAP} {\mathcal{A}^{\P}}

\newcommand{\calBP}{\mathcal{B}^{\P}}

\newcommand{\HONE}  {H^1}

\newcommand{\LTWO}  {L^2}

\newcommand{\LS}[1] {L^{#1}}
\newcommand{\HS}[1] {H^{#1}}

\newcommand{\CS}[1] {C^{#1}}

\newcommand{\HSzr}[1] {H^{#1}_0}

\newcommand{\PS}[1] {\mathbbm{P}_{#1}}

\newcommand{\Bsh}[1]{\Bs_{#1}^{h}}
\newcommand{\Vsh}[1]{\Vs_{#1}^{h}}

\newcommand{\Vsht}[1]{\widetilde{\Vs}_{#1}^{h}}

\newcommand{\Vvhpp}[1]{\Vv^{\FT,h}_{k}}

\newcommand{\FT}{\textit{F2}}

\renewcommand{\P} {\textrm{E}}            
\newcommand  {\E} {\textrm{e}}
\newcommand  {\V} {\textrm{v}}            

\newcommand{\hh}{h}
\newcommand{\Th}{\Omega_{\hh}}

\newcommand{\xvP}{\xv_{\P}}        
\newcommand{\xvE}{\xv_{\E}}        
\newcommand{\xvV}{\xv_{\V}}        

\newcommand{\hP}{\hh_{\P}}

\newcommand{\hE}{\hh_{\E}}
\newcommand{\hV}{\hh_{\V}}

\newcommand{\mP}{\ABS{\P}}

\newcommand{\Eset}{\mathcal{E}}    
\newcommand{\Vset}{\mathcal{V}}    

\newcommand{\NMB}{N}

\newcommand{\NPV}{\NMB^{\Vset}_{\P}}      
\newcommand{\NPE}{\NMB^{\Eset}_{\P}}      

\newcommand{\dS}{\,ds}
\newcommand{\dxv}{\,d\xv}

\newcommand{\norP} {\mathbf{n}_{\P}}

\newcommand{\norPE}{\mathbf{n}_{\P,\E}}

\newcommand{\tngP} {\mathbf{t}_{\P}}

\newcommand{\tngPE}{\mathbf{t}_{\P,\E}} 

\newcommand{\half}{1\slash{2}}

\newcommand{\TILDE}[1]{\widetilde{#1}}

\newcommand{\bil}[2]{\langle#1,#2\rangle}
\newcommand{\bilP}[2]{\langle#1,#2\rangle_{\P}}

\newcommand{\ABS}   [1]{\left|#1\right|}

\newcommand{\snorm}  [2]{|#1|_{#2}}

\newcommand{\norm}  [2]{||#1||_{#2}}

\newcommand{\Piz}[1]{\Pi^{0}_{#1}}
\newcommand{\PinP}[1]{\Pi^{\nabla,\P}_{#1}}
\newcommand{\PizP}[1]{\Pi^{0,\P}_{#1}}

\newcommand{\PiLP}[1]{\Pi^{\calL,\P}_{#1}}

\newcommand{\matPiLP}[1]{\pmb{\Pi}^{\calL,\P}_{#1}}
\newcommand{\matPizP}[1]{\pmb{\Pi}^{0,\P}_{#1}}
\newcommand{\matPinP}[1]{\pmb{\Pi}^{\nabla,\P}_{#1}}

\newcommand{\SPh}{\Ss^{\P}_{\hh}}

\newcommand{\asP}{\as^{\P}}

\newcommand{\calAh} {\calA_{\hh}}
\newcommand{\calAhP}{\calA_{\hh}^{\P}}

\newcommand{\calBh} {\calB_{\hh}}
\newcommand{\calBhP}{\calB_{\hh}^{\P}}

\newcommand{\ush} {\us_{\hh}}
\newcommand{\usI} {\us_{\INTP}}

\newcommand{\vsh} {\vs_{\hh}}
\newcommand{\vsI} {\vs_{\INTP}}

\newcommand{\vsht}{\TILDE{\vs}_{\hh}}

\newcommand{\fsh} {\fs_{\hh}}

\newcommand{\uvh} {\uv_{\hh}}

\newcommand{\vvh} {\vv_{\hh}}

\newcommand{\xiI}{\xi_{\INTP}}

\newcommand{\ilev}{n}
\newcommand{\nR}   {N_{P}}
\newcommand{\nF}   {N_{F}}
\newcommand{\nV}   {N_{V}}
\newcommand{\ndofs}{\#\textrm{dofs}}
\newcommand{\hmax} {\hh}
\newcommand{\Ndfs} {\#\textbf{dofs}}
\newcommand{\Ndofs} {N^{\textrm{dofs}}}

\usepackage[capitalize]{cleveref}

\newcommand{\etal}{et al.}

\newcommand{\StiffnessMatrix}{\pmb{K}}
\newcommand{\MassMatrix}{\pmb{M}}

\newcommand{\PScard}{n_\ks}
\newcommand{\DOFcard}{N^\text{dofs}}

\newcommand{\MonomialVec}{\pmb{m}}
\newcommand{\VirtualSFVec}{\pmb{\varphi}}

\newcommand{\dofOp}[2]{\text{dof}_{#1}\left( #2 \right)}

\newcommand{\Iden}[1]{I_{#1}}

\newcommand{\transpose}[0]{^\text{\textbf{T}}}

\newcommand{\GMatrix}{\pmb{G}}
\newcommand{\BMatrix}{\pmb{B}}
\newcommand{\DMatrix}{\pmb{D}}

\def\EE{\mathcal{E}}
\def\inte{{\rm int}}
\def\bdry{{\rm bdry}}

\def\EEi{\EE_h^{\inte}}
\def\EEbdry{\EE_h^{\bdry}}
\def\VV{\mathcal{V}}
\def\VVi{\VV_h^{\inte}}
\def\VVbdry{\VV_h^{\bdry}}
\newcommand\bn{\boldsymbol{n}}
\newcommand\bt{\boldsymbol{t}}

\usepackage{xcolor}

\usepackage{soul}               
\usepackage[super]{nth}         

\usepackage[a4paper]{geometry}

\begin{document}
%
\begin{frontmatter} 
\title{A $C^1$-conforming arbitrary-order two-dimensional virtual
  element method for the fourth-order phase-field equation}


  \author[CHILE]  {Dibyendu Adak}
  \author[LANLT5] {Gianmarco Manzini\corref{CORR}}\ead{gmanzini@lanl.gov}
  \author[LANLT3] {Hashem M. Mourad}
  \author[LANLXCP]{JeeYeon N. Plohr}
  \author[LANLT3] {Lampros Svolos}
  
  \cortext[CORR]{Corresponding author}

  \address[CHILE]{GIMNAP, Departamento de Matem\'atica, Universidad del B\'io-B\'io, Concepci\'on, Chile}
  \address[LANLT5]{Applied Mathematics and Plasma Physics, T-5, Theoretical Division,  Los Alamos National Laboratory, Los Alamos, NM, USA}
  \address[LANLT3]{Fluid Dynamics and Solid Mechanics, T-3, Theoretical Division, Los Alamos National Laboratory, Los Alamos, NM, USA}
  \address[LANLXCP]{Materials and Physical Data, XCP-5, Computational Physics Division,  Los Alamos National Laboratory, Los Alamos, NM, USA}

  \begin{abstract}
    We present a two-dimensional conforming virtual element method for
    the fourth-order phase-field equation.
    Our proposed numerical approach to the solution of this high-order
    phase-field (HOPF) equation relies on the design of an
    arbitrary-order accurate, virtual element space with $\CS{1}$
    global regularity.
    Such regularity is guaranteed by taking the values of the virtual
    element functions and their full gradient at the mesh vertices as
    degrees of freedom.
    Attaining high-order accuracy requires also edge polynomial
    moments of the trace of the virtual element functions and their
    normal derivatives.
    In this work, we detail the scheme construction, and prove its
    convergence by deriving error estimates in different norms.
    A set of representative test cases allows us to assess the
    behavior of the method.
  \end{abstract}
  
  \begin{keyword}
    Two-dimensional phase field equation,
    virtual element method,
    error analysis.\\
    \emph{2020 Mathematics Subject Classification: Primary: 65M60, 65N30; Secondary: 65M22.}
  \end{keyword}

\end{frontmatter}


\renewcommand{\arraystretch}{1.}
\raggedbottom




\section{Introduction}
\label{sec1:intro}

Fracture is a critical failure mode that can cause a rapid loss of
load-carrying capacity and uncontrolled demolition of inhabited
structures.
The prevention of such catastrophic outcomes is the motivation to
develop computationally efficient models of fracture, which are also
capable of accurately capturing material behavior in the post-failure
regime.
The efficient solution to fracture problems remains one of the most
critical research priorities despite the progress made over the past
decades.
To this end, developing efficient numerical techniques with predictive
capabilities requires a multidisciplinary approach.

Over the past decade, phase field (PF) modeling of fracture has gained
significant attention due to its ability to capture complicated crack
patterns (e.g., merging or branching) by utilizing conventional finite
element techniques.
The first PF fracture method is traced back to Bourdin
\etal~\cite{bourdin_numerical_2000}, where a numerical implementation
of the variational approach to fracture
\cite{francfort_revisiting_1998} was introduced.
Inspired by Griffith's work, the variational
approach describes fracture as a minimization problem of a total
energy functional, which expresses the competition between bulk
elastic energy and crack surface energy.
By introducing an auxiliary field (denoted by $\us$ in this paper),
crack surfaces are represented by ``diffuse'' entities, obviating the
need for injecting discontinuities into the kinematic solution fields.
In the seminal work of Miehe  \etal~\cite{miehe2010thermodynamically},
a framework based on continuum mechanics and thermodynamic
arguments was proposed to model brittle fracture.
It was further developed to treat dynamic brittle failure in
\cite{borden2012phase}.
Moreover, the PF framework was extended to model cohesive fracture in
\cite{Vignollet-May-DeBorst-Verhoosel:2014} and
ductile fracture in
\cite{miehe_phase_2015}.
More recently, multiphysics coupling effects are studied in problems
of brittle fracture \cite{yan2021new}, cohesive fracture
\cite{rezaei2022anisotropic}, ductile fracture
\cite{dittmann2020phase,svolos2020thermal,svolos2021anisotropic},
and structural fragmentation \cite{moutsanidis2018hyperbolic}.
The interested reader is refereed to
\cite{egger2019discrete,rahimi2022modeling}
for more applications of the PF method and to
\cite{wu_chapter_2020} for an extensive review.

The numerical solution to PF fracture problems may require a highly
refined mesh to accurately resolve the steep PF gradients that develop
in the vicinity of a crack~\cite{wu_chapter_2020}, thus being
expensive and possibly non-competitive with respect to other
approaches despite all its modeling capabilities.
Adopting the popular second-order PF fracture model exacerbates this
issue, as it leads to the development of a cusp in the solution field
at the crack surface that negatively impacts the convergence of the
numerical solution when the mesh is refined.
To address this issue, Borden \etal~\cite{borden_higher-order_2014}
proposed a high-order phase-field (HOPF) model, where the solution
regularity is increased, and better spatial convergence behavior can
be achieved.
The HOPF model involves a trade-off, as it necessitates using
specialized computational techniques that ensure $\CS{1}$-continuity
of the PF unknown to attain convergence.
For example, the Finite Element Method (FEM) with a
$\CS{0}$-continuous Lagrange polynomial basis is not a good choice.
Its use to treat higher-order PF problems results in discontinuities
in the PF gradient.
Thus, it constitutes a variational crime~\cite{Strang:1972}, since the
weak form of such problems includes at least second-order spatial
derivatives, e.g., the Laplacian and the Hessian differential
operators.

High regularity of the numerical approximation is of primary
importance when dealing with high-order differential problems.
In addition, global smoothness can be utilized to directly compute
physical quantities (such as fluxes, strains, and stresses) without
resorting to post-processing as in the $\CS{0}$-FEM.
From the earliest works in the 1960s,
e.g.,~\cite{Argyris-Fried-Scharpf:1968,Bell:1969},
to the most current attempts,
e.g.,~\cite{Zhang:2016,Wu-Lin-Hu:2021}, there
are several examples of finite elements with regularity higher than
$\CS{0}$.
The construction of approximation spaces with such global regularity
has been seen as challenging since they require basis functions with
the same global regularity.
Designing approximations with such enhanced regularity is still an
active research topic.
A non-exhaustive list includes the FEM with Hermite polynomial basis
functions \cite{Stogner-Carey-Murray:2008} or B-spline basis functions
(i.e., isogeometric analysis)
\cite{Bartezzaghi-Dede-Quarterni:2015};
machine learning techniques (e.g. physics informed neural networks
\cite{Goswami-Anitescu-Rabczuk:2020}); fast
Fourier transform (FFT)-based methods
\cite{Ma-Sun:2020}; mixed FEM
\cite{Elliott-French-Milner:1989};
continuous/discontinuous Galerkin (C/DG) methods
\cite{Svolos-Mourad-Manzini-Garikipati:2022}; discontinuous Galerkin
(DG) methods \cite{Georgoulis-Houston:2009}.
Such approximations have a natural application in the numerical
treatment of problems involving high-order differential operators, as
in the HOPF problem.

The Virtual Element Method (VEM) was initially designed as a
Galerkin-type projection method to extend the FEM from
triangular/tetrahedral and quadrilateral/hexahedral meshes to
polytopal meshes.
\BLUE{The VEM does not need explicit knowledge of the basis functions
  spanning the approximation
  spaces~\cite{BeiraodaVeiga-Brezzi-Cangiani-Manzini-Marini-Russo:2013}.
Instead, local approximation spaces are constructed by solving a
partial differential equation at the element level. These local spaces
are subsequently ``glued'' together to form a highly regular,
conforming global approximation space.
The basis functions in this case are referred to as ``virtual''
because they are not computed in closed form. Only the projection of
the basis onto a subspace of polynomials is known, and is utilized in
the method's formulation.}
The conforming VEM was first developed for second-order elliptic
problems in primal
formulation~\cite{BeiraodaVeiga-Brezzi-Cangiani-Manzini-Marini-Russo:2013},
and then 
in nonconforming formulation~\cite{AyusodeDios-Lipnikov-Manzini:2016}.
The first works using a $\CS{1}$-regular conforming VEM addressed the
classical plate bending
problems~\cite{Brezzi-Marini:2013,Chinosi-Marini:2016}, second-order
elliptic
problems~\cite{BeiraodaVeiga-Manzini:2014,BeiraodaVeiga-Manzini:2015},
and the nonlinear Cahn-Hilliard
equation~\cite{Antonietti-BeiraodaVeiga-Scacchi-Verani:2016}.
In~\cite{Antonietti-Manzini-Verani:2019}, a highly-regular conforming
VEM is proposed for the two-dimensional polyharmonic problem
$(-\Delta)^{\ps_1}\us=\fs$, $\ps_1\geq 1$.
The VEM is based on an approximation space that locally contains
polynomials of degree $\rs\geq2\ps_1-1$ and has a global $\HS{\ps_1}$
regularity.
In~\cite{Antonietti-Manzini-Scacchi-Verani:2021}, this formulation was
extended to a virtual element space that can have arbitrary regularity
$\ps_2\geq\ps_1\geq1$ and contains polynomials of degree
$\rs\geq\ps_2$.
Highly-regular conforming VEM in any dimension has been proposed in
\cite{Chen-Huang-Wei:2022}.

\BLUE{HOPF models of dynamic fracture are based on two coupled
  governing equations: the momentum conservation equation and the HOPF
  evolution equation.
  In our recent work
  \cite{Antonietti-Manzini-Mazzieri-Mourad-Verani:2021}, we developed
  a VEM for the momentum equation with linear elastic constitutive
  laws.
  In the present work, we propose a VEM for the other ingredient of
  HOPF fracture models, namely the HOPF evolution equation itself.
  The coupling between these two virtual element models is a
  non-trivial task, and will be the topic of our future work.  }
Specifically, we design a conforming VEM for the fourth-order equation
with Laplace and $\LS{2}$ terms.
The numerical approximation relies on an arbitrary order accurate,
virtual element space with \BLUE{$\HS{2}$ }global regularity.
The degrees of freedom of the lowest-order accurate $\CS{1}$-VEM are
the values of the virtual element functions and their gradients at the
mesh vertices.
Attaining high-order accuracy requires additional edge polynomial
moments of the trace of the virtual element functions and their normal
derivatives.
This choice of the degrees of freedom guarantees the global
\BLUE{$\HS{2}$} regularity.
To avoid the computational complexity of the lower order terms, we
introduce an elliptic projection operator that combines the biharmonic
and Laplace operators.
\BLUE{
The calculation of the elliptic projection in every cell reduces to a
single matrix calculation instead of two, and our approach requires
less computation than the ones proposed in the existing VEM
literature.
Furthermore, the use of a single elliptic projection simplifies the
fixing of the kernel that reduces to the kernel of the Laplacian
operator.
This technique reduces the computational cost significantly for
higher order VEM spaces and, potentially, for higher dimensions.}



The remainder of this paper is organized as follows.
In Section~\ref{sec2:phase-field}, we briefly discuss the strong and
weak formulations of the HOPF model along with the mathematical
arguments proving their well-posedness.
In Section~\ref{sec3:VEMapproximation}, we construct the virtual
element approximation of the HOPF equation through the definition of
the virtual element space and the bilinear forms and linear functional
required by the variational formulation.
In Section~\ref{sec4:Theory}, we conduct a convergence analysis of the
proposed VEM by deriving error estimates of the discrete scheme.
In Section~\ref{sec5:implementation}, we discuss the implementation of
the method and provide details on the discretization.
In Section~\ref{sec6:numerical:results}, we carry out a numerical
investigation about the performance of the proposed method by solving
a manufactured solution problem on a set of representative polygonal
meshes.
In addition, we show that optimal convergence rates are attained for
an example involving a diagonal crack modeled by the phase-field
method.
In Section~\ref{sec6:conclusions}, we offer our final conclusions and
remarks on possible future work.

\subsection{Notation and technicalities}
Throughout this paper, we adopt the notation of Sobolev spaces of
Ref.~\cite{Adams-Fournier:2003}.
Accordingly, we denote the space of square integrable functions
defined on any open, bounded, connected domain $\calD\subset\REAL^2$
with boundary $\partial\calD$ by $\LTWO(\calD)$, and the Hilbert space
of functions in $\LTWO(\calD)$ with all partial derivatives up to a
positive integer $m$ also in $\LTWO(\calD)$ by $\HS{m}(\calD)$,
cf.~\cite{Adams-Fournier:2003}.
We endow $\HS{m}(\calD)$ with a norm and a seminorm that we denote as
$\norm{\,\cdot\,}{m,\calD}$ and $\snorm{\,\cdot\,}{m,\calD}$,
respectively.

\medskip
The virtual element method is formulated on the mesh family
$\big\{\Th\big\}_{h}$, where each mesh $\Th$ is a partition of the
computational domain $\Omega$ into nonoverlapping polygonal elements
$\P$.
A polygonal element $\P$ is a compact subset of $\REAL^2$ with
boundary $\partial\P$, area $\mP$, center of gravity $\xvP$, and
diameter $\hP=\sup_{\xv,\yv\in\P}\vert\xv-\yv\vert$.
The mesh elements of $\Th$ form a finite cover of $\Omega$ such that
$\overline{\Omega}=\cup_{\P\in\Th}\P$ and the mesh size labeling each
mesh $\Th$ is defined by $\hh=\max_{\P\in\Th}\hP$.
A mesh edge $\E$ has center $\xvE$ and length $\hE$; a mesh vertex
$\V$ has position vector $\xvV$.

We denote the set of mesh edges by $\EE_{\hh}$ and the set of mesh
vertices by $\calV_{\hh}$.
We decompose the edge set as $\EE_{\hh}:=\EEi\cup\EEbdry$, where $\EEi$
and $\EEbdry$ are the set of interior and boundary edges.
Similarly, we decompose the vertex set as $\VV_h:=\VVi\cup\VVbdry$,
where $\VVi$ and $\VVbdry$ are the set of interior and boundary
vertices.
For each $\P\in\Th$, we denote by $\norP$ the unit normal vector and
by $\tngP$ the unit tangential vector along the boundary $\partial\P$.
We assume a local orientation of $\partial\P$ so that $\norP$ point
out of $\P$.
Besides, we will use $\norPE$ ad $\tngPE$ to denote the unit normal
and tangential vectors to an edge $\E\in\EE_h$ that are locally
oriented consistently with $\partial\P$ and $\bn_e$ and $\bt_{e}$ the
vectors whose orientation is globally fixed once and for all (and
independent of $\partial\P$)
Moreover, in the definition of the degrees of freedom of the next
section, we also associate every vertex $\V$ with a characteristic
lenght $\hV$, which is the average of the diameters of the polygons
sharing that vertex.

\medskip
For any integer number $\ell\geq0$, we let $\PS{\ell}(\P)$ and
$\PS{\ell}(\E)$ denote the space of polynomials defined on the element
$\P$ and the edge $\E$, respectively; $\PS{\ell}(\Th)$ denotes the
space of piecewise polynomials of degree $\ell$ on the mesh $\Th$.
For convenience of exposition, we also use the notation
$\PS{-2}(\P)=\PS{-1}(\P)=\{0\}$.
Accordingly, it holds that $\restrict{\qs}{\P}\in\PS{\ell}(\P)$ if
$\P\in\Th$ for all $\qs\in\PS{\ell}(\Th)$.
Finally, we define the (broken) seminorm of a function
$\vs\in\prod_{\P\in\Th}\HS{2}(\P)$ by
\begin{align*}
  \norm{\vs}{\hh}^2=\sum_{\P\in\Th}\asP(\vs,\vs).
\end{align*}
Throughout the paper, we use the multi-index notation, so that
$\nu=(\nu_1,\nu_2)$ is a two-dimensional index defined by the two
integer numbers $\nu_1,\nu_2\geq0$.
Moreover,
$\Ds^{\nu}\ws=\partial^{|\nu|}\ws\slash{\partial\xs_1^{\nu_1}\partial\xs_2^{\nu_2}}$
denotes the partial derivative of order $|\nu|=\nu_1+\nu_2>0$ of a
sufficiently regular function $\ws(\xs_1,\xs_2)$, and we use the
conventional notation that $\Ds^{(0,0)}\ws=\ws$ for $\nu=(0,0)$.
We also denote the partial derivatives of $\ws$ versus $x$ and $y$ by
the shortcuts $\partial_{x}\ws$ and $\partial_{y}\ws$, and
the normal and tangential derivatives with respect to a given edge by
$\partial_{n}\ws$ and $\partial_{t}\ws$.




\section{The high-order phase-field model}
\label{sec2:phase-field}


Let $\Omega\subset\REAL^2$ be a simply-connected, open, bounded domain
with polygonal boundary $\Gamma$.
The HOPF model is given by the linear, fourth-order, partial
differential equation problem for the real, scalar unknown $\us$:
\begin{subequations}
  \label{eq:HOPF:strong}
  \begin{align}
    \alpha_2\Delta^2\us - \alpha_1\Delta\us + \alpha_0\us &= \fs\phantom{\gs_0\gs_1} \quad\textrm{in~}\Omega,\label{eq:HOPF:strong:A}\\
    \us                                                  &= \gs_0\phantom{\fs\gs_1} \quad\textrm{on~}\Gamma,\label{eq:HOPF:strong:B}\\
    \partial_n\us                                        &= \gs_1\phantom{\fs\gs_0} \quad\textrm{on~}\Gamma,\label{eq:HOPF:strong:C}
  \end{align}
\end{subequations}
where $\alpha_0$, $\alpha_1$ and $\alpha_2$ are strictly-positive,
real constant coefficients; $\fs\in\LTWO(\Omega)$ is the load term;
$\gs_0\in\HS{\frac32}(\Gamma)$ and $\gs_1\in\HS{\frac12}(\Gamma)$ are
the univariate functions definining the Dirichlet boundary conditions
on $\Gamma$. 

Consider the affine space
$\Vs=\big\{\vs\in\HS{2}(\Omega),\restrict{\vs}{\Gamma}=\gs_{0},\restrict{\partial_{n}\vs}{\Gamma}=\gs_{1}\big\}$,
and its linear subspace $\Vszr=\HSzr{2}(\Omega)$, which we
equivalently define by setting $\gs_0=\gs_1=0$ in $\Vs$.
The weak formulation of problem~\eqref{eq:HOPF:strong} reads as:
\begin{align}
  \mbox{\emph{Find $\us\in\Vszr$ such that~}}\quad \calA(\us,\vs) =
  \Fs(\vs) \quad\forall\vs\in\Vszr,
  \label{eq:varform}
\end{align}
where the bilinear form $\calA:\Vszr\times\Vszr\to\REAL$ and the linear
functional $\Fs:\Vszr\to\REAL$ are defined as follows:
\begin{align}
  \calA(\us,\vs) &:=
  \alpha_2\astwo(\us,\vs) +
  \alpha_1\asone(\us,\vs) +
  \alpha_0\aszer(\us,\vs),
  \label{eq:A:def}
\end{align}
with 
\begin{align}
  \astwo(\us,\vs) :=\displaystyle\int_{\Omega}\Delta\us\Delta\vs\,\dxv,\quad
  \asone(\us,\vs) :=\displaystyle\int_{\Omega}\nabla\us\cdot\nabla\vs\,\dxv,\quad
  \aszer(\us,\vs) :=\displaystyle\int_{\Omega}\us\,\vs\dxv,
  \label{eq:A210:def}
\end{align}
and
\begin{align*}
  \Fs(\vs) := \int_{\Omega}\fs\vs\dxv.
\end{align*}
For the exposition's convenience, we will also use the additional
bilinear form $\calB:\Vszr\times\Vszr\to\REAL$ given by
$\calB(\us,\vs):=\alpha_2\astwo(\us,\vs)+\alpha_1\asone(\us,\vs)$,
which is clearly such that $\calA(\us,\vs) = \calB(\us,\vs) +
\alpha_0 \calA_0\big(\us,\vs\big)$.
The positive sublinear functional
$\norm{\,\cdot\,}{\Vszr}:\Vszr\to\REAL^+$ given by
\begin{align}
  \norm{\vs}{\Vszr}^2 = \int_{\Omega}\big( \ABS{\Delta\vs}^2 + \ABS{\nabla\vs}^2 + \ABS{\vs}^2 \big)\dxv
\end{align}
is a norm on $\Vszr$, and it trivially holds that
$\norm{\vs}{\Vszr}\leq2\norm{\vs}{2,\Omega}$.
The bilinear form $\calA(\cdot,\cdot)$ is coercive with respect to the
norm $\norm{\,\cdot\,}{\Vszr}$.
The coercivity of $\calA(\cdot,\cdot)$ with coercivity constant
$\alpha=\min(\alpha_0,\alpha_1,\alpha_2)$ follows on noting that
\begin{align}
  \calA(\vs,\vs)
  \geq \min(\alpha_0,\alpha_1,\alpha_2)
  \int_{\Omega}\big( \ABS{\Delta\vs}^2 + \ABS{\nabla\vs}^2 + \ABS{\vs}^2 \big)\dxv
  = \min(\alpha_0,\alpha_1,\alpha_2)\norm{\vs}{\Vszr}^2,
  \label{eq:As:coercivity}
\end{align}
since $\min(\alpha_0,\alpha_1,\alpha_2)>0$.

The symmetric bilinear form $\calA(\cdot,\cdot)$ is an inner product
on $\HS{2}(\Omega)$ since from the coercivity of $\calA(\cdot,\cdot)$
it follows that $\calA(\vs,\vs)=0$ implies that $\vs=0$.
The Cauchy-Schwarz inequality implies an upper bound for
$\calA(\us,\vs)$ with continuity constant equal to
$\max(\alpha_0,\alpha_1,\alpha_2)$,
\begin{align}
  \calA(\vs,\vs)
  &\leq \max(\alpha_0,\alpha_1,\alpha_2)
  \int_{\Omega}\big( \ABS{\Delta\vs}^2 + \ABS{\nabla\vs}^2 + \ABS{\vs}^2 \big)\dxv
  = \max(\alpha_0,\alpha_1,\alpha_2) \norm{\vs}{\Vszr}^2.
  \label{eq:A:continuity}
\end{align}

The well-posedness follows from an application of the Lax-Milgram
theorem since $\calA(\cdot,\cdot)$ is coercive and continuous and
$\Fs(\cdot)$ is continuous \cite{Brenner-Scott:2008}.





\section{Virtual element approximation of the HOPF problem}
\label{sec3:VEMapproximation}

The virtual element method that approximates the variational
formulation \eqref{eq:varform} reads as
\begin{align}
  \mbox{\emph{Find~$\ush\in\Vsh{\ks}$\,\,such~that~}}
  \calAh(\ush,\vsh) = \bil{\fsh}{\vsh}
  \qquad\forall\vsh\in\Vsh{\ks}.
  \label{eq:VEM}
\end{align}
In this formulation, $\Vsh{\ks}$ is the $\HS{2}$-conforming
approximation of the space $\HS{2}(\Omega)$ provided by the VEM,
$\ush$ and $\vsh$ are the trial and test functions from this space,
and $\fsh\in(\Vsh{\ks})'$ is a liner operator in the dual space
$(\Vsh{\ks})'$ that approximates the load term $\fs$.
We define all these mathematical entities in the rest of this section,
which we devote to the construction of the VEM.

\subsection{Local enlarged virtual element space, functionals and the elliptic projector $\PiLP{\ks}$}

Consider the space of functions on the polygonal boundary $\partial\P$
for all integer $\ks\geq2$ given by
\begin{align}
  \Bsh{\ks}(\partial\P) \eqdef
  \Big\{ \vsh\in\CS{1}(\partial\P)\text{~such~that~}
  \restrict{\vsh}{\E}\in\PS{\rs_0(\ks)}(\E)
  \text{~and~}
  \restrict{\partial_n\vsh}{\E}\in\PS{\rs_1(\ks)}(\E)\,
  \,\forall\E\in\partial\P
  \Big\},
  \label{eq:Bsh:def}
\end{align}
for $\ks\geq2$, where we let $\rs_0(\ks)$ and $\rs_1(\ks)$ be the two
integer-valued functions of the integer $\ks$ such that
\begin{align*}
  \rs_0(\ks) = 
  \begin{cases}
    3   & \mbox{if $\ks=    2$},\\
    \ks & \mbox{if $\ks\geq 3$},
  \end{cases}
  \qquad
  \rs_1(\ks) = \ks-1, \quad if \ks\geq 2.
\end{align*}
The dimension of the local virtual element space
$\Bsh{\rs}(\partial\P)$ is equal to
\begin{align*}
  \dims\Bsh{\ks}(\partial\P)
  = \NPE\big(\rs_0(\ks)+\rs_1(\ks)+2\big) - 3\NPV
  = \NPE\big(\rs_0(\ks)+\rs_1(\ks)-1\big),
\end{align*}
where $\NPE$ and $\NPV$ are the number of edges and vertices of the
polygonal boundary $\partial\P$ (note that $\NPE=\NPV$).
The term $-3\NPV$ takes into consideration the constraint
$\vsh\in\CS{1}(\partial\P)$.

\medskip
Examples for different values of $\ks$ are the following ones:
\begin{itemize}
\item for $k=2$, we find that $\rs_0(2)=3$, $\rs_1(2)=1$, and
\begin{align*}
  \Bsh{2}(\partial\P) \eqdef
  \Big\{ \vsh\in\CS{1}(\partial\P)\text{~such~that~}
  \restrict{\vsh}{\E}\in\PS{3}(\E)
  \text{~and~}
  \restrict{\partial_n\vsh}{\E}\in\PS{1}(\E)\,
  \,\forall\E\in\partial\P
  \Big\};
\end{align*}

\item for $k=3$, we find that $\rs_0(3)=3$, $\rs_1(3)=2$, and
\begin{align*}
  \Bsh{3}(\partial\P) \eqdef
  \Big\{ \vsh\in\CS{1}(\partial\P)\text{~such~that~}
  \restrict{\vsh}{\E}\in\PS{3}(\E)
  \text{~and~}
  \restrict{\partial_n\vsh}{\E}\in\PS{2}(\E)\,
  \,\forall\E\in\partial\P
  \Big\};
\end{align*}

\item for $k=4$, we find that $\rs_0(4)=4$, $\rs_1(4)=3$, and
\begin{align*}
  \Bsh{4}(\partial\P) \eqdef
  \Big\{ \vsh\in\CS{1}(\partial\P)\text{~such~that~}
  \restrict{\vsh}{\E}\in\PS{4}(\E)
  \text{~and~}
  \restrict{\partial_n\vsh}{\E}\in\PS{3}(\E)\,
  \,\forall\E\in\partial\P
  \Big\}.
\end{align*}
\end{itemize}

Then, we consider the differential operator
\begin{align*}
  \calL(\vs) = \big(\alpha_2\Delta^2 - \alpha_1\Delta\big)\vs,
\end{align*}
(where we recall that $\alpha_2,\alpha_1>0$).
For all the integers $\ks\geq2$, we consider
\begin{align*}
  \Vsht{\ks}(\P) \eqdef \Big\{
  \vsh\in\HS{2}(\P):\,
  \calL\vsh\in\PS{\ks}(\P),
  \restrict{\vsh}{\partial\P}\in\Bsh{\ks}(\partial\P)
  \Big\}.
\end{align*}

\begin{figure}[t]
  \begin{center}
    \begin{tabular}{cccc}
      \includegraphics[scale=0.25]{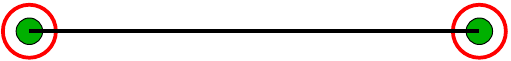} & \hspace{0.5cm}
      \includegraphics[scale=0.25]{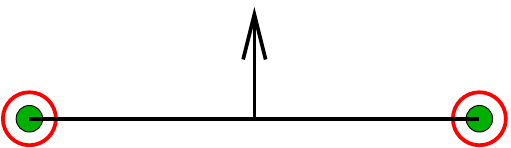} & \hspace{0.5cm}
      \includegraphics[scale=0.25]{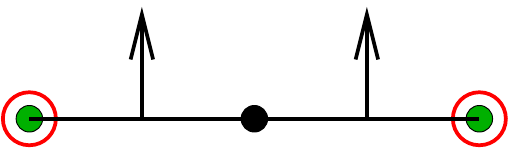} & \hspace{0.5cm}
      \includegraphics[scale=0.25]{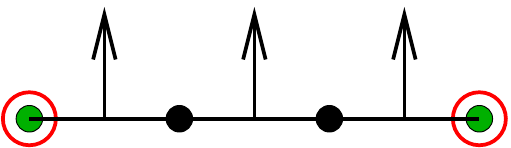} \\
      $\ks=2$ & \hspace{0.5cm} $\ks=3$ & \hspace{0.5cm} $\ks=4$ & \hspace{0.5cm} $\ks=5$ \\[0.75em]
    \end{tabular}
    \caption{ Edge degrees of freedom \DOFS{D1}-\DOFS{D2} of the
      virtual element space $\Vsh{\ks}(\P)$ with polynomial degree
      $\ks$ such that $2\leq\ks\leq5$.
      The (green) dots at the vertices represent the vertex values and
      each (red) vertex circle represents an order of derivation.
      The (black) dots on the edge represent the polynomial moments of
      the trace $\restrict{\vsh}{\E}$; the arrows represent the
      polynomial moments of $\restrict{\partial_n\vsh}{\E}$.
    }
    \label{fig:VEM:edge:dofs}
  \end{center}
\end{figure}

On every mesh element $\P\in\Th$, we consider the set of real valued,
linear and continuous functionals that associates a function
$\vs\in\HS{2}(\P)$ with
\begin{description}
\item\DOFS{D1}: for $\ks\geq2$, $\vs(\xvV)$, $\partial_x\vs(\xvV)$,
  $\partial_y\vs(\xvV)$ for any vertex $\V$ of $\partial\P$;
    
  \medskip
\item\DOFS{D2}: for $\ks\geq4$,
  $\displaystyle\frac{1}{\hE}\int_{\E}\qs\vs\dS$ for any
  $\qs\in\PS{\ks-4}(\E)$, and any edge $\E\in\partial\P$;

  \medskip
\item\DOFS{D3}: for $\ks\geq3$,
  $\displaystyle\int_{\E}\qs\partial_n\vs\dS$ for any
  $\qs\in\PS{\ks-3}(\E)$, and any edge $\E\in\partial\P$;
    
  \medskip
\item\DOFS{D4}: for $\ks\geq2$,
  $\displaystyle\frac{1}{\mP}\int_{\P}\qs\vs\dxv$ for any
  $\qs\in\PS{\ks-2}(\P)$.
\end{description}

Figure~\ref{fig:VEM:edge:dofs} shows the degrees of freedom
\DOFS{D1}-\DOFS{D3} associated with a given mesh edge $\E$.
We note that the traces $\restrict{\vsh}{\E}\in\PS{\rs_{0}(\ks)}(\E)$
and $\restrict{\partial_n\vsh}{\E}\in\PS{\rs_1(\ks)}(\E)$ are
computable using the values of \DOFS{D1}-\DOFS{D3}.
For all $\vs\in\HS{2}(\P)$, the $\LTWO$-orthogonal projection
$\Piz{\ks-2}\vs$ is computable using the values of \DOFS{D4}.

\medskip
Now, consider the integer $\ks\geq2$ and the bilinear form
$\calB(\us,\vs)=\alpha_2\astwo(\us,\vs)+\alpha_1\asone(\us,\vs)$ with
$\alpha_2,\alpha_1>0$.
We define the elliptic projection operator
$\PiLP{\ks}:\HS{2}(\P)\to\PS{\ks}(\P)$ such that for every
$\vs\in\HS{2}(\P)$, the $\ks$-degree polynomial $\PiLP{\ks}\vs$ is the
solution to the variational problem:
\begin{align}
  \calB\big(\PiLP{\ks}\vs-\vs,\qs\big)                             &= 0 \qquad\forall\qs\in\PS{\ks}(\P),\label{eq:proj:PilP:a}
  \\[0.5em]
  \int_{\partial\P}\big(\PiLP{\ks}\vs-\vs\big)\dS                     &= 0.\label{eq:proj:PilP:b}
\end{align}

\begin{remark}
  Equation~\eqref{eq:proj:PilP:a} defines the elliptic projection
  operator $\PiLP{\ks}(\cdot)$ up to constant functions on $\P$ since
  we assume that $\alpha_1>0$.
  This fact will be reflected by the way we fix the kernel of such
  projector in the practical implementation of the method,
  cf. Section~\ref{sec5:implementation}; see
  also~\cite{Antonietti-BeiraodaVeiga-Scacchi-Verani:2016},
  If we relax this condition and assume that only $\alpha_2$ is
  strictly positive and $\alpha_1$ is nonnegative, i.e.,
  $\alpha_1\geq0$, then the solution to~\eqref{eq:proj:PilP:a} is
  defined up to the harmonic polynomials
  $\PS{\ks}^H(\P)=\big\{\qs\in\PS{\ks}(\P):\Delta\qs=0\big\}$,
  \BLUE{which is a subspace of} $\PS{\ks}(\P)$.
  For example, for $\alpha_1=0$, $\ks=2$ and $\qs=xy$,
  Equation~\eqref{eq:proj:PilP:a} is satisfied independently of the
  definition of $\PiLP{\ks}\vsh$ since $\Delta(xy)=0$.
  In such a case, we should
  supplement~\eqref{eq:proj:PilP:a}-\eqref{eq:proj:PilP:b} with a
  condition that removes the indeterminacy due to the harmonic
  polynomials or redefine~\eqref{eq:proj:PilP:a} by taking $\qs$ in
  the quotient space $\PS{\ks}(\P)\setminus{\PS{\ks}^H(\P)}$.
\end{remark}

\RED{
\begin{remark}
 In this article, we have defined the projector $\PiLP{\ks}$ using the
 differential operator $\calL$.
 To define the projector uniquely, we affix one additional equation,
 i.e., \eqref{eq:proj:PilP:b}, to remove the kernel of $\calL$.
 The technique employed here can be extended to approximate the
 operator $\calL+\alpha_0\us$ whenever $\alpha_0>0$.
 In such a case, the projector $\PiLP{\ks}$ is defined uniquely
 by~\eqref{eq:proj:PilP:a}, and we do not need to fix any kernel.
 \BLUE{However, in this work we do not pursue this approach since we want
   to keep the possibility that $\alpha_0=0$.}
\end{remark}
}

The polynomial $\PiLP{\ks}\vsh$ is computable for every
$\vsht\in\Vsht{\ks}(\P)\subset\HS{2}(\P)$ using the values of the
functionals \DOFS{D1}-\DOFS{D4}.
In fact, let $\qs\in\PS{\ks}(\P)$, $\vsh\in\Vsht{\ks}(\P)$.
A repeated integration by parts yields
\begin{align*}
  \calB(\vsh,\qs)
  &= \int_{\P}\vsh\big(\alpha_2\Delta^2\qs - \alpha_1\Delta\qs\big)\dxv
  + \int_{\partial\P} \Big( \vsh\partial_{n}\big(\alpha_1\qs-\alpha_2\Delta\qs\big) + \alpha_2(\partial_{n}\vsh)\Delta\qs \Big)\dS
  \\[0.5em]
  &= \int_{\P}\vsh\calL\qs\dxv
  + \int_{\partial\P} \Big( \vsh\partial_{n}\big(\alpha_1\qs-\alpha_2\Delta\qs\big) + \alpha_2(\partial_{n}\vsh)\Delta\qs \Big)\dS.
\end{align*}
On the one hand, the volume integral on the right is computable using
the values of \DOFS{D4} since
$\calL\qs\in\PS{\ks-2}(\P)$.
On the other hand, the boundary integral on the right is computable
since the values of \DOFS{D1}-\DOFS{D3} for $\vsh$ allow us to compute
the polynomial trace of $\vsh$ and $\partial_{n}\vsh$ in
$\PS{\rs_0(\ks)}(\E)$ and $\PS{\rs_1(\ks)}(\E)$, respectively, on
every edge $\E\in\partial\P$.

\medskip
Now, let $\PS{\ks}(\P)\setminus{\PS{\ks-2}(\P)}$ denote the space of the
polynomials of degree $\ks$ that are orthogonal to the polynomials of
degree $\ks-2$, and consider the linear functionals providing the
values:
\begin{itemize}
\item ($\BAR{\mathbf{D4}}$): for $\ks\geq4$,
  $\displaystyle\frac{1}{\mP}\int_{\P}\qs\vs\dxv$ for any
  $\qs\in\PS{\ks}(\P)\setminus{\PS{\ks-2}(\P)}$.
\end{itemize}

\RED{For notation convenience, we introduce
($\TLD{\mathbf{D4}}$)=\big[\DOFS{D4},($\BAR{\mathbf{D4}}$)\big] to
collect all the functionals associated with the internal moments.
Furthermore, we let
$\calDt$=\big[\DOFS{D1},\DOFS{D2},\DOFS{D3},($\TLD{\mathbf{D4}}$)\big]. bet
the set of the values provided by all the functionals defined so far.
The following lemma states that these functionals are unisolvent in
$\Vsht{\ks}(\P)$, and can thus be taken as the degrees of freedom of
this space.  }
\begin{lemma}[Unisolvency of $\calDt$ in $\Vsht{\ks}(\P)$] \label{lemma:Unisolvency}
  Each function $\vsh\in\Vsht{\ks}(\P)$ is uniquely determined by the
  degrees of freedom \RED{$\calDt$}.
\end{lemma}
\vspace{-0.5\baselineskip}
\begin{proof}
  From a counting argument, we note that the cardinality of
  $\calDt$ equals the dimension of $\Vsht{\ks}(\P)$.
  Then, take $\vsh\in\Vsht{\ks}(\P)$, $\ks\geq2$.
  A repeated integration by parts yields
  \begin{multline*}
    \calB(\vsh,\vsh)
    = \int_{\P}\big( \alpha_2\ABS{\Delta\vsh}^2 + \alpha_1\ABS{\nabla\vsh}^2 \big)\dxv
    = \int_{\P}\vsh\big(\alpha_2\Delta^2\vsh - \alpha_1\Delta\vsh\big)\dxv \\
    + \int_{\partial\P} \Big( \vsh\partial_{n}\big(\alpha_1\vsh-\alpha_2\Delta\vsh\big)
    + \alpha_2(\partial_{n}\vsh)\Delta\vsh \Big)\dS.
  \end{multline*}
  Let all the values provided by $\calDt$ for $\vsh$ be equal to
  zero.
  All the polynomial moments of degree (up to) $\ks$ of $\vsh$ are
  zero since the values ($\TLD{\mathbf{D4}}$) for $\vsh$ are zero.
  The volume integral is zero because it is the moment of $\vsh$
  against $\calL\vsh=\alpha_2\Delta^2\vsh-\alpha_1\Delta\vsh$, which
  is a polynomial of degree $\ks$ from the space definition.
  The boundary integrals are zero because assuming that
  \DOFS{D1}-\DOFS{D3} are equal to zero implies that all the edge
  traces of $\vsh$ and $\partial_n\vsh$ are zero.
  Therefore, $\nabla\vsh=0$, so that $\vsh$ is constant on $\P$.
  This constant must be zero since a constant function equals all its
  degrees of freedom (not implying a derivative) that we suppose to be
  zero.
\end{proof}


\subsection{Enhanced virtual element space}
The enhanced virtual element space is
\begin{align} 
  \Vsh{\ks}(\P) := \bigg\{
  \vsh\in\Vsht{\ks}(\P):\,\int_\P\big(\vsh-\PiLP{\ks}\vsh\big)\qs\dxv = 0
  \quad \forall\qs\in\PS{\ks}(\P)\setminus\PS{\ks-2}(\P)\bigg\}.
  \label{eq:Vsh-enh:def}
\end{align}
A few noteworthy properties follow from this definition.
The polynomial space $\PS{\ks}(\P)$ is a subspace of $\Vsh{\ks}(\P)$.
The $\HS{2}$-orthogonal projection operator
$\PiLP{\ks}:\Vsht{\ks}(\P)\to\PS{\ks}(\P)$ is computable using only
the degrees of freedom \TERM{D1}-\TERM{D4} and it is, thus,
independent of the additional degrees of freedom
($\overline{\mathbf{D4}}$).
The $\LS{2}$-orthogonal projection operator
$\PizP{\ks}:\Vsh{\ks}(\P)\to\PS{\ks}(\P)$ is computable using only the
degrees of freedom \TERM{D1}-\TERM{D4} and is also independent of the
additional degrees of freedom ($\overline{\mathbf{D4}}$).
Finally, as formally stated in the following lemma, the functionals
providing the values of \DOFS{D1}-\DOFS{D4} are linearly independent
on $\Vsh{\ks}(\P)$ and their number equals the dimension of
$\Vsh{\ks}(\P)$.
As a consequence, they are unisolvent in $\Vsh{\ks}(\P)$ and we can
choose them as the degrees of freedom for this space.

\begin{lemma}\label{lemma:D1-D4:unisolvency}
  The linear functionals providing the values of \DOFS{D1}-\DOFS{D4}
  are linearly independent in $\Vsh{\ks}(\P)$.
\end{lemma}
\vspace{-0.5\baselineskip}
\begin{proof}
  Let $\vsh\in\Vsh{\ks}(\P)$ such that all the values provided by the
  functionals \DOFS{D1}-\DOFS{D4} are zero.
  We only need to prove that $\vsh=0$.
  To this end, we first note that $\PiLP{\ks}\vsh=0$ as this
  polynomial projection only depends on the functionals returning the
  values of \DOFS{D1}-\DOFS{D4} (see the computability of $\PiLP{k}$).
  Since $\PiLP{\ks}\vsh=0$, the definition of $\Vsh{\ks}(\P)$
  implies that
  \begin{align}
    \int_{\P}\qs\vsh\dxv =
    \int_{\P}\qs\PiLP{\ks}\vsh\dxv
    = 0
    \qquad
    \forall\qs\in\PS{\ks}(\P)\setminus\PS{\ks-2}(\P).
  \end{align}
  Hence, the values of the functionals ($\TLD{\mathbf{D4}}$) of $\vsh$
  must also be equal to zero.
  Since all functionals \DOFS{D1}, \DOFS{D2}, \DOFS{D3},
  ($\TLD{\mathbf{D4}}$) are zero for
  $\vsh\in\Vsh{\ks}(\P)\subset\Vsht{\ks}(\P)$, and these functionals
  are unisolvent in $\Vsht{\ks}(\P)$, it follows that $\vsh=0$.
\end{proof}

Let $\calD$=\big[ \DOFS{D1}, \DOFS{D2}, \DOFS{D3}, \DOFS{D4} \big]
denote the set of the linear functionals associated with the degrees
of freedom for $\Vsh{\ks}(\P)$.
In view of Lemma~\ref{lemma:D1-D4:unisolvency}, we conclude that the
triplet $\big(\P,\Vsh{\ks}(\P),\calD\big)$ is a finite element in the
sense of Ciarlet, cf.~\cite[Chapter~3]{Ciarlet:2002}.

The global virtual element space for $\ks\geq2$ is given by
\begin{align}
  \Vsh{\ks} := \Big\{
  \vsh\in\HS{2}(\Omega)\,:\,\restrict{\vsh}{\P}\in\Vsh{\ks}(\P)
  \quad\forall\P\in\Th \Big\}.
\end{align}
The degrees of freedom of the functions in $\Vsh{\ks}$ are obtained by
an $\HS{2}$-conforming coupling of the elemental degrees of freedom
and are thus provided by the values of the functionals:
\begin{description}
\item\DOFS{D1}: for $\ks\geq2$, $\vsh(\xvV)$, $\partial_x\vsh(\xvV)$,
  $\partial_y\vsh(\xvV)$ for any vertex $\V$ of $\calV$;
    
  \medskip
\item\DOFS{D2}: for $\ks\geq4$,
  $\displaystyle\frac{1}{\hE}\int_{\E}\qs\vsh\dS$ for any
  $\qs\in\PS{\ks-4}(\E)$, and any edge $\E\in\calE$;

  \medskip
\item\DOFS{D3}: for $\ks\geq3$,
  $\displaystyle\int_{\E}\qs\partial_n\vsh\dS$ for any
  $\qs\in\PS{\ks-3}(\E)$, and any edge $\E\in\calE$;
    
  \medskip
\item\DOFS{D4}: for $\ks\geq2$,
  $\displaystyle\frac{1}{\mP}\int_{\P}\qs\vsh\dxv$ for any
  $\qs\in\PS{\ks-2}(\P)$ and any $\P\in\Th$.
\end{description}
The sign of the normal derivative $\partial_n\vsh$ along the edge $\E$
is determined by the global edge orientation and may differ by a
factor $-1$ from its elementwise value.
The unisolvence of \DOFS{D1}-\DOFS{D4} in $\Vsh{\ks}$ is an immediate
consequence of their unisolvence at the elemental level.
The unisolvence property implies the existence of a global Lagrangian
basis $\varphi_1,\varphi_2,\ldots,\varphi_{\Ndofs}$ (in a global numbering
system) where $\Ndofs$ is the total number of degrees of freedom, and
such that the $i$-th basis function $\varphi_i$ has all degrees of
freedom equal to zero except the $i$-th one whose value is $1$.
The existence of such a set of basis functions, although virtual, is
crucial in the implementation of the method.

\subsection{Virtual element approximation of $\calA(\cdot,\cdot)$
  and $\calF(\cdot)$}
\label{subsec3:VEM:bilinear:form:RHS}

Let $\P\in\Th$ be a mesh element, and consider the bilinear forms
$\astwoP,\asoneP,\aszerP:\Vsh{\ks}(\P)\times\Vsh{\ks}(\P)\to\REAL$
given by integrating on $\P$ instead of $\Omega$ in the corresponding
bilinear forms in~\eqref{eq:A210:def}.
Let
$\calAP(\cdot,\cdot)=\alpha_2\astwoP(\cdot,\cdot)+\alpha_1\asoneP(\cdot,\cdot)+\alpha_0\aszerP(\cdot,\cdot)$.
We use the elliptic projection $\PiLP{\ks}$ and the
$\LS{2}$-orthogonal projection $\PizP{\ks}$ to define the virtual
element bilinear form
$\calAhP:\Vsh{\ks}(\P)\times\Vsh{\ks}(\P)\to\REAL$:
\begin{multline*}
  \calAhP(\ush,\vsh)
  := \alpha_2\astwoP(\PiLP{\ks}\ush,\PiLP{\ks}\vsh)
  + \alpha_1\asoneP(\PiLP{\ks}\ush,\PiLP{\ks}\vsh)
  + \alpha_0\aszerP(\PizP{\ks}\ush,\PizP{\ks}\vsh)\\
  + \SPh\Big( \ush-\PiLP{\ks}\ush, \vsh-\PiLP{\ks}\vsh \Big).
\end{multline*}
Herein, the stabilization term is also built by using the projection
$\PiLP{\ks}(\vsh)$, and the usual formula, so that the bilinear form
$\SPh:\Vsh{\ks}(\P)\times\Vsh{\ks}(\P)\to\REAL$ can be any symmetric, positive
definite, bilinear form such that
\RED{
\begin{align}
  \sigma_*\calBP(\vsh,\vsh)\leq\SPh(\vsh,\vsh)\leq\sigma^*\calBP(\vsh,\vsh)
  \qquad\forall\vsh\in\Vsh{\ks}(\P)\textrm{~with~}\PiLP{\ks}\vsh=0,
  \label{eq:poly:S:stability:property}
\end{align}
}
where $\sigma_*$ and $\sigma^*$ are two positive constants independent
of $\hh$ (and the chosen $\P$).

\medskip
The bilinear form $\calAhP(\cdot,\cdot)$ has the two major properties:
\begin{description}
\item[$(i)$] $\ks$-\textbf{Consistency}: for every polynomial
  $\qs\in\PS{\ks}(\P)$ and virtual element function
  $\vsh\in\Vsh{\ks}(\P)$ it holds:
  \begin{align}
    \calAhP(\vsh,\qs) = \calAP(\vsh,\qs);
    \label{eq:poly:r-consistency}
  \end{align}
\item[$(ii)$] \textbf{Stability}: there exist two positive constants
  $\beta_*$, $\beta^*$ independent of $\hh$ (and $\P$) such that for
  every $\vsh\in\Vsh{\ks}(\P)$ it holds:
  \begin{align}
    \beta_*\calAP(\vsh,\vsh)\leq\calAhP(\vsh,\vsh)\leq\beta^*\calAP(\vsh,\vsh).
    \label{eq:poly:stability}
  \end{align}
  It is immediate to check that for example~\eqref{eq:poly:stability}
  holds by taking
  \begin{align*}
    \beta_*=\min(\alpha_0,\alpha_1,\alpha_2,\sigma_*)
    \quad\textrm{and}\quad
    \beta^*=\max(\alpha_0,\alpha_1,\alpha_2,\sigma^*).
  \end{align*}  
\end{description}
A straightforward consequence of the stability condition $(ii)$ stated
above, is that the discrete bilinear form $\calAhP(\cdot,\cdot)$ is
continuous and coercive.
These properties extend to the global virtual element bilinear form
$\calAh:\Vsh{\ks}\times\Vsh{\ks}\to\REAL$ that we define by adding all
the local terms,
\begin{align*}
  \calAh(\ush,\vsh) := \sum_{\P\in\Th} \calAhP(\ush,\vsh),
  \qquad\forall\ush,\vsh\in\Vsh{\ks}.
\end{align*}
We will find it useful to consider the local and global, discrete
bilinear forms $\calBhP:\Vsh{\ks}(\P)\times\Vsh{\ks}(\P)\to\REAL$, for
all $\P\in\Th$, and $\calBh:\Vsh{\ks}\times\Vsh{\ks}\to\REAL$ that are such that
\begin{align}
  \calAhP(\ush,\vsh) &= \calBhP(\ush,\vsh) + \alpha_0\aszerP(\PizP{\ks}\ush,\PizP{\ks}\vsh),\label{eq:BhP:def}\\
  \calAh (\ush,\vsh) &= \calBh (\ush,\vsh) + \alpha_0\aszer(\Piz{\ks}\ush, \Piz{\ks}\vsh),\label{eq:Bh:def}
\end{align}
where $\PizP{\ks}\vsh$ is the $\LTWO$-orthogonal projection of $\vsh$
onto the local polynomial subspace $\PS{\ks}(\P)$ of $\Vsh{\ks}(\P)$;
$\Piz{\ks}\vsh$ is the global $\LTWO$-orthogonal projection onto the
space of $k$-degree piecewise polynomials $\PS{\ks}(\Th)$ such that
$\restrict{(\Piz{k}\vsh)}{\P}=\PizP{k}\big(\restrict{\vsh}{\P}\big)$ for all
$\P\in\Th$.
Since $\calBhP(\cdot,\cdot)$ and $\calBh(\cdot,\cdot)$ have the same
stabilization term of $\calAhP(\cdot,\cdot)$ and
$\calAh(\cdot,\cdot)$, it is immediate to see that they satisfy the
same consistency and stability properties (with slightly different
constants) stated above as $(i)$ and $(ii)$, and clearly
\begin{align*}
  \calBh(\ush,\vsh) := \sum_{\P\in\Th} \calBhP(\ush,\vsh),
  \qquad\forall\ush,\vsh\in\Vsh{\ks}.
\end{align*}
The bilinear form $\calBh(\cdot,\cdot)$ is also globally continuous
and coercive.

\medskip
Assuming that $\fs\in\HS{k-1}(\Omega)$, we approximate the right-hand
side of~\eqref{eq:varform} through the piecewise $\LTWO$-orthogonal
projection $\restrict{\fsh}{\P}=\PizP{\ks-2}\fs$ for all $\P\in\Th$,
so that
\begin{align}
  \bil{\fsh}{\vsh} := \sum_{\P\in\Th} \bilP{\fsh}{\vsh},\qquad
  \bilP{\fsh}{\vsh} = \int_{\P}\fs\PizP{\ks-2}\vsh\dxv.
  \label{eq:VEM:RHS:def}
\end{align}

\subsection{Well-posedness}
\label{subsec3:well-posedness}
We conclude this section with the well-posedness result.
The well-posedness of the discrete VEM~\eqref{eq:VEM} directly follows
from the global coercivity and continuity of $\calAh(\cdot,\cdot)$ and
the boundedness of $\bil{\fsh}{\cdot}$, as stated in the next
proposition.
\begin{proposition}
  The VEM~\eqref{eq:VEM} with the previous definitions of
  $\calAh(\cdot,\cdot)$ and $\bil{\fsh}{\cdot}$ is well-posed.
\end{proposition}
\vspace{-0.5\baselineskip}
\begin{proof}
  This result is a direct consequence of the Lax-Milgram Theorem,
  cf. \cite[Theorem~2.7.7]{Brenner-Scott:2008}.
\end{proof}




\section{Convergence analysis}
\label{sec4:Theory}

In this section, we prove the convergence of the VEM by quantifying
the error associated with the discrete scheme~\eqref{eq:VEM}.
Upon exploiting the coercivity and continuity of the discrete bilinear
form $\calA_{\hh}(\cdot,\cdot)$, we bound the term
$\snorm{\us-\ush}{2,\Omega}$.
Then, by using the duality argument, and the previous estimate of
$\snorm{\us-\ush}{2,\Omega}$, we derive the error estimates in $\LTWO$
norm and $\HONE$ seminorm assuming that the domain $\Omega$ is convex.
To prove such estimates, we need the regularity results associated
with the biharmonic problem
\begin{align*}
  \alpha_2\Delta^2\xi - \alpha_1\Delta\xi + \alpha_0\xi&=\gs,\phantom{0}\quad \text{in~}\Omega,\\
  \xi=\partial_n\xi                                   &=0,\phantom{\gs}\quad \text{on~}\Gamma.
\end{align*}
Since $\Omega$ is a convex domain, and according to the regularity
results in~\cite{Grisvard:1992}, we know that
\begin{itemize}
\item the regularity result holds:
  \begin{equation}
    \label{eq:regularity:1}
    \gs\in\HS{-1}(\Omega)\implies \xi\in\HS{3}(\Omega)
    \quad\text{and}\quad \norm{\xi}{3,\Omega}\leq\Cs\norm{\gs}{-1,\Omega};
  \end{equation}
  
\item there exists a real number $s$, with $0<s\leq 1$, such that
   \begin{equation}
   \label{eq:regularity:2}
   \gs\in\LTWO(\Omega)\implies\xi\in\HS{3+s}(\Omega)
   \quad\text{and}\quad \norm{\xi}{3+s,\Omega}\leq\Cs\norm{\gs}{0,\Omega}.
   \end{equation}
\end{itemize}
The values of $s$ depends on the maximum angle in $\Omega$.
If the maximum angle \BLUE{is strictly less than} $\pi$, then $s=1$
and $\xi\in\HS{4}(\Omega)$.


\subsection{\BLUE{Mesh assumptions}}

For the convergence analysis, we assume that the mesh family
$\{\Th\}_{0<h\leq 1}$ satisfies the following regularity condition.
\begin{assumption}[Mesh regularity]
  \label{Mesh:regularity}
  There exists a positive real number $\rho$ independent of $\hh$ such
  that for every $\P\in\Th$, it holds that
  \begin{itemize}
    \medskip
  \item[$\mathbf{(A1)}$]\textbf{star-shapedness}: $\P$ is star-shaped
    with respect to an internal ball with radius bigger than
    $\rho\hP$;
    \medskip
  \item[$\mathbf{(A2)}$]\textbf{uniform scaling}: the edge length
    $\hE$ for all $\E\in\partial\P$ is bounded from below by
    $\rho\hP$, i.e., $\hE\geq\rho\hP$.
  \end{itemize}
\end{assumption}
A consequence of these properties is that the element $\P$ admits a
uniformly shape-regular subtriangulation, i.e., the minimum angle of
all the subtriangles partitioning $\P$ is bounded from below by some
positive constant independent of $\hh$ (and $\P$).

\subsection{A priori error estimates}

The convergence of the virtual element method, which we state in
Theorem~\ref{ThmH2} below, follows from the (standard) error bounds
for the polynomial projection and interpolation operators stated in
two technical lemmas, e.g., Lemma~\ref{lemma:polyApprox} and
Lemma~\ref{lemma:VEM:interpolant}, which we report below for
completeness without proof (for a proof, see
\cite{Brezzi-Marini:2013,Adak-Mora-Natarajan-Silgado:2021}) and the
abstract convergence result of Lemma~\ref{lemma:abstract:convergence}.
To ease the notation, in both lemma statements we adopt the convention
$\snorm{\,\cdot\,}{0,\P}=\norm{\,\cdot\,}{0,\P}$ to denote the
$\LTWO$-norm over $\P$.

\begin{lemma}[Polynomial Approximation]
  \label{lemma:polyApprox}
  Under mesh regularity assumptions $\mathbf{(A1)}$-$\mathbf{(A2)}$,
  see Assumption~\ref{Mesh:regularity}, there exists a positive
  constant $\Cs$ independent of $\hh$ such that for all
  $\vs\in\HS{\delta}(\P)$, $\delta$ being a real number such that
  $0\leq\delta\leq\ks+1$, there exists a polynomial approximation
  $\vs_\pi\in\PS{\ks}(\P)$, such that
  \begin{align*}
    \snorm{\vs-\vs_{\pi}}{\ell,\P}\leq\Cs\hh_{\P}^{\delta-\ell}\snorm{\vs}{\delta,\P},
    \quad
    \ell=0,\ldots,[\delta],
  \end{align*}
  where $[\delta]$ denotes the largest integer equal to or smaller
  than $\delta$.
  The constant $\Cs$ may depend on the mesh regularity parameter
  $\rho$.
\end{lemma}

\begin{lemma}[Virtual element interpolation]
  \label{lemma:VEM:interpolant}
  Under mesh regularity assumptions $\mathbf{(A1)}$-$\mathbf{(A2)}$,
  see Assumption~\ref{Mesh:regularity}, there exists a positive
  constant $\Cs$ independent of $\hh$ such that for all
  $\vs\in\HS{\delta}(\P)$, $\delta$ being a real number such that
  $2\leq\delta\leq\ks+1$, there exists a virtual element approximation
  $\vsI\in\Vsh{\ks}(\P)$ such that
  \begin{align*}
    \snorm{\vs-\vsI}{\ell,\P} \leq\Cs\hP^{\delta-\ell} \snorm{\vs}{\delta,\P},
    \quad \ell=0,1,2.
  \end{align*}
  The constant $\Cs$ may depend on the mesh regularity parameter
  $\rho$.
\end{lemma}

\begin{lemma}[Abstract result]
  \label{lemma:abstract:convergence}
  Let $\us$ be the solution of the variational
  problem~\eqref{eq:varform} and $\ush$ the solution of the virtual
  element method~\eqref{eq:VEM} under mesh regularity
  assumptions $\mathbf{(A1)}$-$\mathbf{(A2)}$.
  Then, there exists a positive constant $C>0$, such that the
  (piecewise discontinuous) polynomial approximation
  $\us_{\pi}\in\PS{\ks}(\Th)$ from Lemma~\ref{lemma:polyApprox}, and
  the interpolation approximation $\usI\in\Vsh{\ks}$ of $\us$ from
  Lemma~\ref{lemma:VEM:interpolant} satisfy the inequality
  \begin{align}
    \snorm{\us-\ush}{2,\Omega}
    \leq \Cs\Big(
    \snorm{\us-\usI}{2,\Omega} +
    \snorm{\us-\us_{\pi}}{2,\Omega} +
    \norm{\fs-\fsh}{(\Vsh{\rs})'}
    \Big),
    \label{eq:abstract:convg}
  \end{align}
  where $\norm{\cdot}{(\Vsh{\ks})'}$ denotes the norm of
  $(\Vsh{\ks})'$, the dual space of $\Vsh{\ks}$.
\end{lemma}
\begin{proof}
  Let $u_I$ be the interpolant of $\us$ that
  satisfies Lemma~\ref{lemma:VEM:interpolant}, and set
  $\delta_h:=\ush-\usI\in\Vsh{\ks}$.
  Adding and subtracting $\usI$ and using the Cauchy-Schwarz
  inequality yield
  \begin{align*}
    \snorm{\us-\ush}{2,\Omega} \leq \snorm{\us-\usI}{2,\Omega} +
    \snorm{\delta_h}{2,\Omega}.
  \end{align*}
  Lemma~\ref{lemma:VEM:interpolant} provides the upper bound for the
  first term on the right-hand side, i.e.,
  $\snorm{\us-\usI}{2,\Omega}$.
  Furthermore, by using the coercivity of $\calAh(\cdot,\cdot)$ with
  coercivity constant $\Cs_{\alpha}=\min(\alpha_0,\alpha_1,\alpha_2)$,
  the definition of $\delta_h$, equations~\eqref{eq:varform}
  and~\eqref{eq:VEM}, and the polynomial consistency of
  $\calAh(\cdot,\cdot)$, we find that
  \begin{align*}
    \Cs_{\alpha}\snorm{\delta_h}{2,\Omega}^2
    & \leq\calAh(\delta_h,\delta_h)
    =     \calAh(\ush,\delta_h) - \calAh(\usI,\delta_h) \\
    &= \fsh(\delta_h)
    - \sum_{\P\in\Th} \calAhP(\usI-\us_{\pi}, \delta_h)
    - \sum_{\P\in\Th} \calAhP(\us_{\pi},\delta_h) \\
    &= \fsh(\delta_h) - \fs(\delta_h)
    - \sum_{\P\in\Th} \calAhP(\usI - \us_{\pi},\delta_h)
    + \sum_{\P\in\Th} \calAP(\us-\us_{\pi},\delta_h).
  \end{align*}
  We conclude the proof of inequality~\eqref{eq:abstract:convg} by
  employing the continuity of $\calAhP(\cdot,\cdot)$ and
  $\calAP(\cdot,\cdot)$.
\end{proof}

The following theorem states the convergence of the virtual element
method and an error estimate for the approximation error measured
using the seminorm $\snorm{\,\cdot\,}{2,\Omega}$ (which is a norm on
$\HSzr{2}(\Omega)$).
\begin{theorem}[Convergence in $\HS{2}$-seminorm]
  \label{ThmH2}
  Let $\us$ be the solution of the variational
  problem~\eqref{eq:varform} with $\fs\in\HS{k-1}(\Omega)$ and $\ush$
  the solution of the virtual element method~\eqref{eq:VEM} under mesh
  regularity assumptions $\mathbf{(A1)}$-$\mathbf{(A2)}$.
  Then, there exists a positive constant $\Cs>0$ such that
  \begin{align*}
    \snorm{\us-\ush}{2,\Omega}
    \leq \Cs\hs^{k-1} \Big( \snorm{\us}{k+1,\Omega} + \norm{\fs}{k-1,\Omega} \Big).
  \end{align*}
  The constant $\Cs$ is independent of $\hh$, but depends on the
  $\alpha$-coefficients associated with the model problem, the mesh
  regularity constant $\rho$, and the stability constants of the
  bilinear form $\calAh(\cdot,\cdot)$.
\end{theorem}
\begin{proof}
  The assertion of the theorem is a straightforward consequence of the
  abstract convergence result stated in
  Lemma~\ref{lemma:abstract:convergence}, the error bounds from
  Lemmas~\ref{lemma:polyApprox} and~\ref{lemma:VEM:interpolant}, and a
  standard estimate of the source error term
  $\norm{\fs-\fsh}{(\Vsh{\rs})'}$ in the right-hand side of
  inequality~\eqref{eq:abstract:convg} along with the regularity
  of $\fs$
  (cf.~\cite{BeiraodaVeiga-Brezzi-Cangiani-Manzini-Marini-Russo:2013}).
\end{proof}
\begin{theorem}[Convergence in $\HS{1}$-seminorm]
  \label{ThmH1}
  Let $\us$ be the solution of the variational
  problem~\eqref{eq:varform} with $\fs\in\HS{k-1}(\Omega)$ and $\ush$
  the solution of the virtual element method~\eqref{eq:VEM} under mesh
  regularity assumptions $\mathbf{(A1)}$-$\mathbf{(A2)}$.
  Then, there exists a positive constant $\Cs>0$, such that
  \begin{align}
    \snorm{ \us-\ush }{1,\Omega}
    \leq \Cs\hh^k \Big( \snorm{\us}{k+1,\Omega} + \norm{\fs}{k-1,\Omega} \Big).
    \label{eq:ThmH1}
  \end{align}
  The constant $\Cs$ is independent of $\hh$, but depends on the
  $\alpha$-coefficients associated with the model problem, the mesh
  regularity constant $\rho$, and the stability constants of the
  bilinear form $\calAh(\cdot,\cdot)$.
\end{theorem}
\begin{proof}
  To prove the assertion of the theorem, we use the duality argument.
  Let $\xi$ be the solution of the variational formulation of the auxiliary
  problem
  \begin{subequations}
    \begin{align}
      \alpha_2\Delta^2\xi - \alpha_1\Delta\xi + \alpha_0\xi &= -\Delta(\us-\ush)\phantom{0} \quad\textrm{in~}\Omega,\label{eq:H1:dualProblem:A}\\
      \xi = \partial_n\xi                                  &= 0\phantom{-\Delta(\us-\ush)} \quad\textrm{on~}\Gamma.\label{eq:H1:dualProblem:B}
    \end{align}
  \end{subequations}
  From suitable hypothesis on $\Omega$, and a standard regularity
  argument, see~\cite{Grisvard:1985}
  and~\eqref{eq:regularity:1}-\eqref{eq:regularity:2}, we know that
  \begin{align}
    \norm{\xi}{3,\Omega}
    \leq \Cs \norm{\Delta(\us-\ush)}{-1}
    \leq \Cs \norm{\us-\ush}{1,\Omega}
    \leq \Cs \snorm{\us-\ush}{1,\Omega},
    \label{eq:H1:dualProblem:05}
  \end{align}
  where the last step follows from the Poincare inequality, i.e.,
  $\norm{\us-\ush}{0,\Omega}\leq\Cs\snorm{\us-\ush}{1,\Omega}$, which
  holds because $\us-\ush\in\HSzr{2}(\Omega)$.
  For $\xi$ at least in $\HS{3}(\Omega)$, the following interpolation
  estimate holds:
  \begin{align}
    \norm{\xi-\xiI}{0,\Omega} + \hP\snorm{\xi-\xiI}{1,\Omega} + \hP^2\snorm{\xi-\xiI}{2,\Omega} \leq
    \Cs\hP^3\snorm{\xi}{3,\Omega}.
    \label{eq:xi:interpolation:error}
  \end{align}
  
  Note that
  $\snorm{\us-\ush}{1,\Omega}^2=-(\Delta(\us-\ush),\us-\ush)_{0,\Omega}$
  since the traces of $(\us-\ush)$ and $\partial_n(\us-\ush)$ are zero
  on the boundary of $\Omega$.
  We multiply both sides of~\eqref{eq:H1:dualProblem:A} by $\us-\ush$,
  integrate over $\Omega$ and integrate by parts the $\Delta^2$ and
  $\Delta$ terms.
  Then, we add and subtract the virtual element interpolant
  $\xiI\in\Vsh{\rs}$, cf. Lemma~\ref{lemma:VEM:interpolant}, and
  we obtain
  \begin{align}
    \snorm{\us-\ush}{1,\Omega}^2
    &= -\big(\Delta(\us-\ush),\us-\ush\big)_{0,\Omega}                              \nonumber\\
    &= \big(\alpha_2\Delta^2\xi - \alpha_1\Delta\xi + \alpha_0\xi, \us-\ush\big)   \nonumber\\
    &= \alpha_2\astwo(\us-\ush,\xi) + \alpha_1\asone(\us-\ush,\xi) + \alpha_0\aszer(\us-\ush,\xi), \nonumber\\
    &= \calA(\us-\ush,\xi-\xiI) + \calA(\us-\ush,\xiI) \nonumber\\
    &= \calA(\us-\ush,\xi-\xiI)
    + \Big( (\fs,\xiI) - (\fsh,\xiI) \Big)
    + \Big( \calAh(\ush,\xiI) - \calA(\ush,\xiI) \Big)\nonumber\\
    &= \TERM{T1} + \TERM{T2} + \TERM{T3}.
    \label{eq:main:H1}
  \end{align}
  To complete the proof of inequality~\eqref{eq:ThmH1}, we derive an
  upper bound for the three terms $\TERM{T1}$, $\TERM{T2}$,
  $\TERM{T3}$ separately.

  \medskip
  \noindent
  We bound the first term as follows below by using the continuity of
  the bilinear forms $\aszer(\cdot,\cdot)$, $\asone(\cdot,\cdot)$, and
  $\astwo(\cdot,\cdot)$, and~\eqref{eq:xi:interpolation:error}
  and~\eqref{eq:H1:dualProblem:05}:
  \begin{align*}
    &\TERM{T1}
    = \calA(\us-\ush,\xi-\xiI)
    =
    \alpha_2\astwo(\us-\ush,\xi-\xiI) +
    \alpha_1\asone(\us-\ush,\xi-\xiI) +
    \alpha_0\aszer(\us-\ush,\xi-\xiI) \\
    &\quad\leq
    \Cs(\alpha_2)\snorm{\us-\ush}{2,\Omega}\,\snorm{\xi-\xiI}{2,\Omega} +
    \Cs(\alpha_1)\snorm{\us-\ush}{1,\Omega}\,\snorm{\xi-\xiI}{1,\Omega} +
    \Cs(\alpha_0)\norm {\us-\ush}{0,\Omega}\,\norm {\xi-\xiI}{0,\Omega} \\
    &\quad\leq
    \Cs(\alpha_2)\hh^k\snorm{\us}{k+1}\,\snorm{\us-\ush}{1,\Omega} +
    \Cs(\alpha_1)\hh^2\snorm{\us-\ush}{1,\Omega}^2 +
    \Cs(\alpha_0)\hh^3\snorm{\us-\ush}{1,\Omega}^2.
  \end{align*}

  \medskip
  \noindent
  Since we assume that $\fs\in\HS{k-1}(\Omega)$, we bound the second
  term by noting that for all $k\geq2$, the quantity
  $(\fs-\PizP{k-2}\fs)$ is $\LS{2}$-orthogonal to the constant
  functions, e.g., $\PizP{0}\xiI$.
  Therefore, using the Cauchy-Schwarz inequality, the approximation
  properties of the interpolation operator, the $\LS{2}$-orthogonal
  projection operator, and~\eqref{eq:xi:interpolation:error}
  and~\eqref{eq:H1:dualProblem:05} yield:
  \begin{align}
    \TERM{T2}
    & \leq \sum_{\P\in\Th} \int_{\P}\big(\fs-\PizP{k-2}\fs\big)\,\big(\xiI-\PizP{0}\xiI\big)\dxv 
      \leq \norm{\fs-\Piz{k-2}\fs}{0,\Omega}\,\norm{\xiI-\Piz{0}\xiI}{0,\Omega} \nonumber\\
    & \leq \Cs\hh^{k-1}\snorm{\fs}{k-1,\Omega}\,\Big( \norm{\xiI -\xi}{0,\Omega} + \norm{ \xi -\Piz{0}\xi }{0,\Omega} + \norm{ \Piz{0}(\xi-\xiI) }{0,\Omega}\Big) \nonumber\\
    & \leq \Cs\hh^{k}\snorm{\fs}{k-1,\Omega}\,\snorm{\us-\ush}{1,\Omega}.
    \label{eq:RHS:error:bound}
  \end{align}
  
  \medskip
  \noindent
  Finally, we bound the term $\TERM{T3}$ by using the (local)
  polynomial consistency, which allows us to rewrite the term
  $\calAh(\cdot,\cdot)$ as follows
  \begin{equation}
    \TERM{T3} = \sum_{\P\in\Th}\Big( \calAhP(\ush-\us_{\pi},\xiI-\xi_{\pi}) + \calAP(\us_{\pi}-\ush,\xiI-\xi_{\pi}) \Big).
  \end{equation}
  An application of the continuity property of $\calAh(\cdot,\cdot)$
  and $\calA(\cdot,\cdot)$, the approximation properties of
  Lemmas~\ref{lemma:polyApprox}, and~\ref{lemma:VEM:interpolant},
  and~\eqref{eq:xi:interpolation:error} yield
 \begin{align*}
   \ABS{\TERM{T3}}
   & \leq \Cs(\beta^{\ast})
   \big[\calA(\ush-\us_{\pi},\ush-\us_{\pi})\big]^{\half}\,
   \big[\calA(\xiI-\xi_{\pi},\xiI-\xi_{\pi})\big]^{\half} \\[0.25em]
   & \leq \Cs(\beta^{\ast})\,\Big[
     \alpha_2\snorm{\ush-\us_{\pi}}{2,\Omega}^2 +
     \alpha_1\snorm{\ush-\us_{\pi}}{1,\Omega}^2 +
     \alpha_0\snorm{\ush-\us_{\pi}}{0,\Omega}^2 \Big]^{\half}\\[0.25em]
   & \hspace{10mm} \times\Big[
     \alpha_2\snorm{\xiI-\xi_{\pi}}{2,\Omega}^2 +
     \alpha_1\snorm{\xiI-\xi_{\pi}}{1,\Omega}^2 +
     \alpha_0\snorm{\xiI-\xi_{\pi}}{0,\Omega}^2 \Big]^{\half}\\[0.25em]
   &\leq
   \Cs(\beta^{\ast},\alpha_2,\alpha_1,\alpha_0)\hh^k\snorm{\us}{k+1,\Omega}
   \Big[ \alpha_2 + \alpha_1\hh^2 + \alpha_0\hh^3 \Big]^{\half}\snorm{\us-\ush}{1,\Omega}.
 \end{align*}
 Finally, we substitute the upper bounds of terms $\TERM{T1}$,
 $\TERM{T2}$, and $\TERM{T3}$ into \eqref{eq:main:H1}, and obtain the
 assertion of the lemma for $\hh\to0$.
\end{proof}

Now, we focus to derive the convergence analysis in the $\LTWO$ norm.
\begin{theorem}[Convergence in $\LS{2}$-norm]
  \label{ThmL2}
  Let $\us$ be the solution of the variational
  problem~\eqref{eq:varform} with $\fs\in\HS{k-1}(\Omega)$ and $\ush$
  the solution of the virtual element method~\eqref{eq:VEM} under mesh
  regularity assumptions~\ref{Mesh:regularity} (as in
  Theorem~\ref{ThmH2}).
  Then, the following estimates holds
  \begin{subequations}
    \begin{align}
      \norm{\us-\ush}{0,\Omega} &\leq \Cs\hh^2    \Big( \snorm{\us}{3,\Omega}   + \norm{\fs}{1,\Omega}   \Big) && \text{for}~k=2     \label{eq:L2Estimate:1} \\[0.5em]
      \norm{\us-\ush}{0,\Omega} &\leq \Cs\hh^{k+s} \Big( \snorm{\us}{k+1,\Omega} + \norm{\fs}{k-1,\Omega} \Big) && \text{for}~k\geq 3 \label{eq:L2Estimate:2},
    \end{align}
  \end{subequations}
  where $s$ is the regularity index of the function $u$ as mentioned
  in \eqref{eq:regularity:2}.
\end{theorem}
\begin{proof}
  Estimate \eqref{eq:L2Estimate:1} directly follows from an
  application of Theorem~\ref{ThmH1}, Eq.~\eqref{eq:ThmH1}, and the
  Poincar\'e inequality
  $\norm{\us-\ush}{0,\Omega}\leq\Cs\snorm{\us-\ush}{1,\Omega}$, which
  holds since $\us-\ush\in\HSzr{2}(\Omega)$.
  To prove inequality~\eqref{eq:L2Estimate:2}, we use the duality argument.
  Let $\xi\in\HSzr{2}(\Omega)$ be the solution of the variational formulation of the auxiliary
  problem
  \begin{subequations}
    \begin{align}
      \alpha_2\Delta^2\xi - \alpha_1\Delta\xi + \alpha_0\xi &= -\Delta(\us-\ush)\phantom{0} \quad\textrm{in~}\Omega,
      \label{eq:L2:dualProblem:A}\\
      \xi = \partial_n\xi                                  &= 0\phantom{-\Delta(\us-\ush)} \quad\textrm{on~}\Gamma.
      \label{eq:L2:dualProblem:B}
    \end{align}
  \end{subequations}
  According to~\cite{Grisvard:1992,Chinosi-Marini:2016}, we deduce the
  following regularity result:
  \begin{align}
    \norm{\xi}{3+s} \leq \Cs\norm{\us-\ush}{0,\Omega}
    \label{eq:interpo:xi}
  \end{align}
  for some real number $s\in(1/2,1]$.
  We multiply \eqref{eq:L2:dualProblem:A} by $\us-\ush$ and integrate
  over the domain $\Omega$, and we find that
  \begin{align*}
    \norm{\us-\ush}{0,\Omega}^2
    = (\us-\ush,\us-\ush)
    = \Big( \alpha_2\Delta^2\xi - \alpha_1\Delta\xi + \alpha_0\xi, \us-\ush \Big).
  \end{align*}
  Using the same arguments of the derivation of \eqref{eq:main:H1}, we
  obtain
  \begin{align}
    \norm{ \us-\ush}{0,\Omega}^2
    &= \calA(\us-\ush,\xi-\xiI)
    + \Big( (\fs,\xiI) - (\fsh,\xiI) \Big)
    + \Big( \calAh(\ush,\xiI) - \calA(\ush,\xiI) \Big)
    \nonumber\\
    &= \TERM{T1} + \TERM{T2} + \TERM{T3}.
    \label{eq:L2:main}
  \end{align}

  Upon splitting the bilinear form $\calA(\cdot,\cdot)$ into the
  $\HS{2}$-, $\HS{1}$-, and $\LS{2}$-scalar products, using the
  continuity of the bilinear forms, and again the interpolation
  estimate for $\xi-\xiI$ and~\eqref{eq:interpo:xi}, we obtain that
  \begin{align}
    \TERM{T1}
    &= \alpha_2\astwo(\us-\ush,\xi-\xiI) + \alpha_1\asone(\us-\ush,\xi-\xiI) + \alpha_0\aszer(\us-\ush,\xi-\xiI) \nonumber\\
    & \leq
    \Cs(\alpha_2)\hh^{k+s}\snorm{\us}{k+1,\Omega}\,\norm{\us-\ush}{0,\Omega} + 
    \Cs(\alpha_1)\hh^{k+s+1}\snorm{\us}{k+1,\Omega}\,\norm{\us-\ush}{0,\Omega} \nonumber\\
    &\quad
    +\Cs(\alpha_0)\hh^{k+s+2} \norm{\us}{k+1,\Omega}\,\norm{\us-\ush}{0,\Omega}\nonumber\\
    & \leq \hh^{k+s} \Big(
    \Cs(\alpha_2) \snorm{\us}{k+1,\Omega} +
    \Cs(\alpha_1) \hh  \snorm{\us}{k+1,\Omega} +
    \Cs(\alpha_0) \hh^2\snorm{\us}{k+1,\Omega} \Big) \norm{\us-\ush}{0,\Omega}.
    \label{L2:T1}
  \end{align}
  Since we assume that $\fs\in\HS{k-1}(\Omega)$, we bound the second
  term by noting that for all $k\geq3$, the quantity
  $(\fs-\PizP{k-2}\fs)$ is $\LS{2}$-orthogonal to the linear polynomial
  functions, e.g., $\PizP{1}\xiI$.
  Therefore, using the Cauchy-Schwarz inequality, the approximation
  properties of the interpolation operator, the $\LS{2}$-orthogonal
  projection operator, and~\eqref{eq:xi:interpolation:error}
  and~\eqref{eq:H1:dualProblem:05} yield:
  \begin{align}
    \TERM{T2}
    & \leq \sum_{\P\in\Th} \int_{\P}\big(\fs-\PizP{k-2}\fs\big)\,\big(\xiI-\PizP{1}\xiI\big)\dxv 
      \leq \norm{\fs-\Piz{k-2}\fs}{0,\Omega}\,\norm{\xiI-\Piz{1}\xiI}{0,\Omega} \nonumber\\
    & \leq \Cs\hh^{k-1}\snorm{\fs}{k-1,\Omega}\,\Big( \norm{\xiI -\xi}{0,\Omega} + \norm{ \xi -\Piz{1}\xi }{0,\Omega} + \norm{ \Piz{1}(\xi-\xiI) }{0,\Omega}\Big) \nonumber\\
    & \leq \Cs\hh^{k}\snorm{\fs}{k-1,\Omega}\,\norm{\us-\ush}{0,\Omega}.
    \label{L2:T2}
  \end{align}
  Finally, we proceed to estimate last term $\TERM{T3}$.
  To this end, we employ the polynomial consistency property of the
  discrete bilinear forms, the polynomial approximation property of
  Lemma~\ref{lemma:polyApprox}, and the approximation property of the
  interpolation operator of Lemma~\ref{lemma:VEM:interpolant}, and we
  deduce
  \begin{align}
    \TERM{T3}
    &\leq \Cs(\beta^{\ast}) \Big[
      \alpha_2\snorm{\ush-\us_{\pi}}{2,\Omega}^2 +
      \alpha_1\snorm{\ush-\us_{\pi}}{1,\Omega}^2 +
      \alpha_0\snorm{\ush-\us_{\pi}}{0,\Omega}^2 \Big ]^{\half}\nonumber\\
    & \quad\quad \times\Big[
      \alpha_2\snorm{\xiI-\xi_{\pi}}{2,\Omega}^2 +
      \alpha_1\snorm{\xiI-\xi_{\pi}}{1,\Omega}^2 +
      \alpha_0\snorm{\xiI-\xi_{\pi}}{0,\Omega}^2 \Big ]^{\half}\nonumber\\
    &\leq \Cs(\beta^{\ast}) \Big[
      \alpha_2\snorm{\ush-\us_{\pi}}{2,\Omega}^2 +
      \alpha_1\snorm{\ush-\us_{\pi}}{1,\Omega}^2 +
      \alpha_0\snorm{\ush-\us_{\pi}}{0,\Omega}^2 \Big ]^{\half}\nonumber\\
    & \qquad \times\hh^{1+s}\Big[ \alpha_2 + \alpha_1\hh^2 + \alpha_0\hh^{4}\Big]^{\half} \norm{\xi}{3+s,\Omega}\nonumber\\
    &\leq \Cs\big(\beta^{\ast},\alpha_2,\alpha_1,\alpha_0\big) \hh^{k+s}\snorm{\us}{k+1}\,\norm{\us-\ush}{0,\Omega}.
    \label{L2:T3}
  \end{align}
  Upon inserting the estimates of \eqref{L2:T1}, \eqref{L2:T2}, and
  \eqref{L2:T3} into \eqref{eq:L2:main}, we obtain the required result
  for $\hh\to0$.
\end{proof}




\section{Implementation}
\label{sec5:implementation}

In this section, we briefly describe how we implemented the VEM.
The approach follows the general guidelines
of~\cite{BeiraodaVeiga-Brezzi-Marini-Russo:2014}, here adapted to the
biharmonic problem.

\subsection{Vector and matrix notation}

We consider the following compact notation.
For all element $\P\in\Th$, we locally number the degrees of freedom
\DOFS{D1}, \DOFS{D2}, \DOFS{D3}, and \DOFS{D4} from $1$ to $\DOFcard$.
Then, we introduce the bounded, linear functionals
$\text{dof}_i:\Vsh{\ks}(\P)\rightarrow\mathbb{R}$,
$i=1,\dots,\DOFcard$,
such that 
\begin{align*}
  \dofOp{i}{\vsh}~:=\text{~$i$-th~degree~of~freedom~of~}\vsh
\end{align*}
for $\vsh\in\Vsh{\ks}(\P)$.
Let $\Lambda_{\P}=\big\{\dofOp{i}{\cdot}\big\}_{i}$ denote the set of
such functionals and collect the degrees of freedom of $\vsh$ in the
vector
$\vvh=\big(\dofOp{1}{\vsh},\ldots,\dofOp{\DOFcard}{\vsh}\big)^T$.
Since the degrees of freedom \DOFS{D1}, \DOFS{D2}, \DOFS{D3}, and
\DOFS{D4} are unisolvent in $\Vsh{\ks}(\P)$, the triplet
$\big(\P,\Vsh{\ks}(\P),\Lambda_{\P}\big)$ is a finite element in the sense
of Ciarlet, cf.~\cite{Ciarlet:2002}.
This property implies the existence of a Lagrangian basis set
$\big\{\varphi_i\big\}_i$, with $\varphi_i\in\Vsh{\ks}(\P)$,
$i=1,\dots,\DOFcard$, which satisfies
\begin{align*}
  \dofOp{i}{\varphi_j} = \delta_{ij}, \quad i,j = 1,2,\dots,\DOFcard.
\end{align*}
We refer to the basis function set $\big\{\varphi_i\big\}_i$ as the
``canonical'' basis of $\Vsh{\ks}(\P)$.
We introduce the compact notation
\begin{align*}
  \VirtualSFVec (\xv) = \left( \varphi_1(\xv),\dots,\varphi_{\DOFcard}(\xv) \right)\transpose,
\end{align*}
and write the expansion of a virtual element function $\vsh$ on such a
basis set as
\begin{align*}
  \vsh(\xv) = \VirtualSFVec(\xv)\transpose\vvh = \sum_{i=1}^{\DOFcard}\dofOp{i}{\vsh}\varphi_i(\xv)
  \qquad\forall\xv\in\P.
\end{align*}

We also introduce a compact notation for the basis of the polynomial subspace
$\PS{\ks}(\P)\subset\Vsh{\ks}(\P)$, which reads as
\begin{align*}
  \MonomialVec(\xv) = \left( \ms_1(\xv), \dots, \ms_{\PScard}(\xv) \right)\transpose,
\end{align*}
where $\PScard$ is the cardinality of $\PS{\ks}(\P)$.
Since the polynomials $\ms_\alpha(\xv)$ are also virtual element
functions, we can expand them on the canonical basis $\VirtualSFVec$.
We express such expansions as
\begin{align*}
  \MonomialVec(\xv)\transpose
  = \VirtualSFVec(\xv)\transpose\DMatrix,
\end{align*}
where matrix $\DMatrix$ has size $\DOFcard\times\PScard$ and collects
all the expansion coefficients
\begin{align*}
  \Ds_{i\ell} = \dofOp{i}{\ms_\ell},
\end{align*}
so that
\begin{align*}
  \ms_{\ell}(\xv) = \sum_{i=1}^{\DOFcard} \varphi_i(\xv)\Ds_{i\ell}
  \quad\ell=1,\ldots,\PScard.
\end{align*}

\medskip
Following this notation, we also express the action of a differential operator $\calD$, e.g.,
$\calD=\Delta$ or $\calD=\nabla$,
in a entry-wise way, so that
\begin{align*}
\calD\VirtualSFVec(\xv) = \left( \calD\varphi_1(\xv),\dots,\calD\varphi_{\DOFcard}(\xv) \right)\transpose,
\end{align*}
and
\begin{align*}
  \calD\MonomialVec(\xv) = \left( \calD\ms_1(\xv), \dots, \calD\ms_{\PScard}(\xv) \right)\transpose.
\end{align*}
Similarly, we express the action of the
projectors $\PiLP{k}$, $\PinP{k}$, and $\PizP{k}$ on the canonical basis functions
$\VirtualSFVec$
and their expansion
on the polynomial basis $\MonomialVec$ as follows:
\begin{align*}
  \PiLP{k}\VirtualSFVec\transpose
  &= \left[ \PiLP{k}\varphi_1, \PiLP{k}\varphi_2, \ldots \PiLP{k}\varphi_{\DOFcard} \right]
  = \MonomialVec\transpose \matPiLP{k},\\[0.5em]
  \PinP{k}\VirtualSFVec\transpose
  &= \left[ \PinP{k}\varphi_1, \PinP{k}\varphi_2, \ldots \PinP{k}\varphi_{\DOFcard} \right]
  = \MonomialVec\transpose \matPinP{k},\\[0.5em]
  \PizP{k}\VirtualSFVec\transpose
  &= \left[ \PizP{k}\varphi_1, \PizP{k}\varphi_2, \ldots \PizP{k}\varphi_{\DOFcard} \right]
  = \MonomialVec\transpose \matPizP{k}.
\end{align*}
The expansion coefficients for the three projection operators applied
to the basis function $\varphi_j$ are collected along the $j$-th
column of the projection matrices $\matPiLP{k}$, $\matPinP{k}$,
$\matPizP{k}$.

\subsection{Elliptic projector operator}
Using this compact notation, we rewrite the variational
problem~\eqref{eq:proj:PilP:a}-\eqref{eq:proj:PilP:b}
that
defines the polynomial projection $\PiLP{\ks}$ on the virtual element
space $\Vsh{\ks}(\P)$ as $\GMatrix\matPiLP{\ks}=\BMatrix$, where
$\GMatrix=\widetilde{\GMatrix}+\GMatrix^0\in\REAL^{\PScard\times\PScard}$,
$\BMatrix=\widetilde{\BMatrix}+\BMatrix^0\in\REAL^{\PScard\times\DOFcard}$,
with
\begin{align*}
  \widetilde{\GMatrix}
  &= \alpha_2\widetilde{\GMatrix}_2 + \alpha_1\widetilde{\GMatrix}_1
  = \alpha_2\int_\P\Delta\MonomialVec\,\Delta\MonomialVec\transpose\dxv
  + \alpha_1\int_\P\nabla\MonomialVec\cdot\nabla\MonomialVec\transpose\dxv,\\[1em]
  \widetilde{\BMatrix}
  &= \alpha_2\widetilde{\BMatrix}_2 + \alpha_1\widetilde{\BMatrix}_1
  = \alpha_2\int_\P\Delta\MonomialVec\,\Delta\VirtualSFVec\transpose\dxv
  +  \alpha_1\int_\P\nabla\MonomialVec\cdot\nabla\VirtualSFVec\transpose\dxv,
\end{align*}
and
\begin{align*}
  \GMatrix^0 =
  \left(\begin{array}{c}
    \widehat{\PinP{\ks}\MonomialVec\transpose} \\ \mathbf{0}\transpose \\ \vdots \\ \mathbf{0}\transpose
  \end{array}\right),\qquad
  \BMatrix^0 =
  \left(\begin{array}{c}
    \widehat{\PinP{\ks}\VirtualSFVec\transpose} \\ \mathbf{0}\transpose \\ \vdots \\ \mathbf{0}\transpose
  \end{array}\right),\qquad  
\end{align*}
We can split the linear system providing the projection matrix
$\matPiLP{\ks}$ into the summation of the two conditions
$\widetilde{\GMatrix}\matPiLP{\ks}=\widetilde{\BMatrix}$, which
expresses~\eqref{eq:proj:PilP:a}, and
$\GMatrix^0\matPiLP{\ks}=\BMatrix^0$, which
espresses~\eqref{eq:proj:PilP:b}.
The latter condition fixes the kernel of the differential operator
$\calL$.
Therefore, matrix $\GMatrix$ is nonsingular by construction, and we
can formally state that $\matPiLP{\ks}=\GMatrix^{-1}\BMatrix$.
\RED{
  A (double) integration by parts of the right-hand-side of the
  definition of $\widetilde{\BMatrix}$ yields
  \begin{align*}
    \alpha_2\widetilde{\BMatrix}_2 + \alpha_1\widetilde{\BMatrix}_1
    &= \alpha_2\int_\P\Delta\MonomialVec\,\Delta\VirtualSFVec\transpose\dxv
    +  \alpha_1\int_\P\nabla\MonomialVec\cdot\nabla\VirtualSFVec\transpose\dxv
    \\[0.5em]
    &= \int_{\P}\Big[\alpha_2\Delta^2\MonomialVec-\alpha_1\Delta\MonomialVec\Big]\VirtualSFVec\dxv
    + \int_{\partial\P}\Big[\big(\partial_n(-\alpha_2\MonomialVec+\alpha_1\MonomialVec)\big)\VirtualSFVec\transpose
      + \alpha_2\Delta\MonomialVec\big(\partial_n\VirtualSFVec\transpose\big)\Big]\dxv.
  \end{align*}
  According to what has already been observed about the computability
  of the bilinear form $\calB(\vsh,\qs)$, we see that the last
  right-hand side above can be easily computable by an evaluation of
  the volume and surface integrals through the degrees of freedom of
  $\VirtualSFVec$.
  We perform the numerical integration of the volume integral through
  the degrees of freedom \DOFS{D4} and the numerical integration of
  the edge integrals by evaluating the trace of each basis function
  $\phi_i$ and its normal derivative $\partial_n\phi_i$ through the
  univariate polynomial interpolation of the degrees of freedom
  \DOFS{D1}, \DOFS{D2}, and \DOFS{D3}.  }

\subsection{$\LS{2}$-orthogonal projector operator}

Using this compact notation, we rewrite the variational
problem~\eqref{eq:proj:PilP:a}-\eqref{eq:proj:PilP:b} that defines the
$\LTWO$-orthogonal polynomial projection $\PizP{\ks}$ on the virtual
element space $\Vsh{\ks}(\P)$ as $\Hv\,\matPizP{\ks}=\Cv$,
where
\begin{align*}
  \Hv
  = \int_\P\MonomialVec\,\MonomialVec\transpose\dxv
  \quad\text{and}\quad
  \Cv
  = \int_\P\MonomialVec\,\VirtualSFVec\transpose\dxv.
\end{align*}
We recall that the entries of matrix $\Cv$ that correspond to the
polynomial moments of the Lagrangian basis function $\VirtualSFVec$
against the polynomial of degree up to $k-2$ are computable from the
degrees of freedom \DOFS{D4}.
The enhancing condition in the definition of the local virtual element
space~\eqref{eq:Vsh-enh:def} makes it possible to compute the
polynomial moments of the basis functions $\VirtualSFVec$ against the
polynomials of $\PS{\ks}(\P)\setminus\PS{\ks-2}(\P)$ by using the elliptic
projector.
Since matrix $\Hv$ is nonsingular by construction, we can formally
state that $\matPizP{\ks}=\Hv^{-1}\Cv$.

\subsection{Local matrices}

The stiffness matrix is given by the sum of two terms: a
rank-deficient term, which is responsible for the accuracy of the
method, and a stability term, which fixes the correct rank:
\begin{align*}
  \StiffnessMatrix_\P := \StiffnessMatrix_\P^{\text{cons}} + \StiffnessMatrix_\P^{\text{stab}},
\end{align*}
where
\begin{align*}
  \left(\StiffnessMatrix_\P^{\text{cons}}\right)_{ij}
  = \calBhP\big(\PiLP{k}\varphi_i,\PiLP{k}\varphi_j\big)
  = \bigg(
  \big(\matPiLP{k}\big)^T
  \left[
    \alpha_2\int_{\P}\Delta\mv\Delta\mv^T\dxv +
    \alpha_1\int_{\P}\nabla\mv\cdot\nabla\mv^T\dxv
    \right]
  \matPiLP{k}
  \bigg)_{ij}
\end{align*}
and
\begin{align*}
  \left(\StiffnessMatrix_\P^{\text{stab}}\right)_{ij}
  &= \SPh\left(\left(\Iden{}-\PiLP{k}\right)\varphi_i,\left(\Iden{}-\PiLP{k}\right)\varphi_j\right)
  = \sum_{\ell}
  \text{dof}_{\ell}\left(\left(\Iden{}-\PiLP{k}\right)\varphi_i\right)\,
  \text{dof}_{\ell}\left(\left(\Iden{}-\PiLP{k}\right)\varphi_j\right)\\
  &=
  \bigg(\,\big(\matI-\matD\matPiLP{k}\big)^T\,\big(\matI-\matD\matPiLP{k}\big)\,\bigg)_{ij}.
\end{align*}

The construction of the local mass matrix uses the $\LS{2}$-orthogonal
projection:
\begin{align*}
  \left(\MassMatrix_\P\right)_{ij}
  := \left(\MassMatrix_\P^{\text{cons}}\right)_{ij}
  = \alpha_0\aszerP\big(\PizP{k}\varphi_i,\PizP{k}\varphi_j\big)
  = \bigg(
  \big(\matPizP{k}\big)^T
  \left[
    \alpha_0\int_{\P}\mv\mv^T\dxv
    \right]
  \matPizP{k}
  \bigg)_{ij}.
\end{align*}
Note that we do not need to specify a stabilization term in the local
mass matrix.

\subsection{Right-hand side approximation}

Using again the compact notation, we rewrite the virtual element
approximation of the right-hand side~\eqref{eq:VEM:RHS:def} as
\begin{align}
  \bilP{\fsh}{\varphi_i}
  = \int_{\P}\fs\PizP{\ks}\varphi_i\dxv
  = \bigg(\left[\int_{\P}\fs\mv^T\dxv\right]\matPizP{\ks}\bigg)_{i}.
  \label{eq:VEM:RHS:imple}
\end{align}

\subsection{Formulation of VEM \BLUE{in matrix representation}}
We assemble the local matrices $\StiffnessMatrix_\P$ and
$\MassMatrix_\P$, the forcing term
$\Fv_{\P}=\bilP{\fsh}{\VirtualSFVec}$ and incorporate the boundary
conditions into the global matrices $\StiffnessMatrix$ and
$\MassMatrix$ and the right-hand side vector $\Fv$ in a finite element
way.
The final linear system reads as:
\begin{align*}
  \big(\StiffnessMatrix+\MassMatrix\big)\uvh = \Fv,
\end{align*}
where $\uvh$ is the vector collecting the degrees of freedom of the
virtual element solution $\ush$.
The final matrix $\StiffnessMatrix+\MassMatrix$ is nonsingular by
construction because of the stability property imposed on the local
matrices $\StiffnessMatrix_\P+\MassMatrix_\P$ and the boundary
conditions.




\section{Numerical Results}
\label{sec6:numerical:results}

\begin{figure}
  \centering
  \begin{tabular}{cccc}
    \includegraphics[scale=0.15]{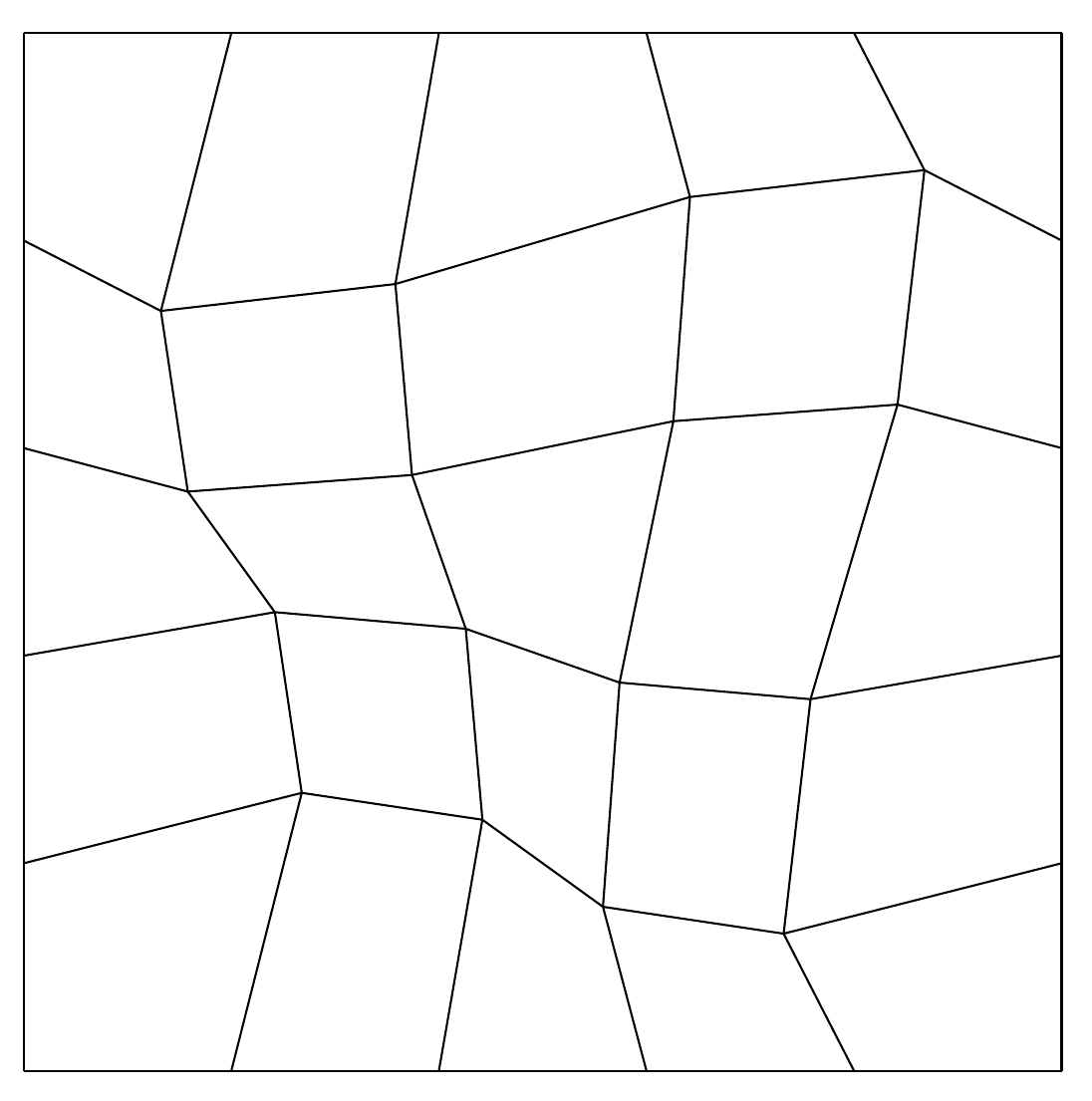} &
    \includegraphics[scale=0.15]{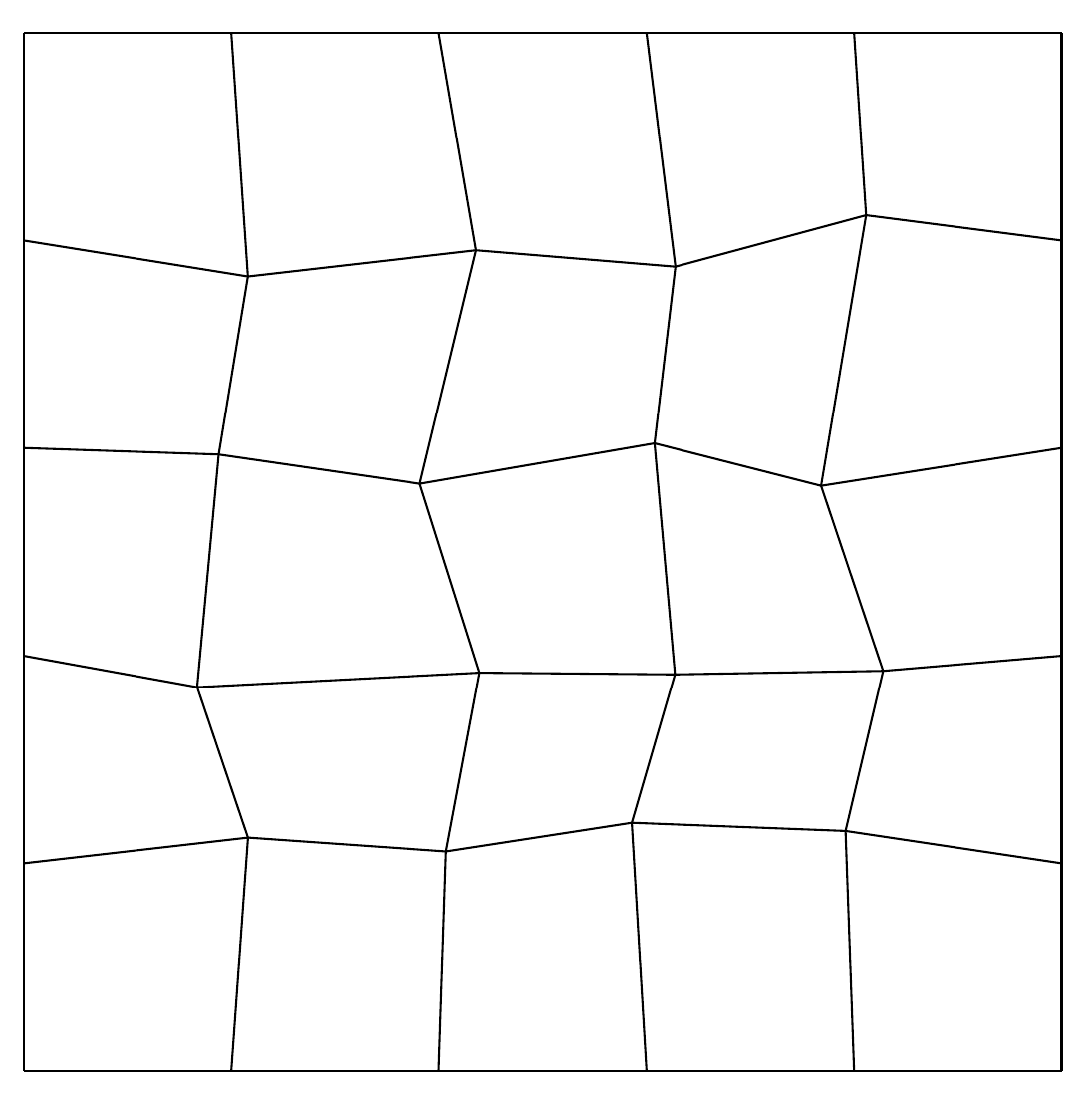} &
    \includegraphics[scale=0.15]{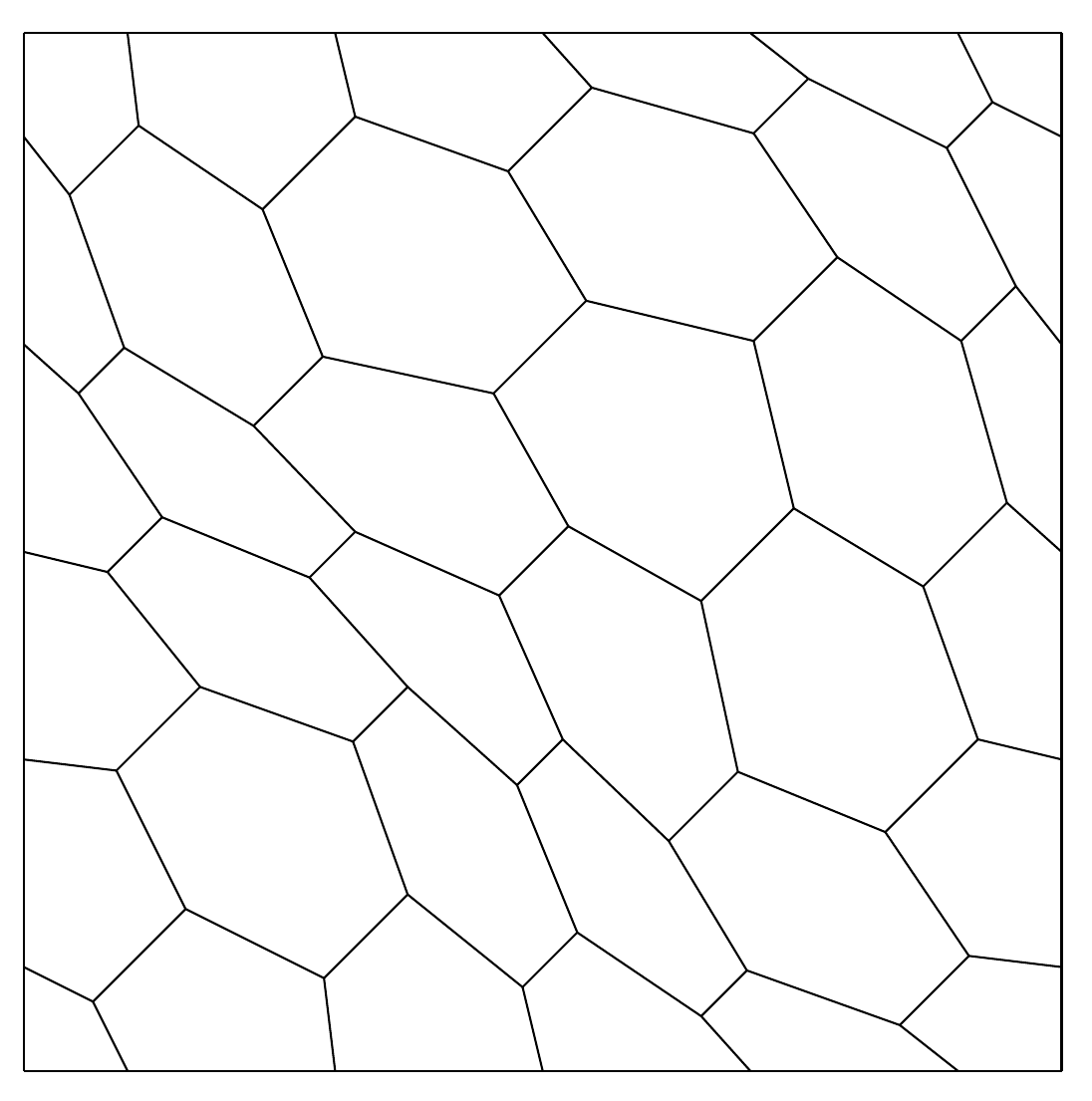} &
    \includegraphics[scale=0.15]{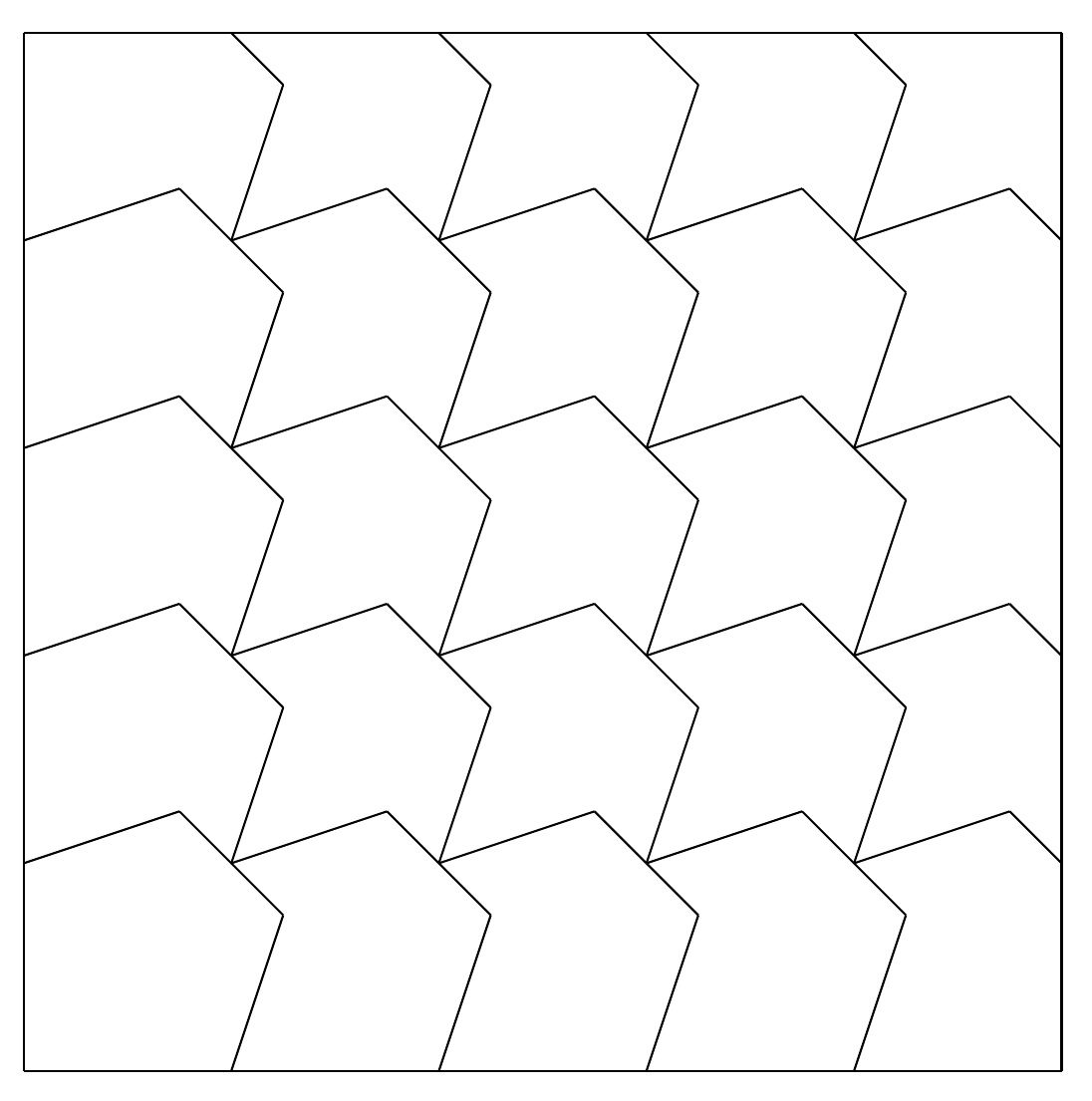} \\
    \includegraphics[scale=0.15]{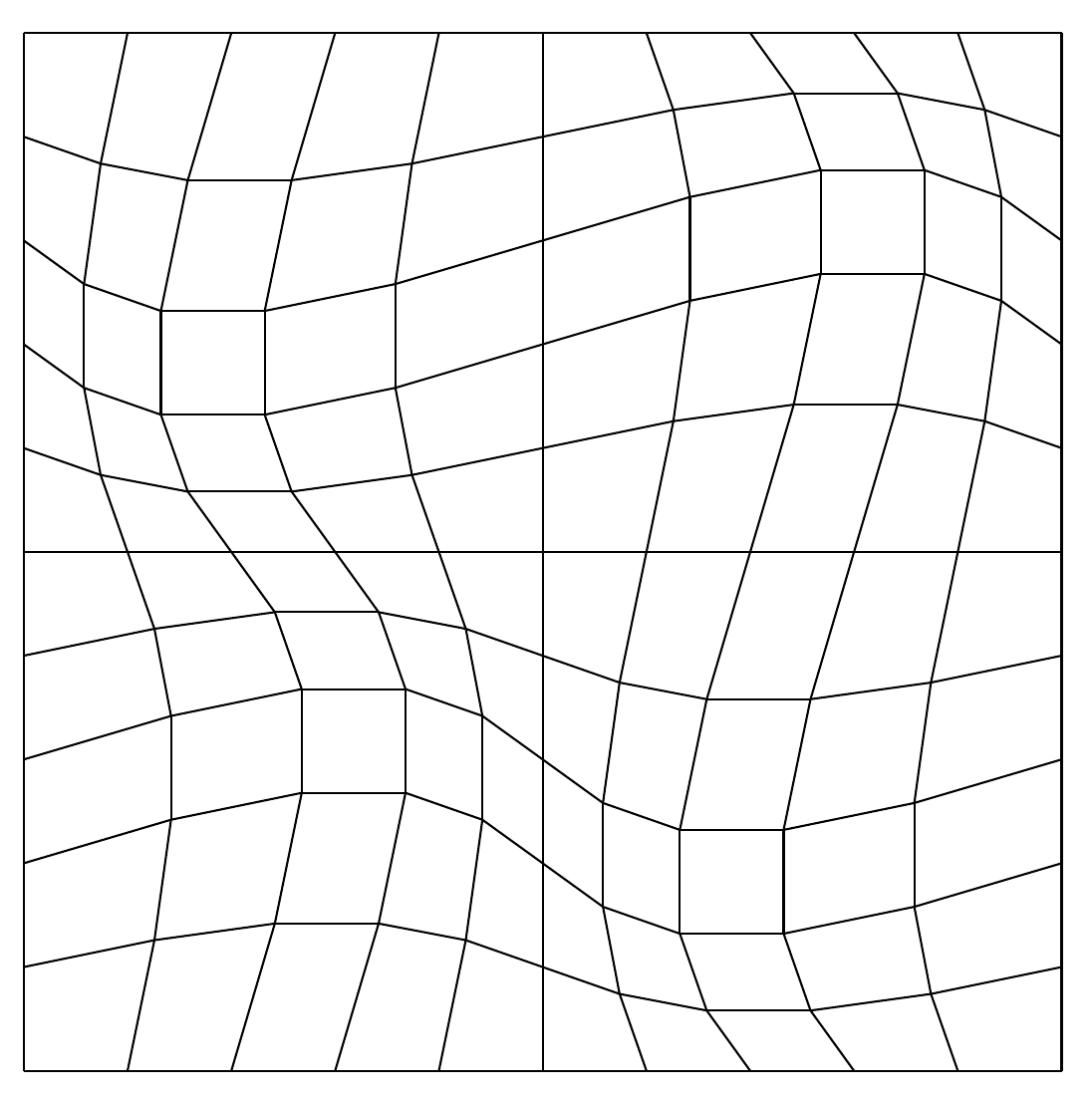} &
    \includegraphics[scale=0.15]{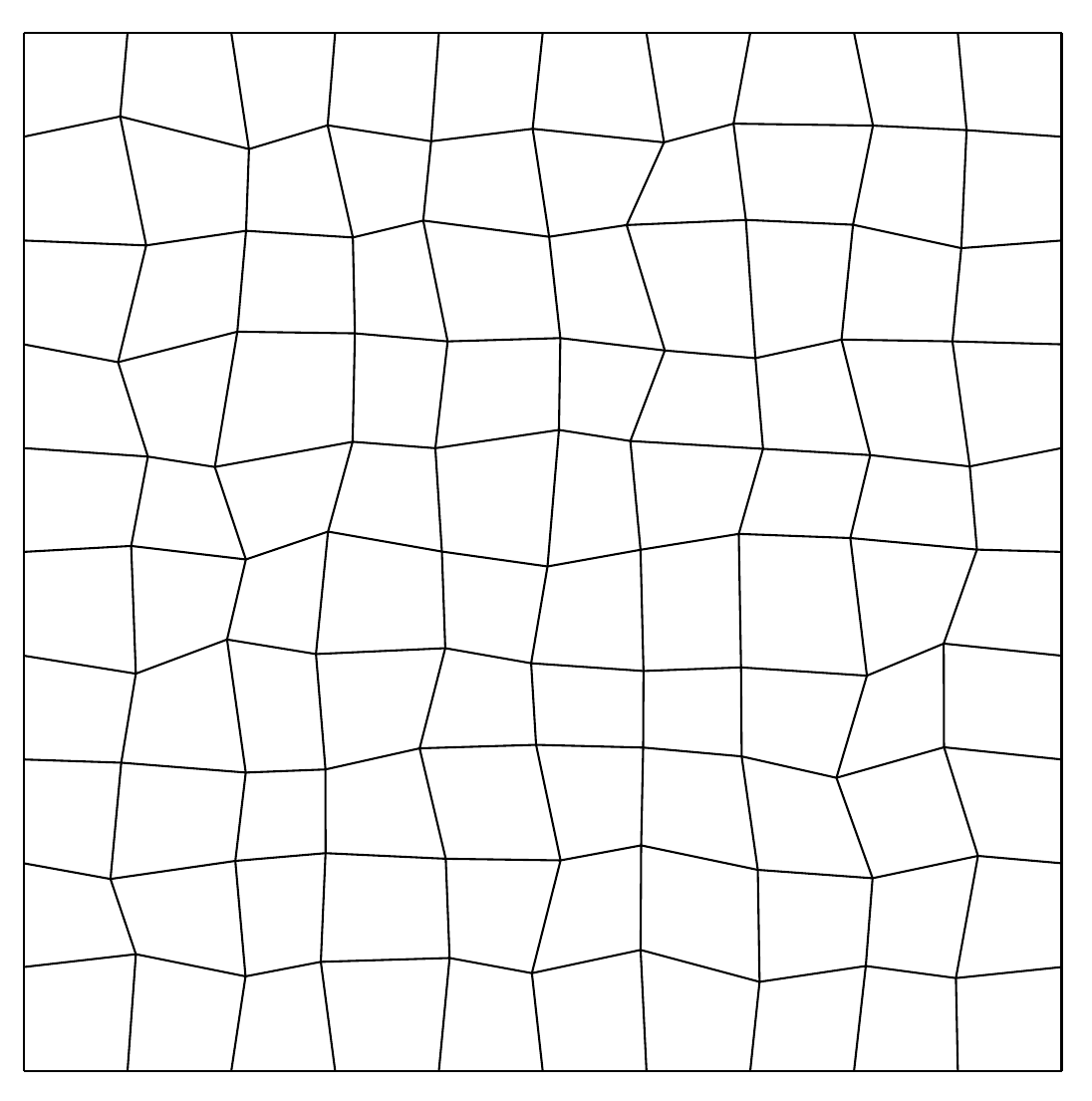} &
    \includegraphics[scale=0.15]{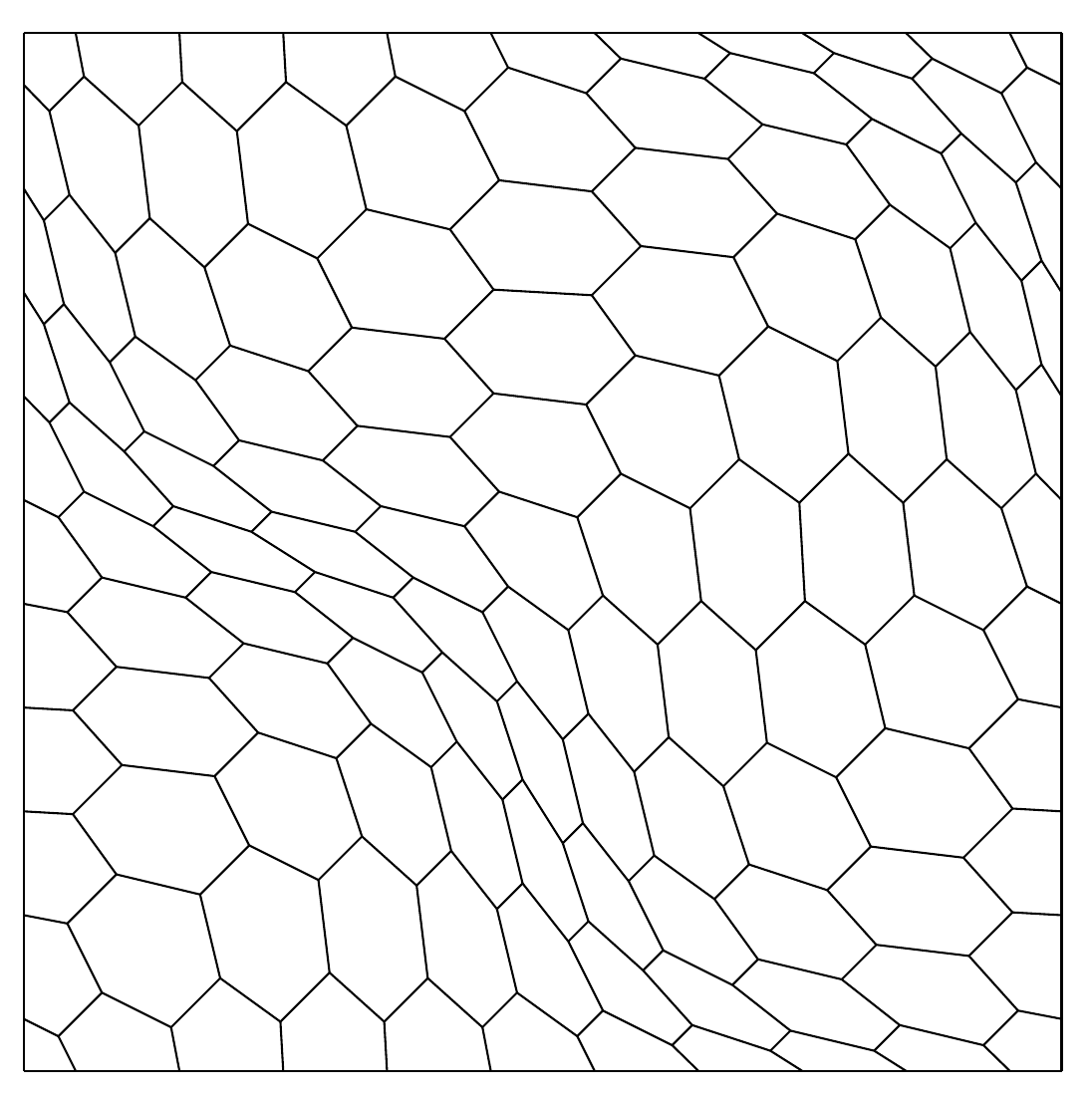} &
    \includegraphics[scale=0.15]{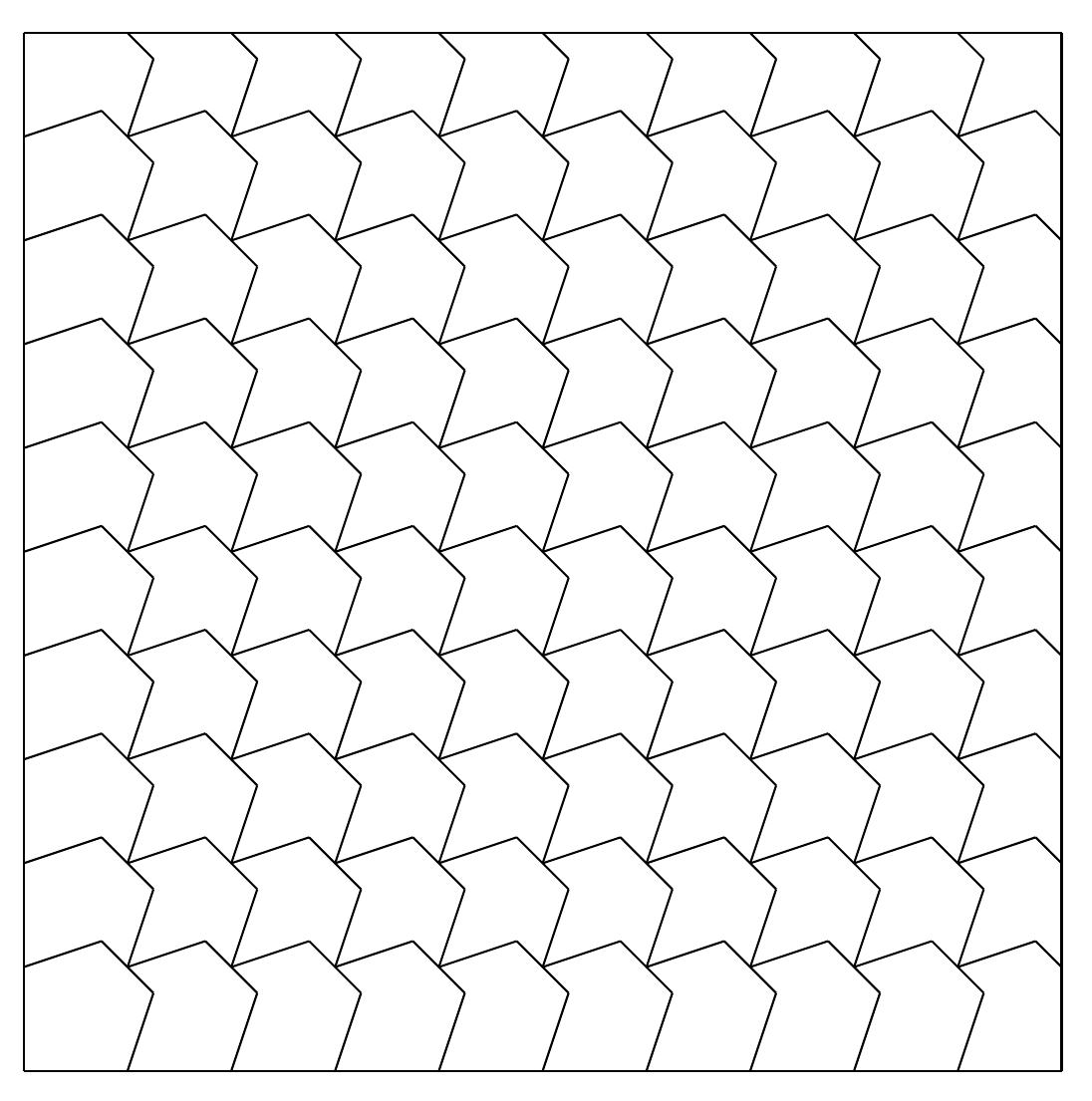} \\[0.25em]
    $a$ & $(b)$ & $(c)$ & $(d)$
  \end{tabular}
  \caption{Base mesh (top panel) and first refinement (bottom panel)
    of $(a)$ smoothly remapped quadrilateral mesh family; $(b)$
    randomized quadrilateral mesh family; $(c)$ remapped hexagonal
    mesh family; $(d)$ nonconvex octagonal mesh family.}
  \label{fig:mesh:family}
\end{figure}

We investigate the behavior of the virtual element method on two
manufactured test cases by numerically solving
problem~\eqref{eq:varform} with the virtual element
method~\eqref{eq:VEM} described in
Section~\ref{sec3:VEMapproximation}.
In the first test case, see Section~\ref{subsec:1st-test-case}, we
study the convergence behavior of the method.
In the second test case, see Section~\ref{subsec:2nd-test-case}, we
study the lenght-scale sensitivity of a bell-shaped crack-like
solution.

\medskip
In the first test case, we consider four representative mesh families
including $(a)$ smoothly remapped quadrilateral meshes; $(b)$
randomized quadrilateral meshes; $(c)$ smoothly remapped hexagonal
meshes; $(d)$ nonconvex octagonal meshes.
In the second test case, we consider only the mesh families $(b)$ and
$(c)$.
The construction of these meshes and the way they are refined is
rather standard, and details can easily be found in the virtual
element literature,
cf.~\cite{Berrone-Borio-Manzini:2018:CMAME:journal}.
For every mesh family, we consider a base mesh and eight refinements.
Tables~\ref{tab:Res-Quads-Remapped}, \ref{tab:Res-Quads-Randomized},
\ref{tab:Res-Hexa-Remapped}, and \ref{tab:Res-Octa-NonConvex} in the
final appendix report the mesh data and the number of degrees of
freedom of the virtual element approximations with polynomial order
$k=2,3,4$.
Figure~\ref{fig:mesh:family} shows the base mesh (top panel) and the
first refined mesh (bottom panel) of each mesh family.

\medskip
On any set of refined meshes, we measure the relative errors in the
$\LTWO$, $\HONE$, and energy norms.
Instead of the virtual element solution $\ush$, which is unknown, we
use its orthogonal polynomial projection $\Piz{k}\ush$.
We compute the $\LTWO$ relative error according to the formula
\begin{align}
  \text{error}_{\LTWO(\Omega)}(\ush) = \dfrac{\norm{\us-\Piz{k}\ush}{0,\Omega}}{\norm{\us}{0,\Omega}}
  \approx \dfrac{ \norm{\us-\ush}{0,\Omega} }{\norm{\us}{0,\Omega}};
  \label{eq:error:L2:norm}
\end{align}
the $\HS{1}$ relative errors according to the
formula
\begin{align}
  \text{error}_{\HS{1}(\Omega)}(\ush) = \frac{\snorm{\us-\Piz{k}\ush}{1,\hh}}{\snorm{\us}{1,\Omega}}
  \approx \dfrac{\snorm{\us-\ush}{1,\Omega}}{\snorm{\us}{1,\Omega}},
  \label{eq:error:Hn:norm}
\end{align}
and the energy error according the formula
\begin{align}
  \text{error}_{\calA_{\hh}}(\ush) = \left(\dfrac{\calA_{\hh}(\us-\Piz{k}\ush,\us-\Piz{k}\ush)}{\calA_{\hh}(\us,\us)}\right)^{\frac12}
  \approx \dfrac{ \norm{\us-\ush}{2,\Omega} }{\norm{\us}{2,\Omega}}.
  \label{eq:error:energy:norm}
\end{align}

\subsection{\BLUE{Test Case~1. Convergence behavior of a manufactured solution in various norms}}
\label{subsec:1st-test-case}



\begin{figure}[t]
  \centering
  \begin{tabular}{ccc}
    \begin{overpic}[scale=0.3]{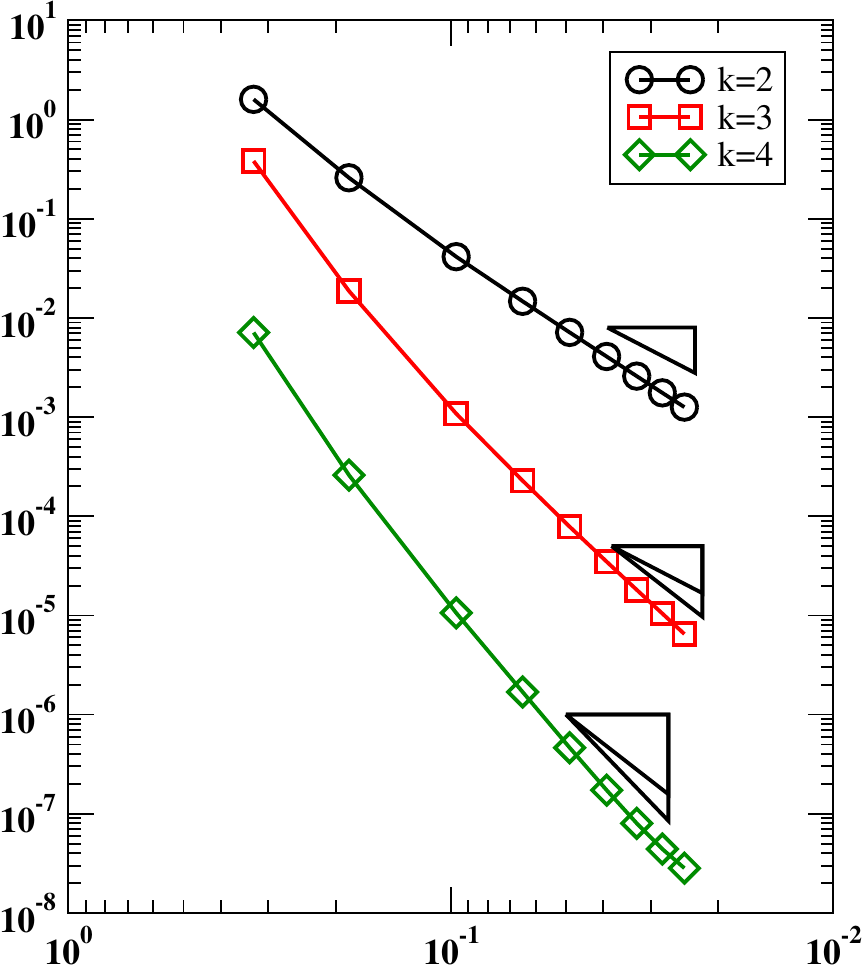} 
      \put( 45, -8){\textbf{h}}
      \put( -8, 7){\begin{sideways}\textbf{Energy Approximation Error}\end{sideways}}
      \put(73,  62){\begin{small}$\mathbf{1}$\end{small}}
      \put(73.5,39){\begin{small}$\mathbf{2}$\end{small}}
      \put(71,  20){\begin{small}$\mathbf{3}$\end{small}}
    \end{overpic}
    &\quad
    \begin{overpic}[scale=0.3]{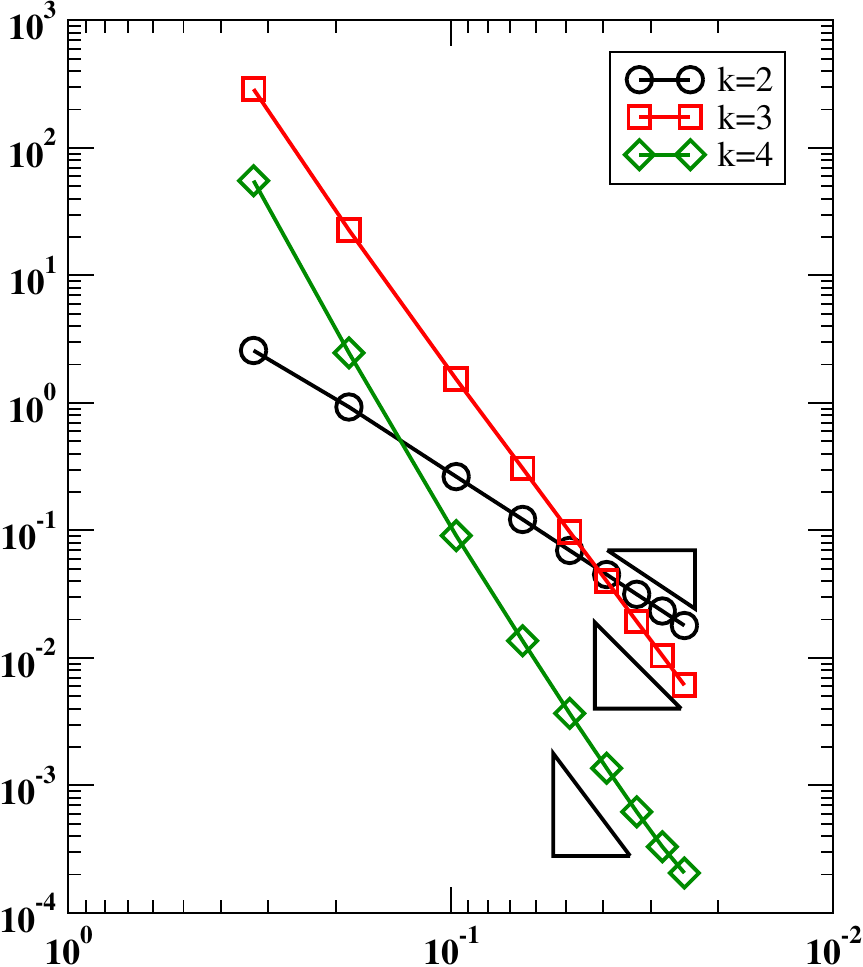} 
      \put( 45, -8){\textbf{h}}
      \put( -8, 14){\begin{sideways}\textbf{$\mathbf{H^1}$ Approximation Error}\end{sideways}}
      \put(73.5,39){\begin{small}$\mathbf{2}$\end{small}}
      \put(65,  21){\begin{small}$\mathbf{3}$\end{small}}
      \put(52.5,15){\begin{small}$\mathbf{4}$\end{small}}
    \end{overpic}
    &\quad
    \begin{overpic}[scale=0.3]{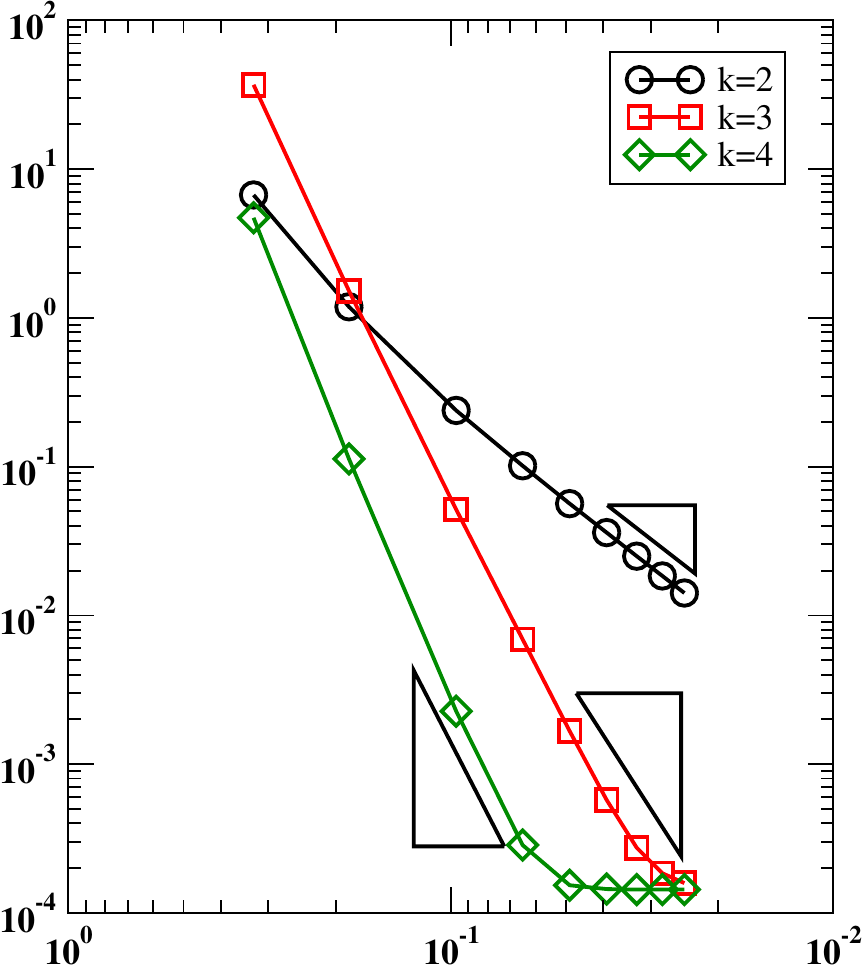} 
      \put( 45, -8){\textbf{h}}
      \put( -8, 14){\begin{sideways}\textbf{$\mathbf{L^2}$ Approximation Error}\end{sideways}}
      \put(73,41){\begin{small}$\mathbf{2}$\end{small}}
      \put(72,18){\begin{small}$\mathbf{4}$\end{small}}
      \put(38,18){\begin{small}$\mathbf{5}$\end{small}}
    \end{overpic}
    \\[0.75cm]
    \begin{overpic}[scale=0.3]{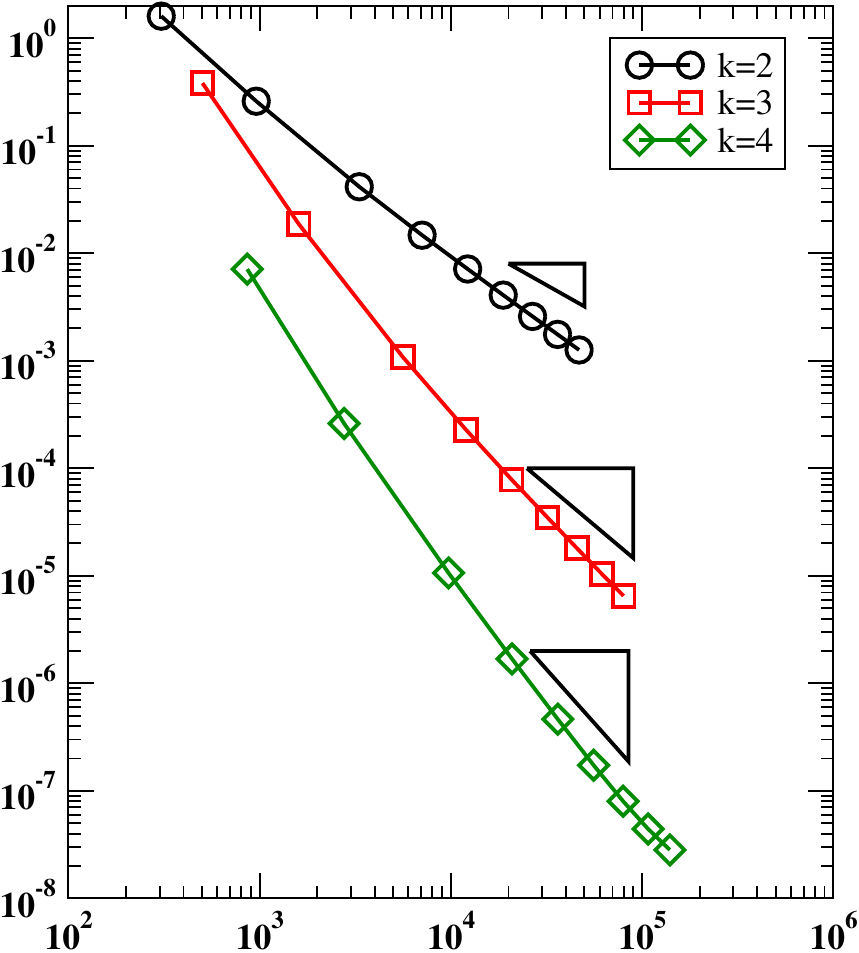} 
      \put( 35, -8){$\Ndfs$}
      \put( -8, 7){\begin{sideways}\textbf{Energy Approximation Error}\end{sideways}}
      \put(63,68){\begin{small}$\mathbf{\frac12}$\end{small}}
      \put(67,44){\begin{small}$\mathbf{1}$\end{small}}
      \put(67,24){\begin{small}$\mathbf{\frac32}$\end{small}}
    \end{overpic}
    &\quad
    \begin{overpic}[scale=0.3]{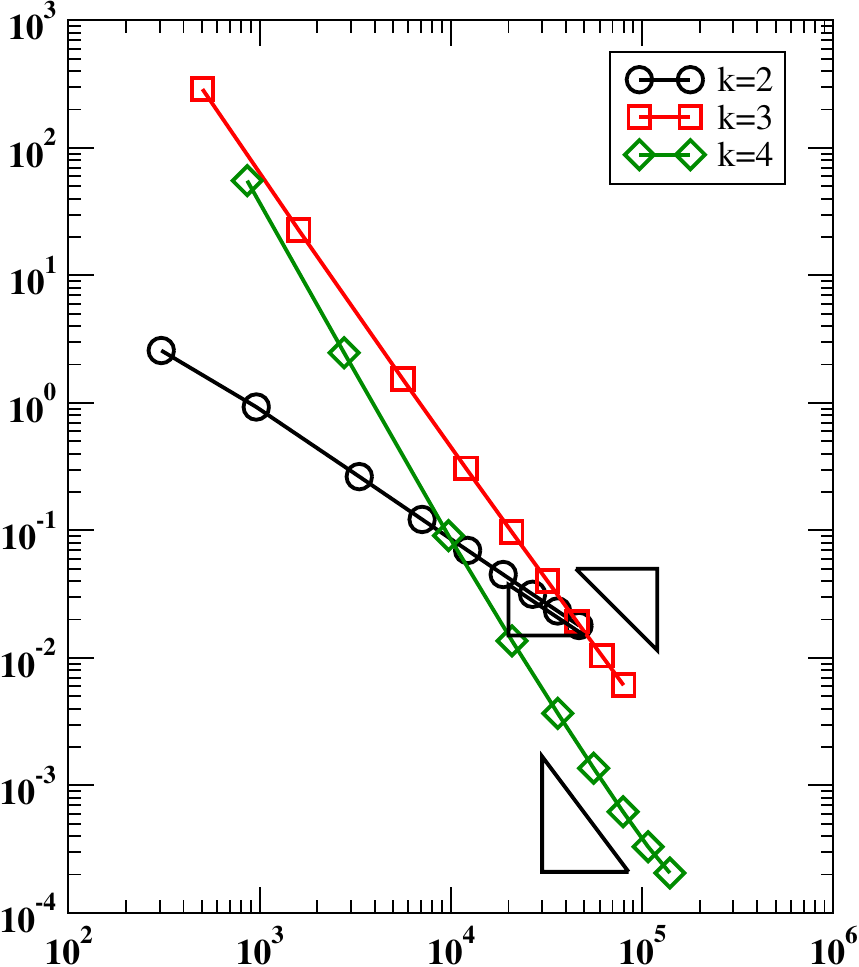} 
      \put( 35, -8){$\Ndfs$}
      \put( -8, 14){\begin{sideways}\textbf{$\mathbf{H^1}$ Approximation Error}\end{sideways}}
      \put(46,    35){\begin{small}$\mathbf{1}$\end{small}}
      \put(69,    36){\begin{small}$\mathbf{\frac32}$\end{small}}
      \put(50.5,13.5){\begin{small}$\mathbf{2}$\end{small}}
    \end{overpic}
    &\quad
    \begin{overpic}[scale=0.3]{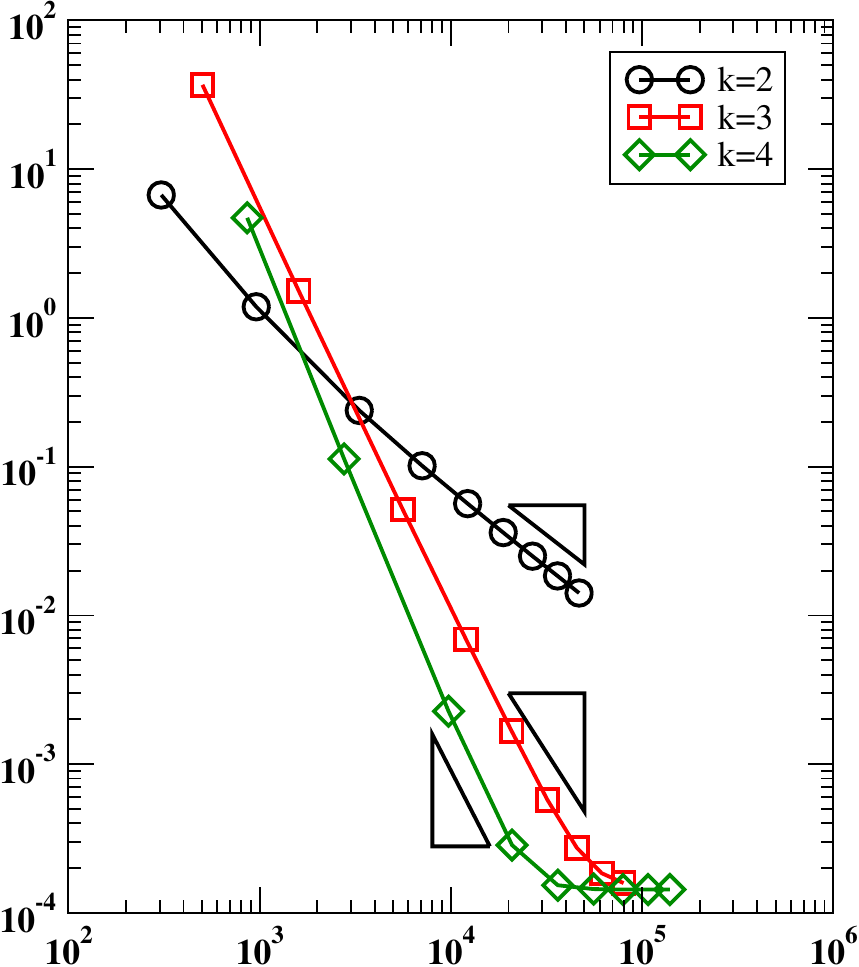} 
      \put( 35, -8){$\Ndfs$}
      \put( -8, 14){\begin{sideways}\textbf{$\mathbf{L^2}$ Approximation Error}\end{sideways}}
      \put(62,43){\begin{small}$\mathbf{1}$\end{small}}
      \put(62,21){\begin{small}$\mathbf{2}$\end{small}}
      \put(39,17){\begin{small}$\mathbf{\frac52}$\end{small}}  
    \end{overpic}
  \end{tabular}
  \vspace{0.25cm}
  \caption{Test Case 1. Error curve measured using energy norm, $\HS{2}$ norm,
    $\HS{1}$ norm and $\LS{2}$ norm. The calculations are carried out
    on the family of smoothly remapped quadrilateral meshes.}
  \label{fig:Res-Quads-Remapped}
\end{figure}
\begin{figure}[t]
  \centering
  \begin{tabular}{ccc}
    \begin{overpic}[scale=0.3]{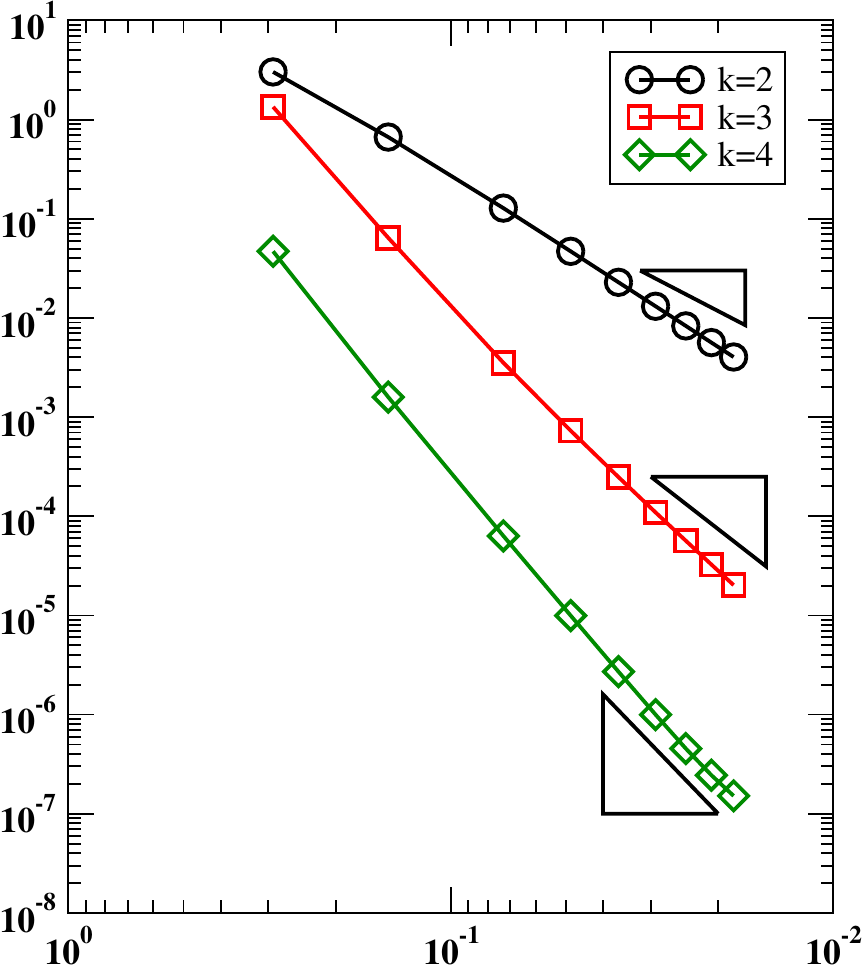} 
      \put( 45, -8){\textbf{h}}
      \put( -8, 7){\begin{sideways}\textbf{Energy Approximation Error}\end{sideways}}
      \put(78,67){\begin{small}$\mathbf{1}$\end{small}}
      \put(72,52){\begin{small}$\mathbf{2}$\end{small}}
      \put(57,20){\begin{small}$\mathbf{3}$\end{small}}
    \end{overpic}
    &\quad
    \begin{overpic}[scale=0.3]{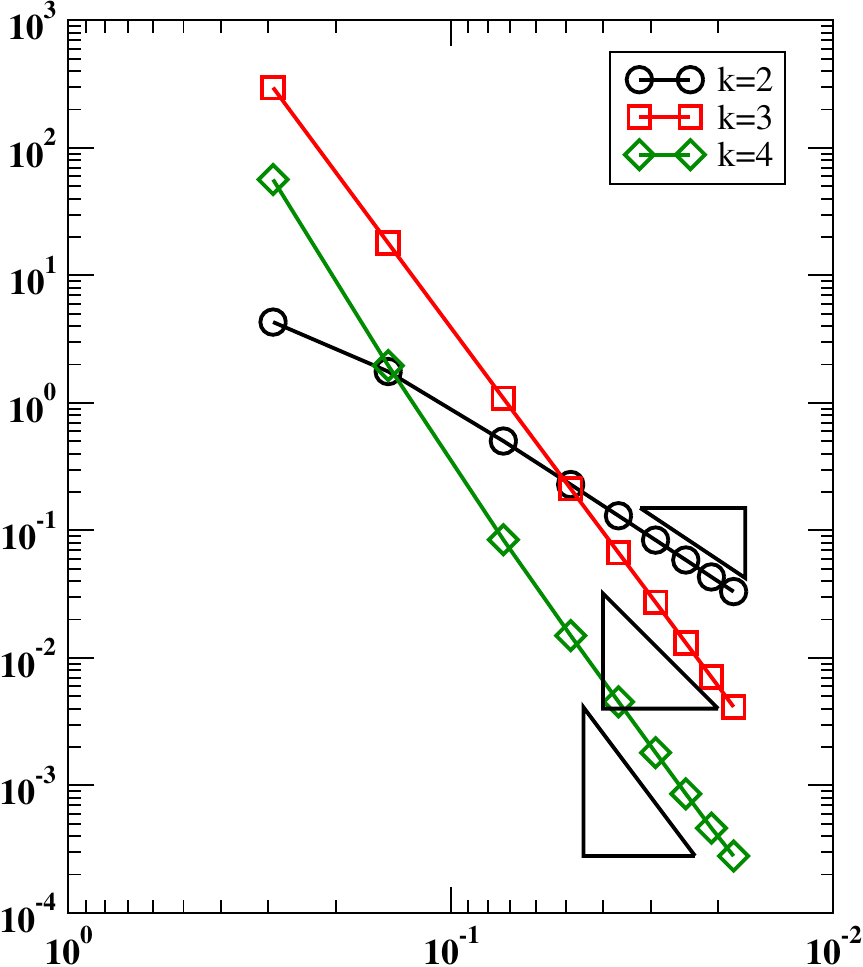} 
      \put( 45, -8){\textbf{h}}
      \put( -8, 14){\begin{sideways}\textbf{$\mathbf{H^1}$ Approximation Error}\end{sideways}}
      \put(70,  49){\begin{small}$\mathbf{2}$\end{small}}
      \put(70,20.5){\begin{small}$\mathbf{3}$\end{small}}
      \put(55,  17){\begin{small}$\mathbf{4}$\end{small}}
    \end{overpic}
    &\quad
    \begin{overpic}[scale=0.3]{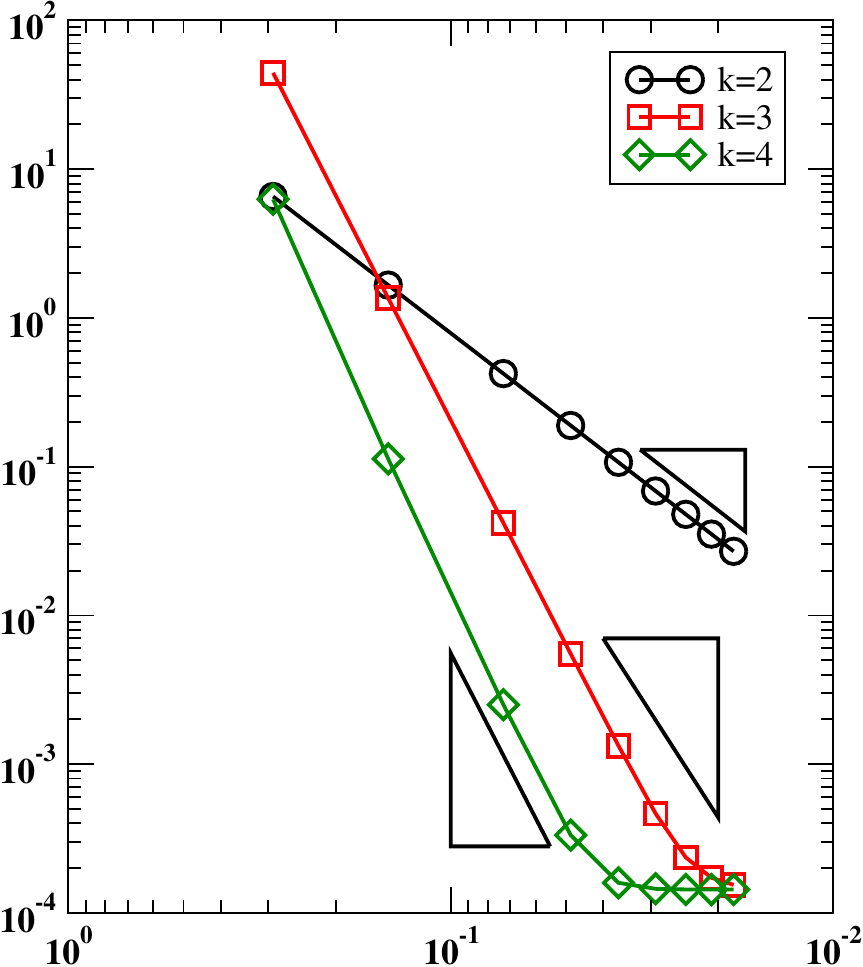} 
      \put( 45, -8){\textbf{h}}
      \put( -8, 14){\begin{sideways}\textbf{$\mathbf{L^2}$ Approximation Error}\end{sideways}}
      \put(78.5,47.5){\begin{small}$\mathbf{2}$\end{small}}
      \put(75.5,  24){\begin{small}$\mathbf{4}$\end{small}}
      \put(41.5,  20){\begin{small}$\mathbf{5}$\end{small}}
    \end{overpic}
    \\[0.75cm]
    \begin{overpic}[scale=0.3]{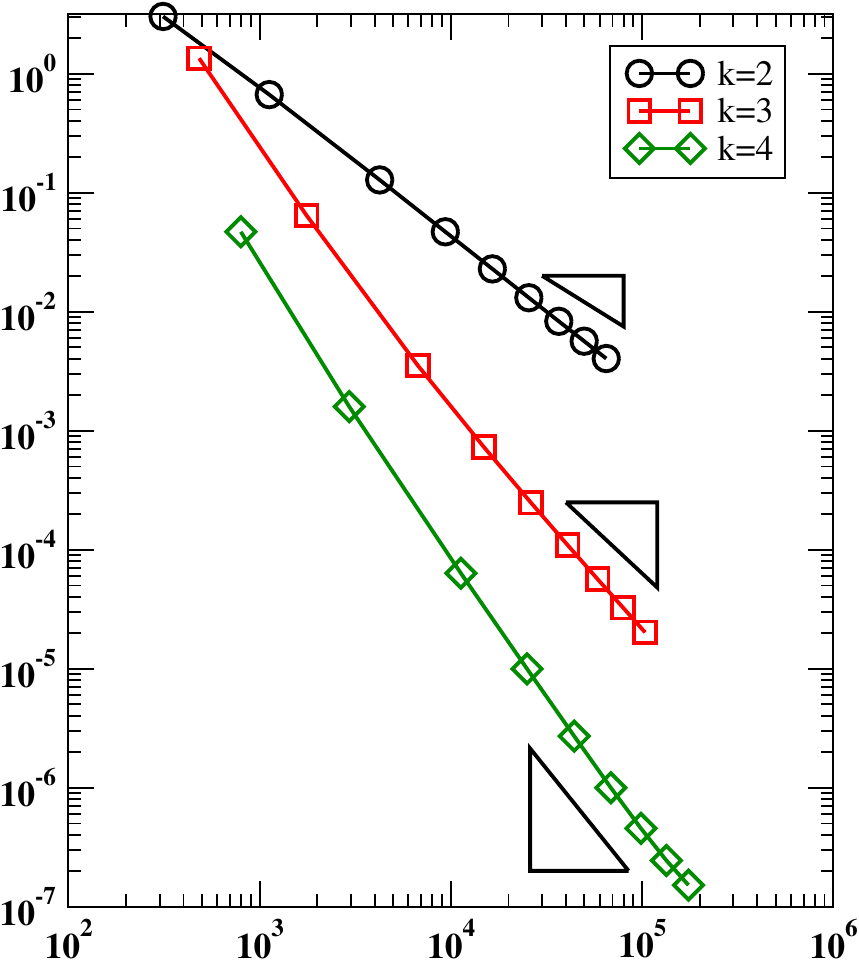} 
      \put( 35, -8){$\Ndfs$}
      \put( -8, 7){\begin{sideways}\textbf{Energy Approximation Error}\end{sideways}}
      \put(66,    67){\begin{small}$\mathbf{\frac12}$\end{small}}
      \put(70,    41){\begin{small}$\mathbf{1}$\end{small}}
      \put(49.5,13.5){\begin{small}$\mathbf{\frac32}$\end{small}}
    \end{overpic}
    &\quad
    \begin{overpic}[scale=0.3]{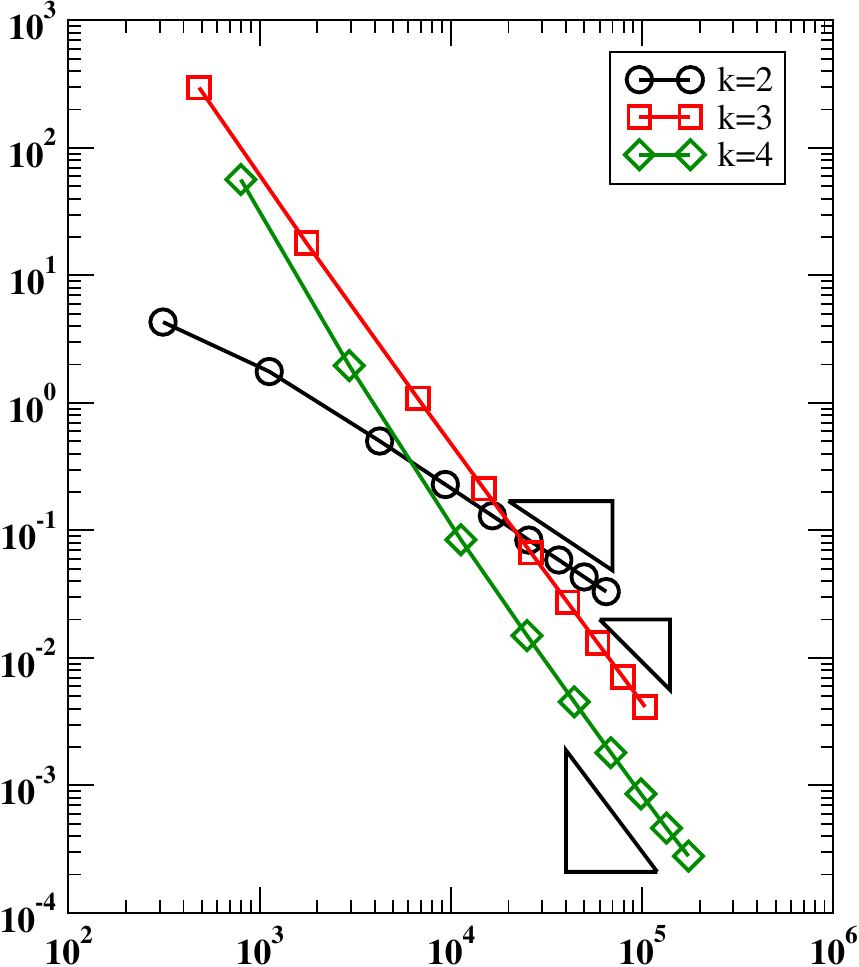} 
      \put( 35, -8){$\Ndfs$}
      \put( -8, 14){\begin{sideways}\textbf{$\mathbf{H^1}$ Approximation Error}\end{sideways}}
      \put(65,42){\begin{small}$\mathbf{1}$\end{small}}
      \put(70,31){\begin{small}$\mathbf{\frac32}$\end{small}}
      \put(53,14){\begin{small}$\mathbf{2}$\end{small}}
    \end{overpic}
    &\quad
    \begin{overpic}[scale=0.3]{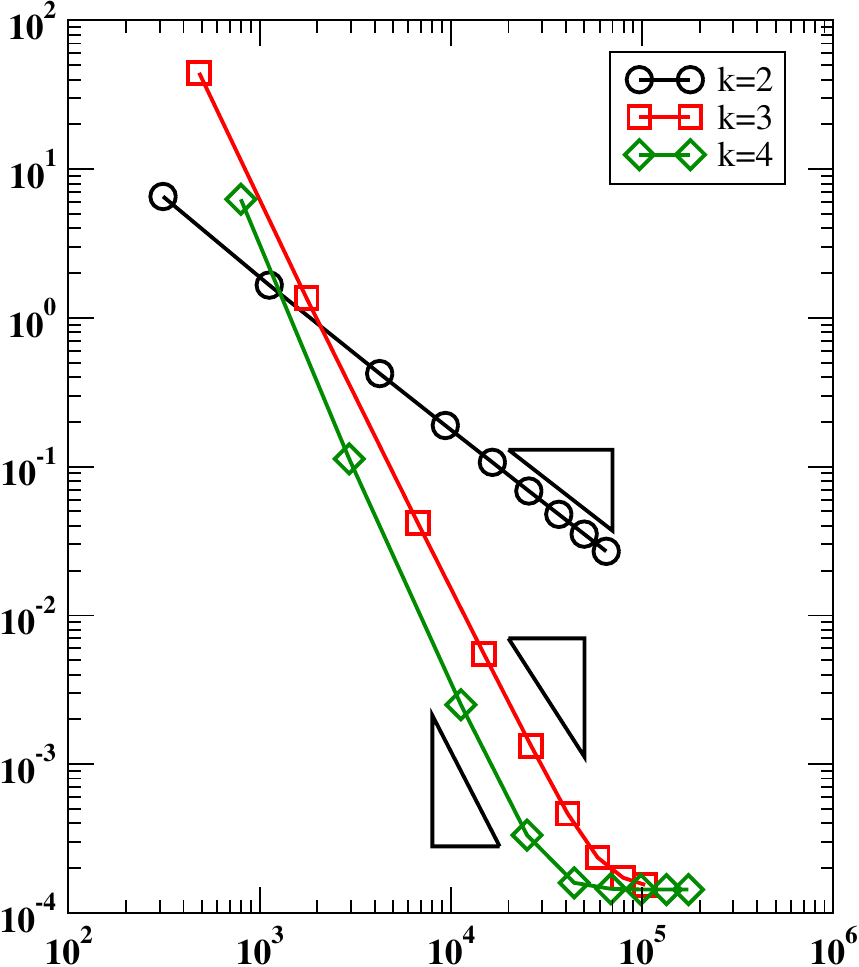} 
      \put( 35, -8){$\Ndfs$}
      \put( -8, 14){\begin{sideways}\textbf{$\mathbf{L^2}$ Approximation Error}\end{sideways}}
      \put(64.5,47.5){\begin{small}$\mathbf{1}$\end{small}}
      \put(62.5,  26){\begin{small}$\mathbf{2}$\end{small}}
      \put(39,    18){\begin{small}$\mathbf{\frac52}$\end{small}}
    \end{overpic}
  \end{tabular}
  \caption{Test Case 1. Error curve measured using energy norm, $\HS{2}$ norm,
    $\HS{1}$ norm and $\LS{2}$ norm. The calculations are carried
    out on the family of randomized quadrilateral meshes.}
  \label{fig:Res-Quads-Randomized}
\end{figure}
\begin{figure}[t]
  \centering
  \begin{tabular}{ccc}
    \begin{overpic}[scale=0.3]{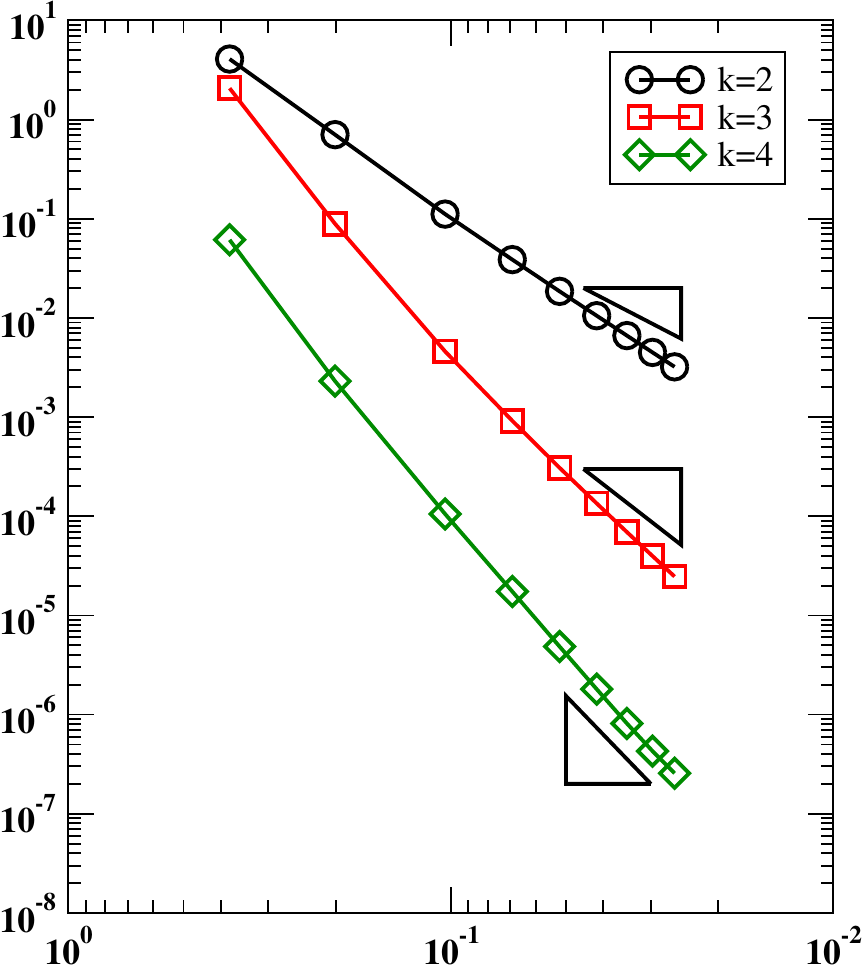} 
      \put( 45, -8){\textbf{h}}
      \put( -8, 7){\begin{sideways}\textbf{Energy Approximation Error}\end{sideways}}
      \put(71.5,65){\begin{small}$\mathbf{1}$\end{small}}
      \put(71.5,46){\begin{small}$\mathbf{2}$\end{small}}
      \put(53.5,21.5){\begin{small}$\mathbf{3}$\end{small}}
    \end{overpic}
    &\quad
    \begin{overpic}[scale=0.3]{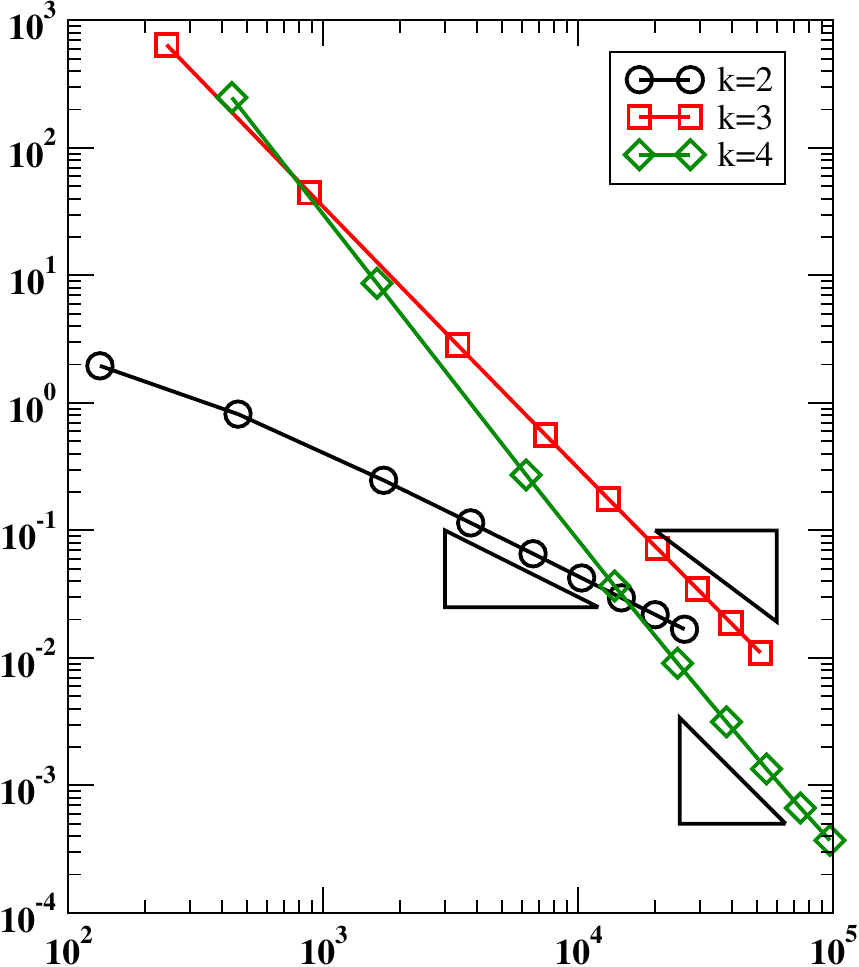} 
      \put( 45, -8){\textbf{h}}
      \put( -8, 14){\begin{sideways}\textbf{$\mathbf{H^1}$ Approximation Error}\end{sideways}}
      \put(41,    39){\begin{small}$\mathbf{2}$\end{small}}
      \put(72,  46.5){\begin{small}$\mathbf{3}$\end{small}}
      \put(65.5,18.5){\begin{small}$\mathbf{4}$\end{small}}
    \end{overpic}
    &\quad
    \begin{overpic}[scale=0.3]{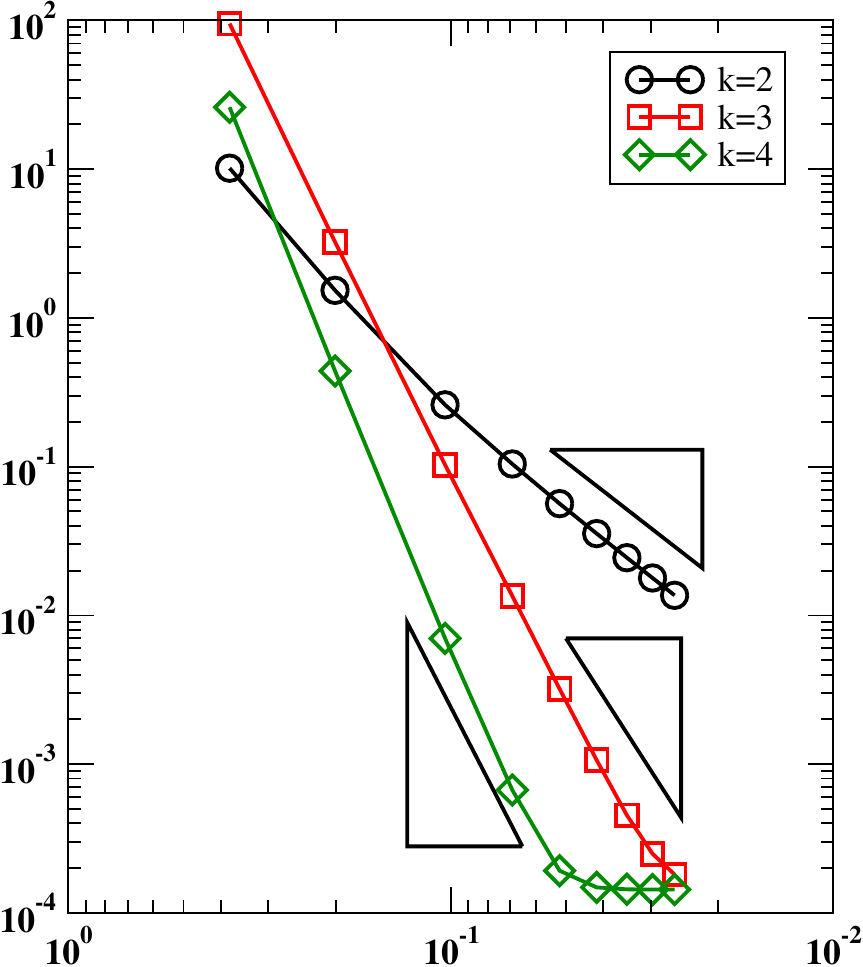} 
      \put( 45, -8){\textbf{h}}
      \put( -8, 14){\begin{sideways}\textbf{$\mathbf{L^2}$ Approximation Error}\end{sideways}}
      \put(74,45){\begin{small}$\mathbf{2}$\end{small}}
      \put(72,24){\begin{small}$\mathbf{4}$\end{small}}
      \put(37,20){\begin{small}$\mathbf{5}$\end{small}}
    \end{overpic}
    \\[0.75cm]
    \begin{overpic}[scale=0.3]{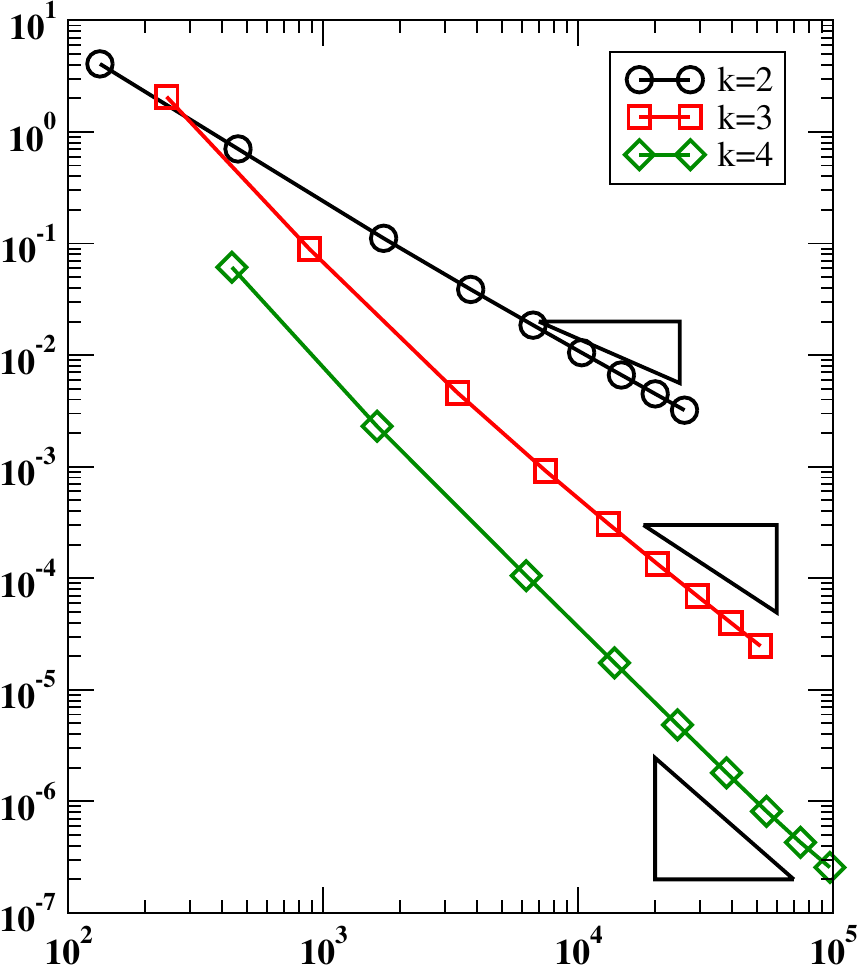} 
      \put( 35, -8){$\Ndfs$}
      \put( -8, 7){\begin{sideways}\textbf{Energy Approximation Error}\end{sideways}}
      \put(71.5,62){\begin{small}$\mathbf{\frac12}$\end{small}}
      \put(72,47.5){\begin{small}$\mathbf{1}$\end{small}}
      \put(62,  13){\begin{small}$\mathbf{\frac32}$\end{small}}
    \end{overpic}
    &\quad
    \begin{overpic}[scale=0.3]{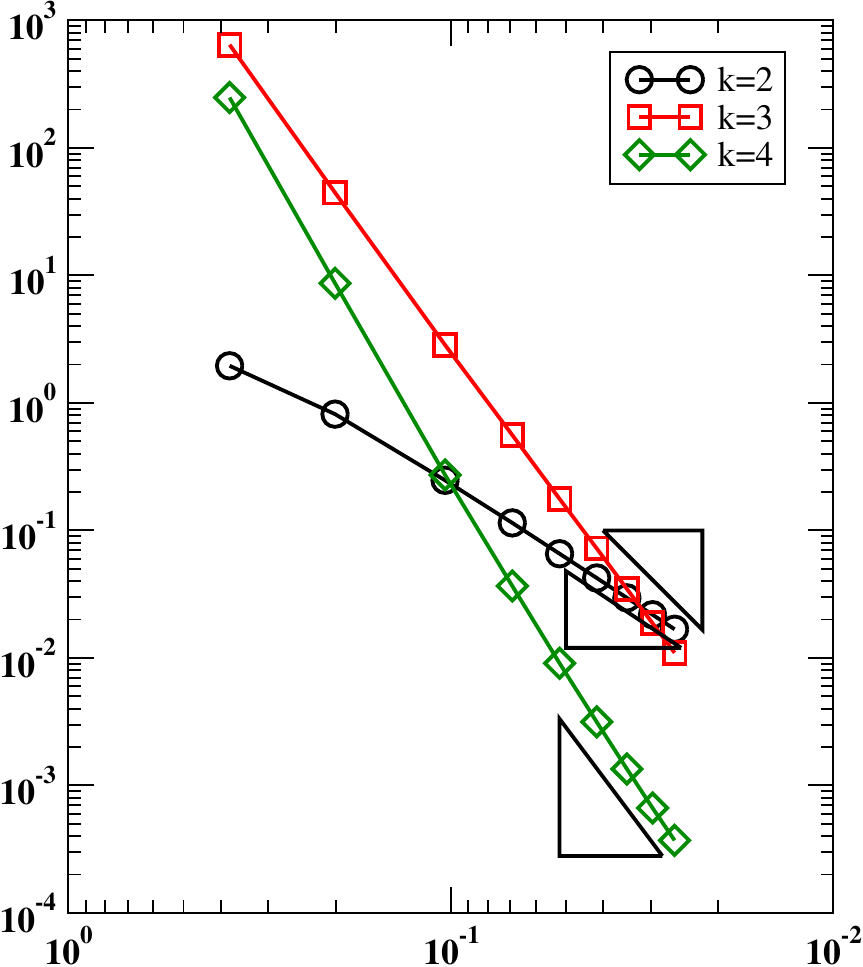} 
      \put( 35, -8){$\Ndfs$}
      \put( -8, 14){\begin{sideways}\textbf{$\mathbf{H^1}$ Approximation Error}\end{sideways}}
      \put(63,  27){\begin{small}$\mathbf{1}$\end{small}}
      \put(73,  38){\begin{small}$\mathbf{\frac32}$\end{small}}
      \put(52.5,17){\begin{small}$\mathbf{2}$\end{small}}
    \end{overpic}
    &\quad
    \begin{overpic}[scale=0.3]{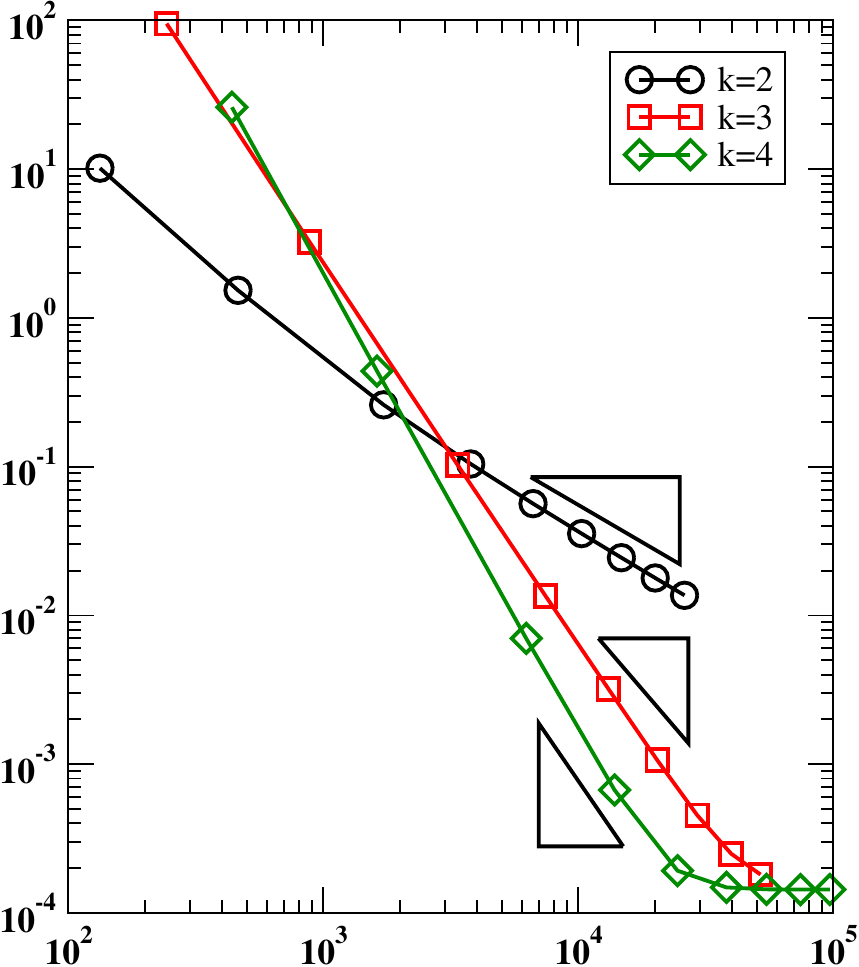} 
      \put( 35, -8){$\Ndfs$}
      \put( -8, 14){\begin{sideways}\textbf{$\mathbf{L^2}$ Approximation Error}\end{sideways}}
      \put(71,44){\begin{small}$\mathbf{1}$\end{small}}
      \put(72,27){\begin{small}$\mathbf{2}$\end{small}}
      \put(50,17){\begin{small}$\mathbf{\frac52}$\end{small}}
    \end{overpic}
  \end{tabular}
  \caption{Test Case 1. Error curve measured using energy norm, $\HS{2}$ norm,
    $\HS{1}$ norm and $\LS{2}$ norm. The calculations are carried out
    on the family of smoothly remapped hexagonal meshes.}
  \label{fig:Res-Hexa-Remapped}
\end{figure}
\begin{figure}[t]
  \centering
  \begin{tabular}{ccc}
    \begin{overpic}[scale=0.3]{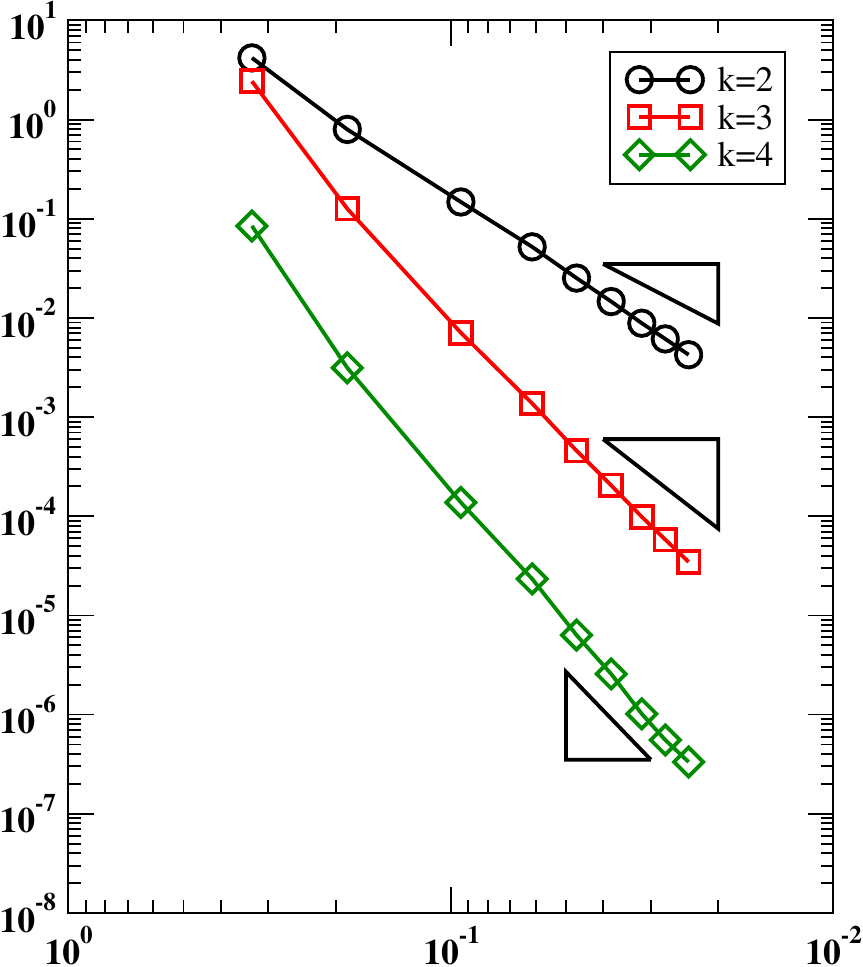} 
      \put( 45, -8){\textbf{h}}
      \put( -8, 7){\begin{sideways}\textbf{Energy Approximation Error}\end{sideways}}
      \put(75,67){\begin{small}$\mathbf{1}$\end{small}}
      \put(75,48){\begin{small}$\mathbf{2}$\end{small}}
      \put(53.5,23.5){\begin{small}$\mathbf{3}$\end{small}}
    \end{overpic}
    &\quad
    \begin{overpic}[scale=0.3]{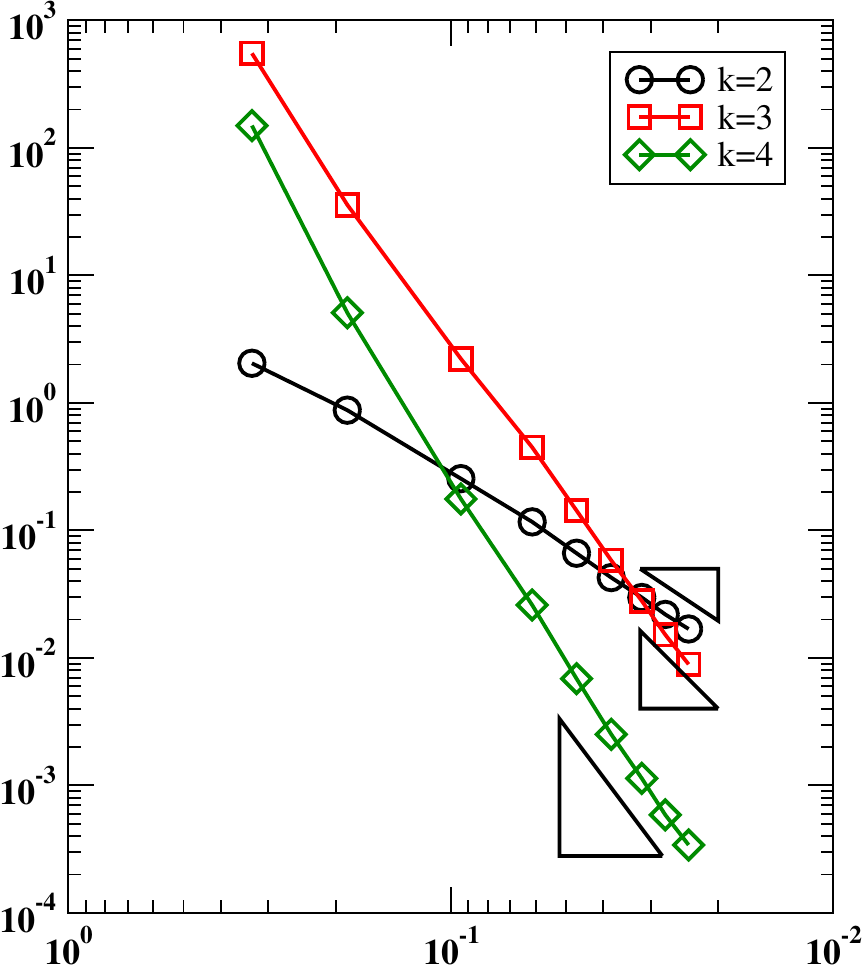} 
      \put( 45, -8){\textbf{h}}
      \put( -8, 14){\begin{sideways}\textbf{$\mathbf{H^1}$ Approximation Error}\end{sideways}}
      \put(75.5,36.5){\begin{small}$\mathbf{2}$\end{small}}
      \put(68,    20){\begin{small}$\mathbf{3}$\end{small}}
      \put(53,  15.5){\begin{small}$\mathbf{4}$\end{small}}
    \end{overpic}
    &\quad
    \begin{overpic}[scale=0.3]{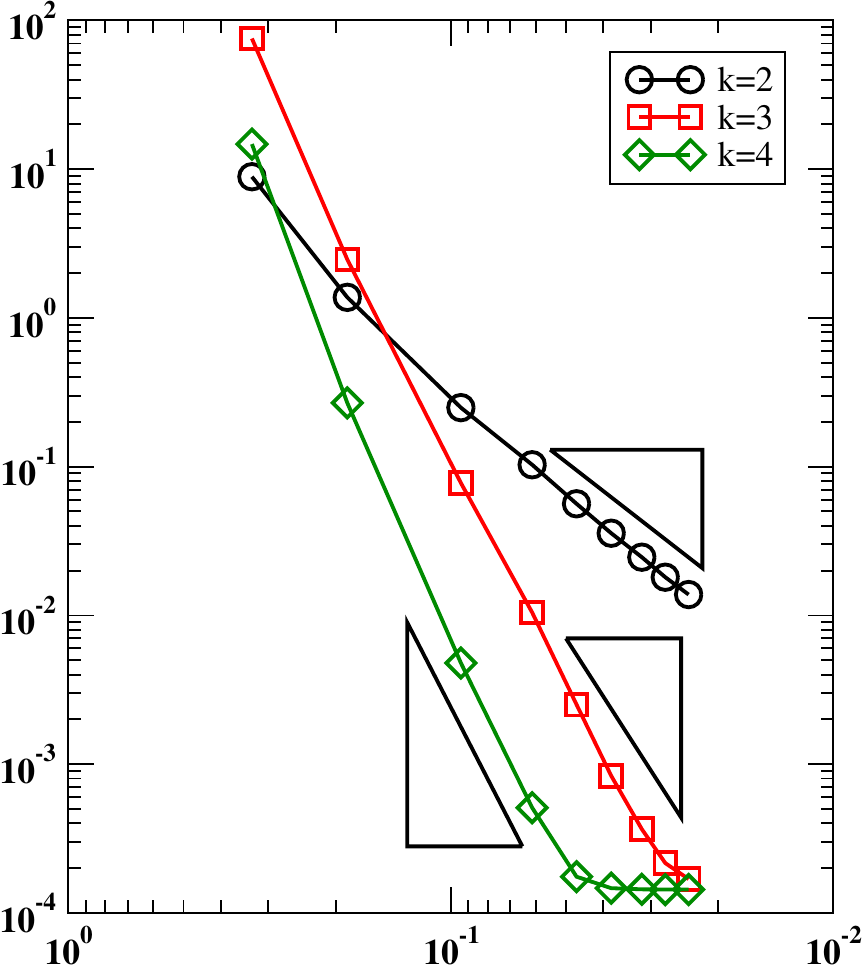} 
      \put( 45, -8){\textbf{h}}
      \put( -8, 14){\begin{sideways}\textbf{$\mathbf{L^2}$ Approximation Error}\end{sideways}}
      \put(74,45){\begin{small}$\mathbf{2}$\end{small}}
      \put(72,24){\begin{small}$\mathbf{4}$\end{small}}
      \put(37,20){\begin{small}$\mathbf{5}$\end{small}}
    \end{overpic}
    \\[0.75cm]
    \begin{overpic}[scale=0.3]{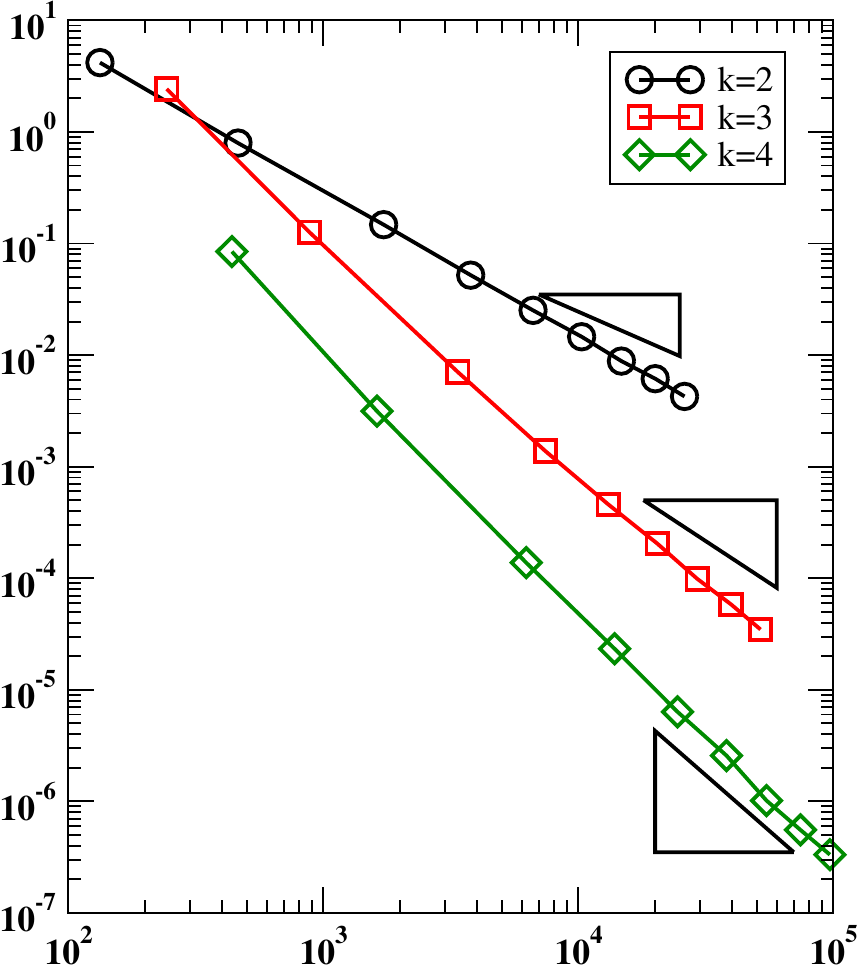} 
      \put( 35, -8){$\Ndfs$}
      \put( -8, 7){\begin{sideways}\textbf{Energy Approximation Error}\end{sideways}}
      \put(70.5,64.5){\begin{small}$\mathbf{\frac12}$\end{small}}
      \put(73,  49.5){\begin{small}$\mathbf{1}$\end{small}}
      \put(62,    16){\begin{small}$\mathbf{\frac32}$\end{small}}
    \end{overpic}
    &\quad
    \begin{overpic}[scale=0.3]{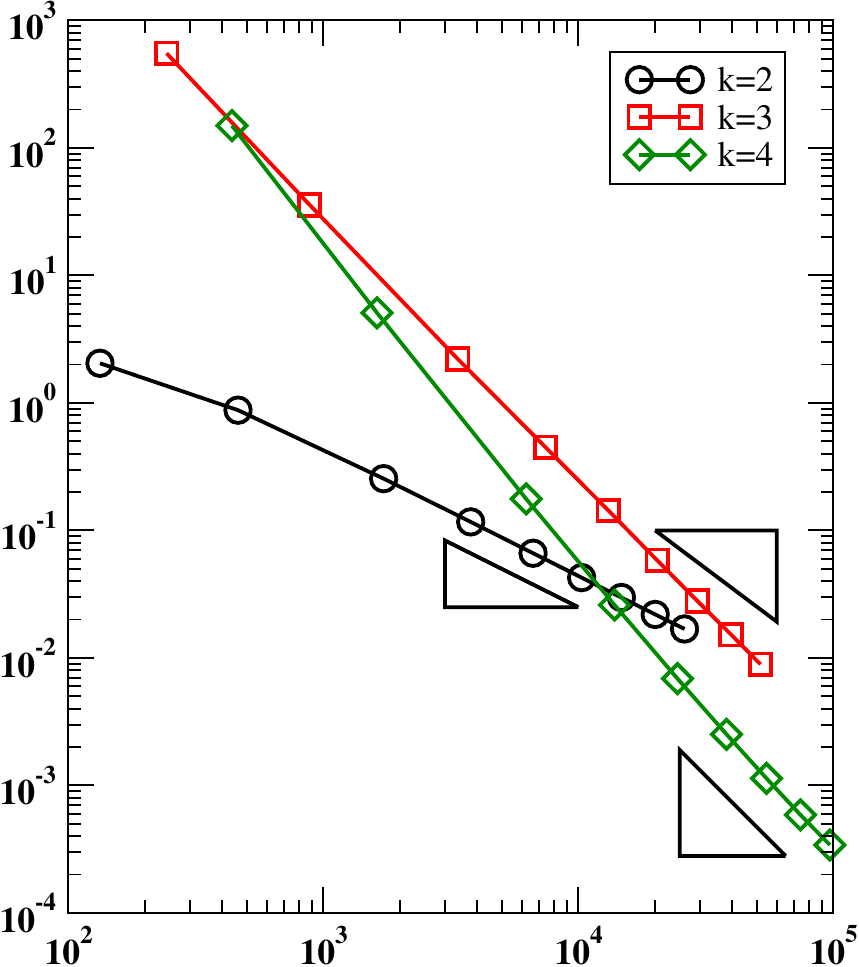} 
      \put( 35, -8){$\Ndfs$}
      \put( -8, 14){\begin{sideways}\textbf{$\mathbf{H^1}$ Approximation Error}\end{sideways}}
      \put(41.5,38.5){\begin{small}$\mathbf{1}$\end{small}}
      \put(72.5,48){\begin{small}$\mathbf{\frac32}$\end{small}}
      \put(65.5,15){\begin{small}$\mathbf{2}$\end{small}}
    \end{overpic}
    &\quad
    \begin{overpic}[scale=0.3]{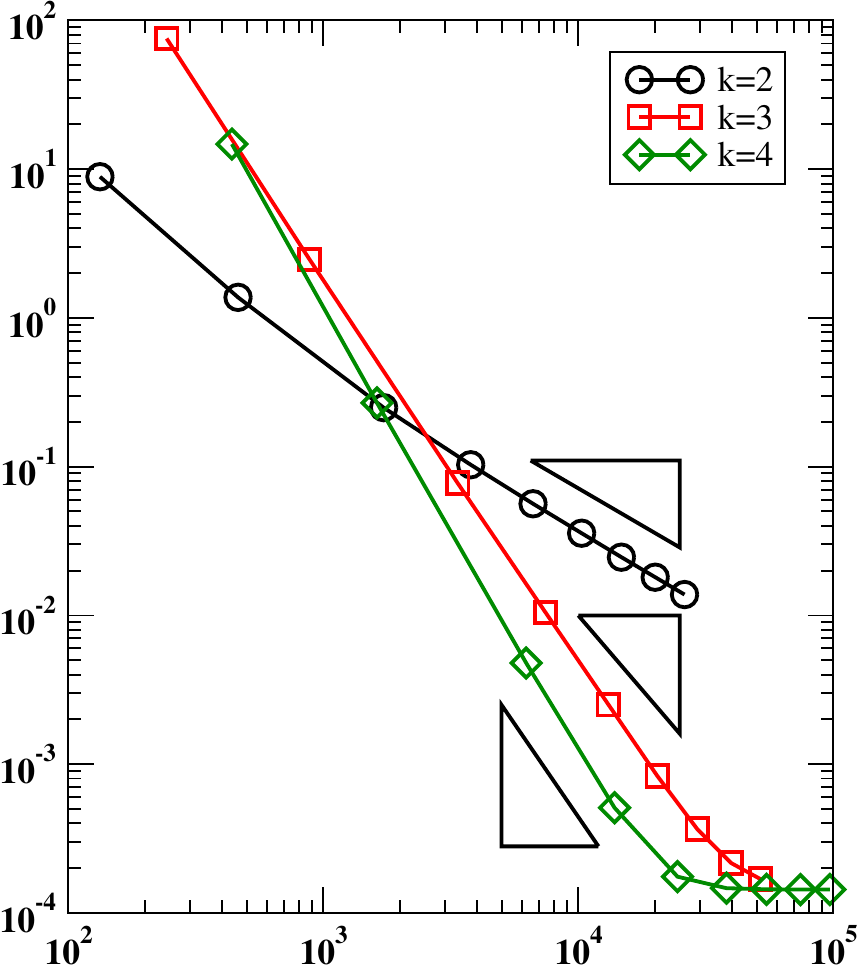} 
      \put( 35, -8){$\Ndfs$}
      \put( -8, 14){\begin{sideways}\textbf{$\mathbf{L^2}$ Approximation Error}\end{sideways}}
      \put(71,46){\begin{small}$\mathbf{1}$\end{small}}
      \put(71.5,29){\begin{small}$\mathbf{2}$\end{small}}
      \put(46.5,17){\begin{small}$\mathbf{\frac52}$\end{small}}
    \end{overpic}
  \end{tabular}
  \caption{Test Case 1. Error curve measured using energy norm, $\HS{2}$ norm,
    $\HS{1}$ norm and $\LS{2}$ norm. The calculations are carried out
    on the family of nonconvex octagonal meshes.}
  \label{fig:Res-Octa-NonConvex}
\end{figure}
\begin{figure}[t] 
\centering
  \begin{tabular}{cc}
    \begin{overpic}[scale=0.45]{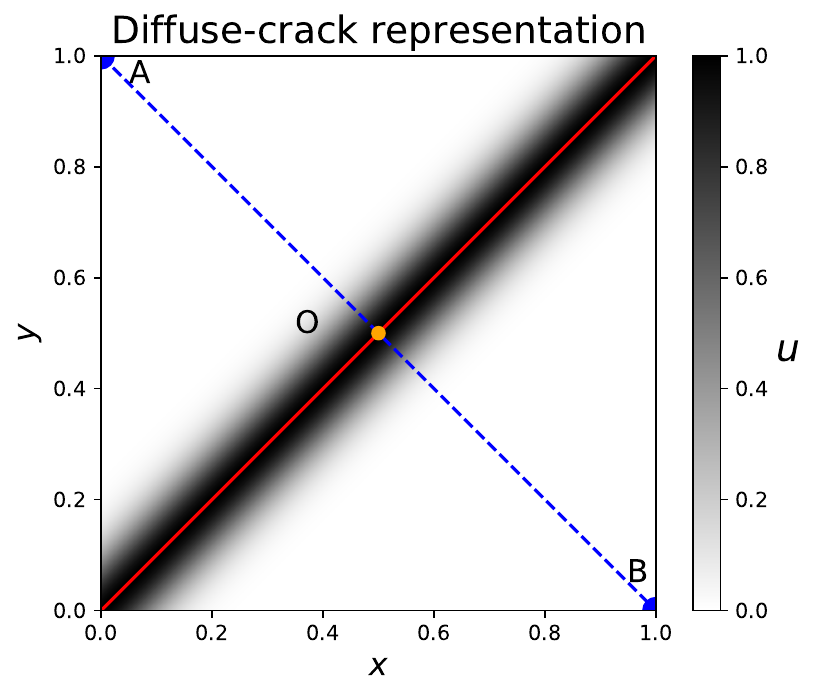} 
    \put( 35,   -6.0){(a)}
    \end{overpic}
    &\quad
    \begin{overpic}[scale=0.45]{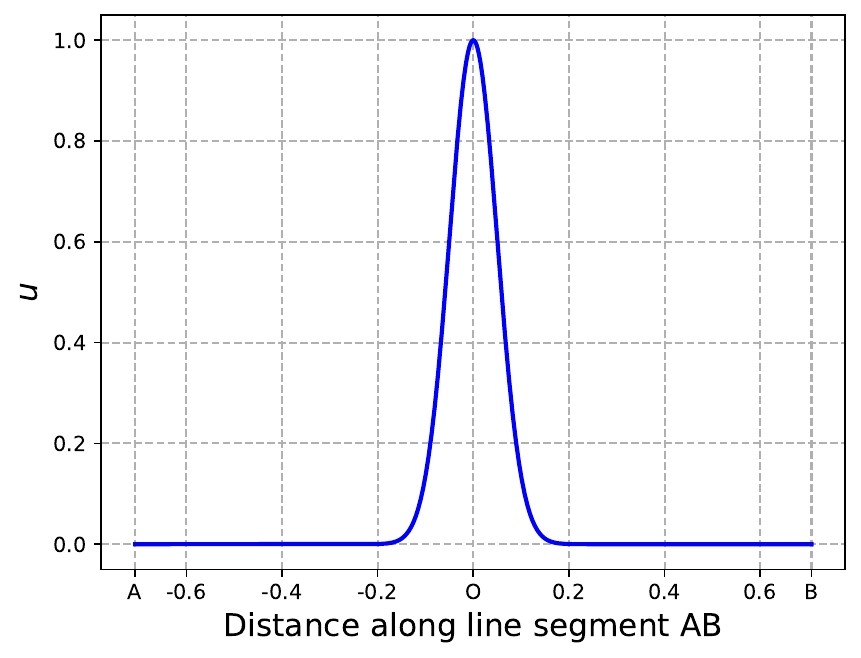} 
    \put( 35,   -6.0){(b)}
    \end{overpic}
  \end{tabular}
\vspace{0.25cm}
\caption{\BLUE{Test Case 2. A manufactured solution to HOPF equation. (a)
  Diffuse representation of a diagonal crack defined by the bisecting
  line $y=x$ (drawn in red color) and (b) PF profile along line
  segment $AB$ for $\epsilon=10^{-2}$.} }
\label{fig:DiffuseCrackRepresentations}
\vspace{0.25cm}
\end{figure}
\begin{figure}[t]
 \centering
  \begin{tabular}{ccc}
    \begin{overpic}[scale=0.3]{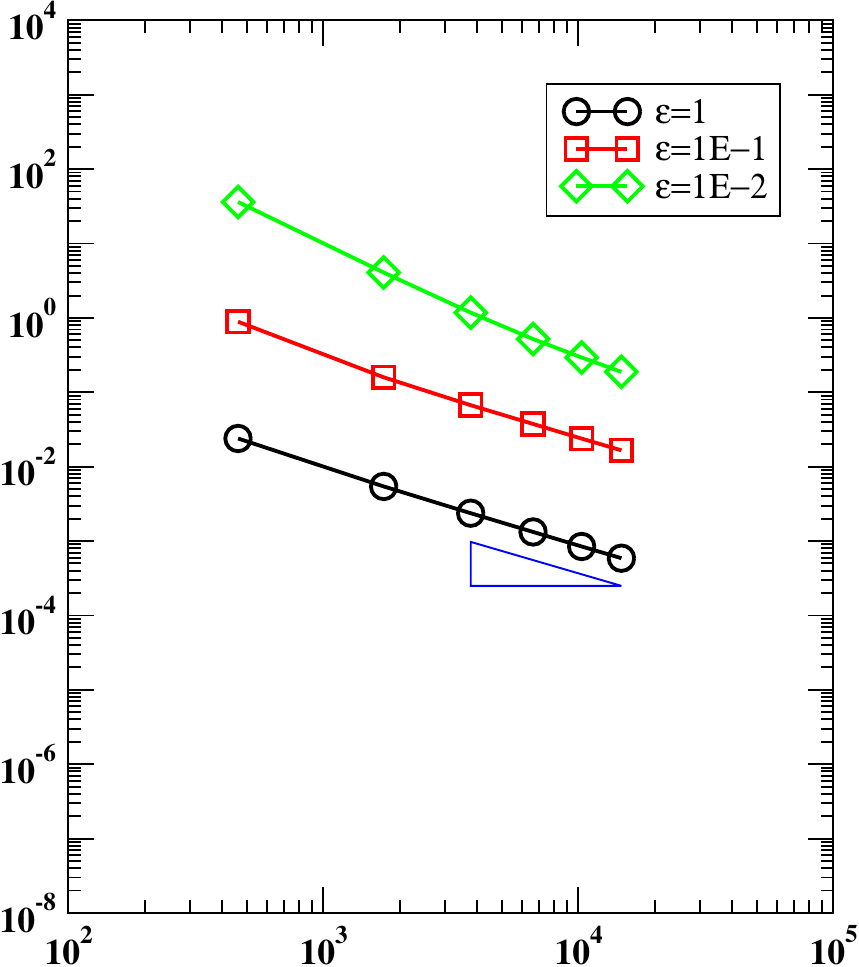} 
      \put( 38,  101.0){\begin{small}$\mathbf{k=2}$\end{small}}
      \put( 35,   -8.0){$\Ndfs$}
      \put( 44,   39.5){\begin{small}$\mathbf{1}$\end{small}}
      \put( -8,   14.0){\begin{sideways}\textbf{$\mathbf{L^2}$ Approximation Error}\end{sideways}}
    \end{overpic}
    &\quad
    \begin{overpic}[scale=0.3]{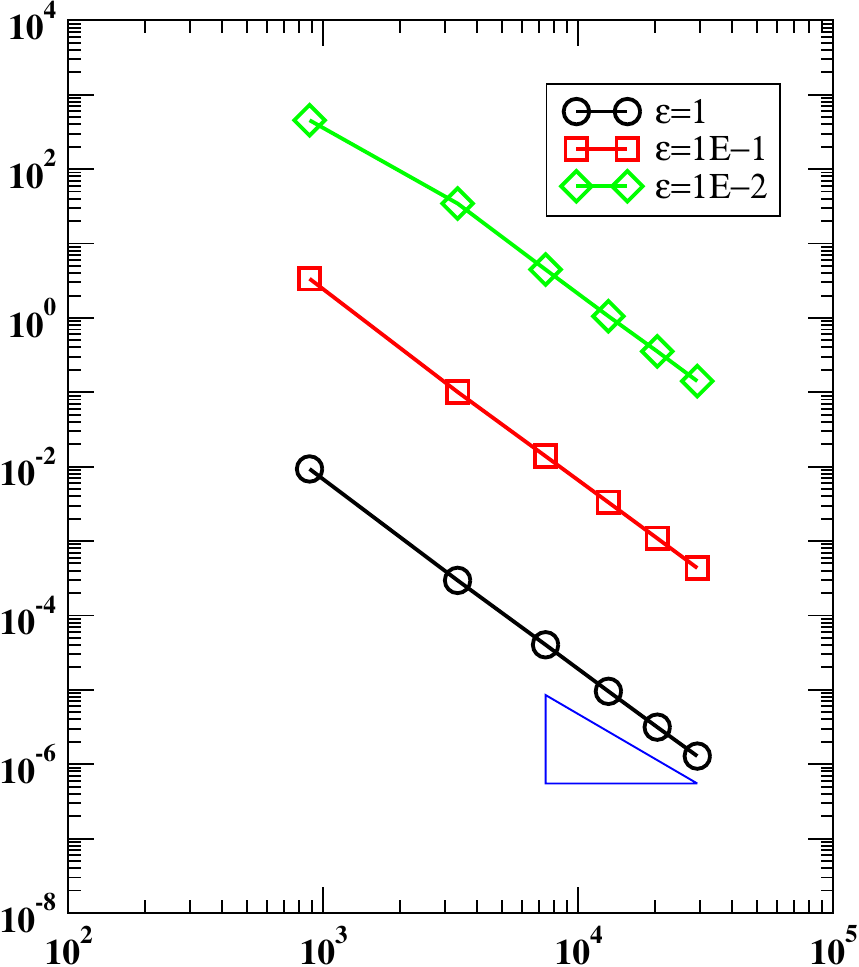} 
      \put( 38,  101.0){\begin{small}$\mathbf{k=3}$\end{small}}
      \put( 35,   -8.0){$\Ndfs$}
      \put( 50,   21.5){\begin{small}$\mathbf{2}$\end{small}}
      \put( -8,   14.0){\begin{sideways}\textbf{$\mathbf{L^2}$ Approximation Error}\end{sideways}}
    \end{overpic}
    &\quad
    \begin{overpic}[scale=0.3]{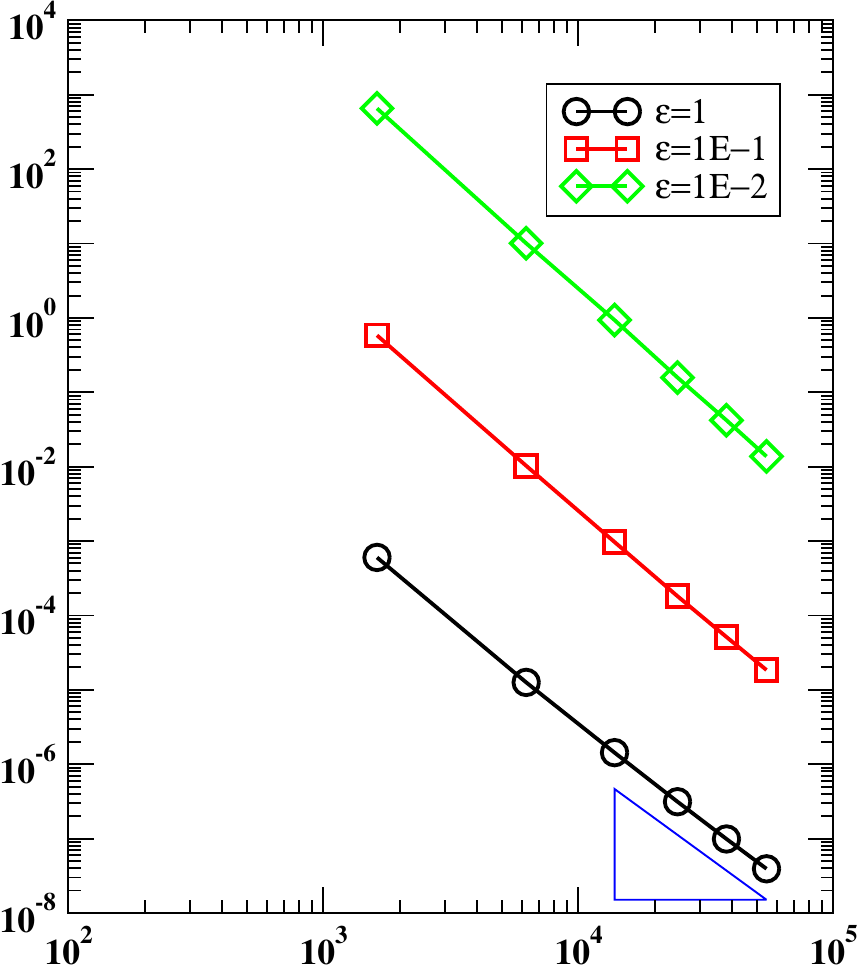} 
      \put( 38,  101.0){\begin{small}$\mathbf{k=4}$\end{small}}
      \put( 35,   -8.0){$\Ndfs$}
      \put( 56.5, 11.0){\begin{small}$\mathbf{\frac52}$\end{small}}
      \put( -8,   14.0){\begin{sideways}\textbf{$\mathbf{L^2}$ Approximation Error}\end{sideways}}
    \end{overpic}
  \end{tabular}
  \vspace{0.25cm}
  \caption{\BLUE{Test Case 2. Error curve measured using the $\LS{2}$
      norm for $\epsilon\in\big\{1,10^{-1},10^{-2}\big\}$.. The
      calculations are carried out on the family of randomized
      quadrilateral meshes using the virtual element method for
      $k=2,3,4$.}}
  \label{fig:TestCase2:Res-Quads-Randomized}
  \vspace{0.25cm}
\end{figure}

\begin{figure}[t]
 \centering
  \begin{tabular}{ccc}
    \begin{overpic}[scale=0.3]{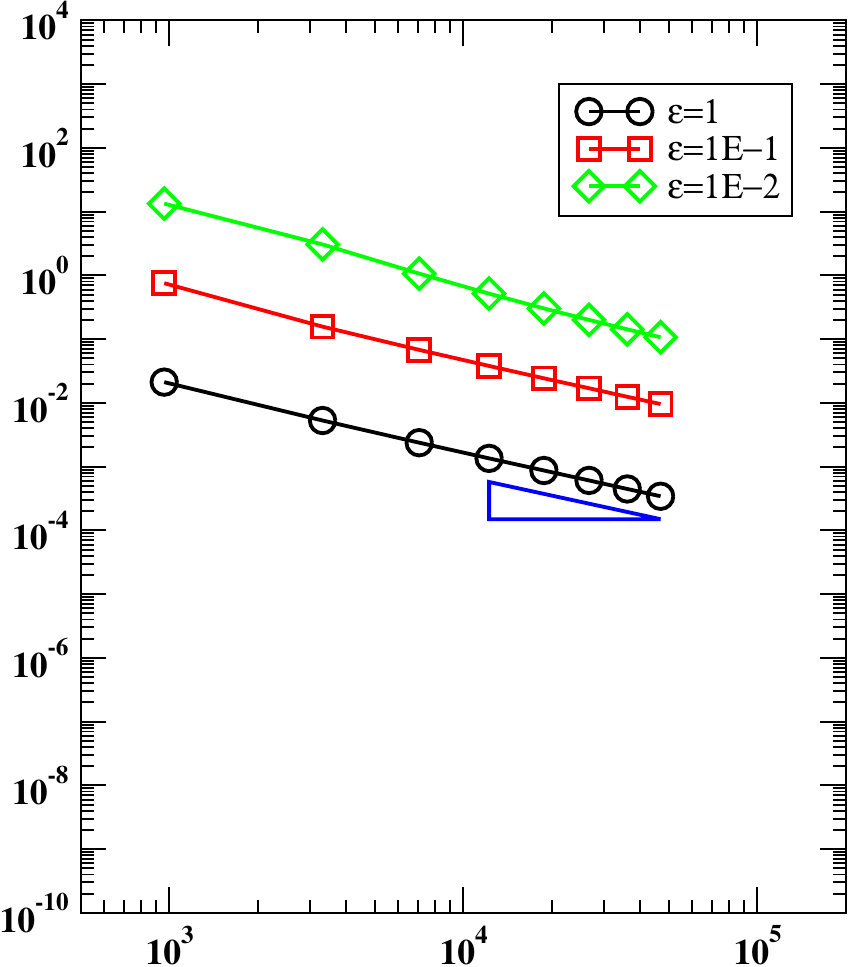} 
      \put( 38,  101.0){\begin{small}$\mathbf{k=2}$\end{small}}
      \put( 35,   -8.0){$\Ndfs$}
      \put( 45.5, 46.5){\begin{small}$\mathbf{1}$\end{small}}
      \put( -8,   14.0){\begin{sideways}\textbf{$\mathbf{L^2}$ Approximation Error}\end{sideways}}
    \end{overpic}
    &\quad
    \begin{overpic}[scale=0.3]{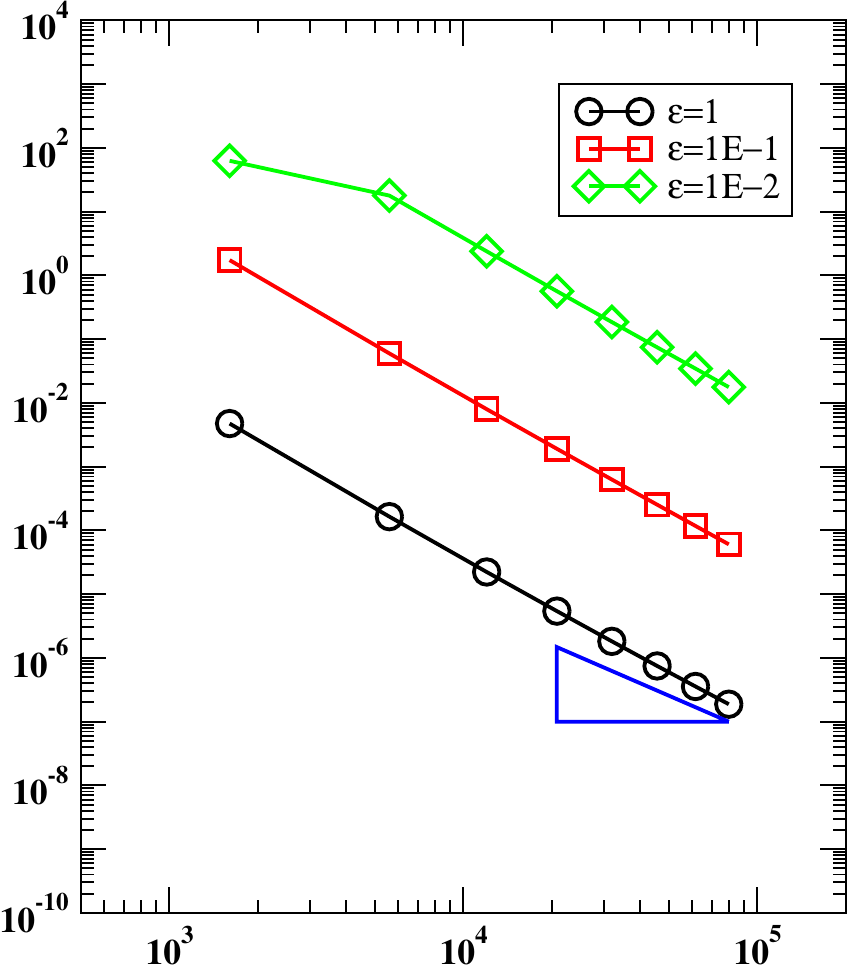} 
      \put( 38,  101.0){\begin{small}$\mathbf{k=3}$\end{small}}
      \put( 35,   -8.0){$\Ndfs$}
      \put( 51,   27.0){\begin{small}$\mathbf{2}$\end{small}}
      \put( -8,   14.0){\begin{sideways}\textbf{$\mathbf{L^2}$ Approximation Error}\end{sideways}}
    \end{overpic}
    &\quad
    \begin{overpic}[scale=0.3]{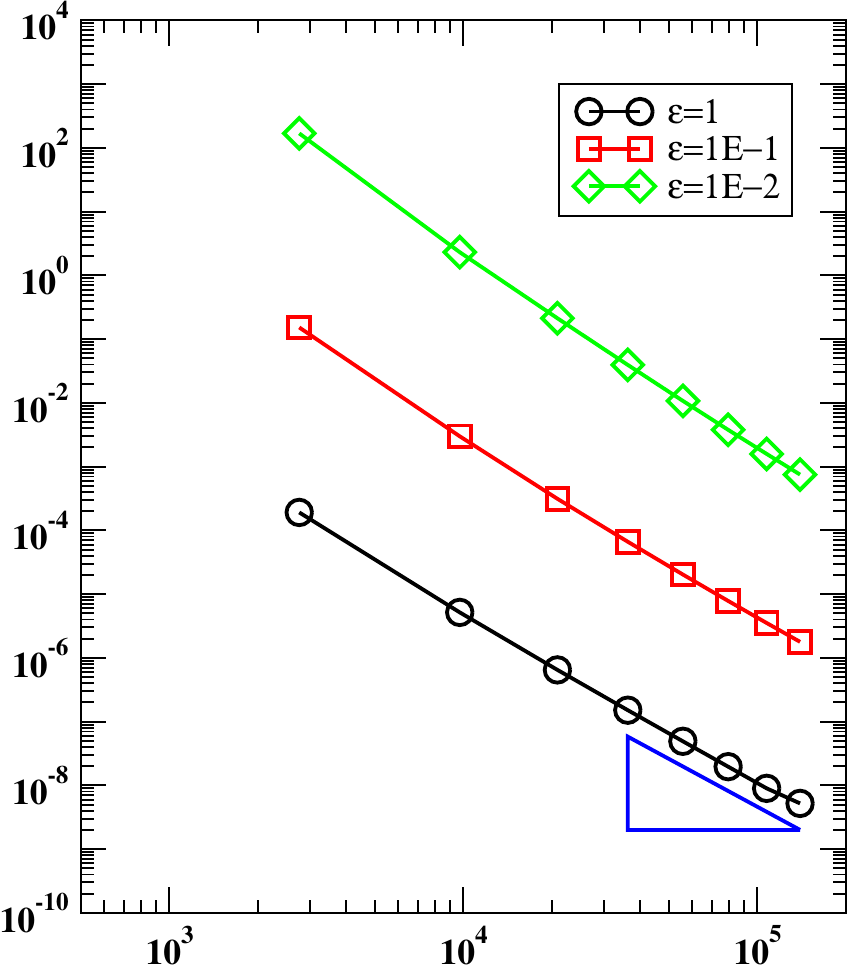} 
      \put( 38,  101.0){\begin{small}$\mathbf{k=4}$\end{small}}
      \put( 35,   -8.0){$\Ndfs$}
      \put( 58.5, 17.0){\begin{small}$\mathbf{\frac52}$\end{small}}
      \put( -8,   14.0){\begin{sideways}\textbf{$\mathbf{L^2}$ Approximation Error}\end{sideways}}
    \end{overpic}
  \end{tabular}
  \vspace{0.25cm}
  \caption{\BLUE{Test Case 2. Error curve measured using the $\LS{2}$
      norm for $\epsilon\in\big\{1,10^{-1},10^{-2}\big\}$. The
      calculations are carried out on the family of smoothly remapped
      hexagonal meshes using the virtual element method for
      $k=2,3,4$.}}
  \label{fig:TestCase2:Res-Hexa-Remapped}
\end{figure}


The nonhomogeneous Dirichlet boundary conditions and the source term
are set on the computational domain
$\overline{\Omega}=[0,1]\times[0,1]$ in accordance with the exact
solutions
\begin{align*}
  \us(x,y) &= \sin(2\pi\xs)\sin(2\pi\ys) + \xs^5 + \ys^5.\\[0.5em]
\end{align*}
We apply the VEM on the four mesh families introduced in the previous
section.

\medskip
We show the log-log plots of the error curves for the four mesh
families in Figures \ref{fig:Res-Quads-Remapped},
\ref{fig:Res-Quads-Randomized}, \ref{fig:Res-Hexa-Remapped}, and
\ref{fig:Res-Octa-NonConvex}.
In every figure, the three panels on top show the error curves versus
the mesh size parameter $\hh$; the three panels on bottom show the
error curves versus the number of degrees of freedom, which is denoted
by $\Ndfs$ on the axis of the plots.
The slopes of these curves reflects the numerical convergence rate,
while the triangles shown near such curves show the theoretical
convergence rate according to the results of Theorems~\ref{ThmL2},
\ref{ThmH1}, and \ref{ThmH2}.
We recall that the convergence rate with respect to $\Ndfs$ is equal
to the convergence rate with respect to $\hh$ divided by two since we
roughly have that $\Ndfs\simeq\hh^{-2}$.
On some of the coarser meshes, the $\LS{2}$-norm errors for
\BLUE{$k=2$} are smaller than those produced by the VEM with
\BLUE{$k=3,4$}.
This phenomenon probably occurs because the numerical method is
working in a pre-asymptotic region.
However, these numerical results are in good agreement with the
theoretical predictions, and they show that the VEM proposed in this
work provides the expected optimal convergence rate in the energy norm
and lower-order norms on all the mesh families considered in this
benchmark test.
Since we can assume that the cost of solving the numerical method is
somehow proportional to the number of degrees of freedom $\Ndfs$, the
bottom plots show that on the finer meshes, the higher-order schemes
are more convenient to achieve a pre-fixed accuracy level.
\RED{We also note a \emph{locking effect} in the convergence rate for
  $\ks=4$ when we measure the error using the $\LTWO$-norm.
  This effect is visible in all the meshes considered in our numerical
  experiments.
  This loss of convergence when the mesh is highly refined is very
  likely related to the increasing ill-conditioning of the linear
  system that we need to solve and will be the subject of further
  investigation.  }

\subsection{\BLUE{Test Case~2. Diagonal crack example: Length-scale sensitivity}}
\label{subsec:2nd-test-case}

\BLUE{To showcase the merits of the VEM in PF fracture problems, a
  manufactured solution to HOPF equation is assumed as follows }
\begin{align*}
  \us(x,y) &= \exp( -(x-y)^2/\epsilon ),
  \qquad
  \epsilon\in\big\{1,10^{-1},10^{-2}\big\}.
\end{align*}
\BLUE{ As illustrated in Figure \ref{fig:DiffuseCrackRepresentations}(a), this function can be conceived as a PF approximation of a
  diagonal crack defined by the bisecting line $y=x$ and drawn in red color in the
  computational domain $\Omega=(0,1)\times(0,1)$.
  The length-scale parameter $\epsilon$ controls the width of the PF
  diffuse crack representation and can be decreased to approximate the
  sharp-crack limit.
  Figure \ref{fig:DiffuseCrackRepresentations}(b) shows the solution $u$ along line segment $AB$ for $\epsilon=10^{-2}$, where $A$ and $B$ are points depicted in \ref{fig:DiffuseCrackRepresentations}(a).
  The non-homogeneous Dirichlet boundary conditions and source term
  shown in equation (1) are set in accordance with this exact
  solution.}

\BLUE{In Figure \ref{fig:TestCase2:Res-Quads-Randomized}, the proposed
  VEM is applied to the aforementioned problem on randomized
  quadrilateral meshes (\ref{fig:mesh:family}(b)) for different values
  of $\epsilon\in\big\{1,10^{-1},10^{-2}\big\}$.
  The convergence curves are shown for the polynomial orders
  $k=2,3,4$, from left to right, respectively. Similar numerical
  behavior is observed in Figure~\ref{fig:TestCase2:Res-Hexa-Remapped}
  for the VEM solution on smoothly remapped hexagonal meshes
  (\ref{fig:mesh:family}(c)).}

These results show that optimal convergence rates are maintained by
the VEM even for very sharp crack profiles, although the error constant
increases, thus needing higher refinements to achieve a given solution
accuracy.



\section{Conclusions and final remark}
\label{sec6:conclusions}

\noindent
In this work, we proposed a $\CS{1}$-regular virtual element method
for the numerical approximation of a fourth-order phase-field
equation.
The virtual element formulation presented in this work employs a
polyomial projection operator that combines the biharmonic and Laplace
differential operators.
We proved the convergence of the method and derived optimal
convergence rates in different norms.
The theoretical convergence results were confirmed by conducting
numerical experiments on a set of four polygonal meshes.
The good approximation properties of the proposed VEM are also
demonstrated in the context of a problem involving HOPF representation of a crack.

Developing analogous discrete spaces and carrying out the \emph{a
priori} analyses for higher dimension model problems will be the topic
of future research.
\BLUE{ Furthermore, coupling the HOPF evolution equation with the
  momentum balance equation forms the basis of HOPF models of dynamic
  fracture.
  With this in mind, in future work, we will be focusing our efforts on
  coupling the VEM that we recently developed for the momentum equation
  in \cite{Antonietti-Manzini-Mazzieri-Mourad-Verani:2021} with the VEM
  presented herein.
  }



\section*{Acknowledgments}
The authors gratefully acknowledge the support of the Laboratory
Directed Research and Development (LDRD) program of Los Alamos
National Laboratory under project number 20220129ER.
Los Alamos National Laboratory is operated by Triad National Security,
LLC, for the National Nuclear Security Administration of
U.S. Department of Energy (Contract No.\ 89233218CNA000001).
This work is registered as the Los Alamos Technical Report
\emph{LA-UR-21-31995}.




\clearpage
\section*{Statements \& Declarations}

\PGRAPH{Funding}: Laboratory Directed Research and Development (LDRD)
program operated at Los Alamos National Laboratory, Grant
N. 20220129ER

\PGRAPH{Competing Interests}: All the Authors have no financial
interests.

\PGRAPH{Author Contributions}: All the Authors equally contributed to
this work.

\PGRAPH{Data Availability}: No new data were specifically generated
for this work. The meshes utilized for this study and reported in the
final appendix are the property of Los Alamos National Laboratory and
can be made available on reasonable request.


\clearpage


\appendix
\section{Mesh data}

For completeness, we report the data corresponding to the four mesh
families that we used in the calculation of Section{sec6:numerical:results}.
The four tables report the following data:
\begin{itemize}
\item $\ilev$: refinement level;
\item $\nR$, $\nF$, $\nV$: number of elements, edges, and vertices;
\item $\hmax$: mesh size parameter;
\item $\ndofs_{k=\ell}$: total number of degrees of freedom for the
  VEM with polynomial degree $\ell=2,3,4$.
\end{itemize}

\begin{table}[ht]
  \begin{center}
    \begin{tabular}{|c|ccc|c|ccc|}
      \cline{1-8}
      $\ilev$ & $\nR$ & $\nF$ & $\nV$ & $\hmax$ & $\ndofs_{k=2}$ & $\ndofs_{k=3}$ & $\ndofs_{k=4}$ \\
      \cline{1-8}
      $0$ &   $25$ &    $60$ &   $36$ &  $3.788\,10^{-1}$ &   $133$ &   $243$  &   $438$ \\
      $1$ &  $100$ &   $220$ &  $121$ &  $2.007\,10^{-1}$ &   $463$ &   $883$  &  $1623$ \\
      $2$ &  $400$ &   $840$ &  $441$ &  $1.035\,10^{-1}$ &  $1723$ &  $3363$  &  $6243$ \\
      $3$ &  $900$ &  $1860$ &  $961$ &  $6.907\,10^{-2}$ &  $3783$ &  $7443$  & $13863$ \\
      $4$ & $1600$ &  $3280$ & $1681$ &  $5.195\,10^{-2}$ &  $6643$ & $13123$  & $24483$ \\
      $5$ & $2500$ &  $5100$ & $2601$ &  $4.155\,10^{-2}$ & $10303$ & $20403$  & $38103$ \\
      $6$ & $3600$ &  $7320$ & $3721$ &  $3.466\,10^{-2}$ & $14763$ & $29283$  & $54723$ \\
      $7$ & $4900$ &  $9940$ & $5041$ &  $2.970\,10^{-2}$ & $20023$ & $39763$  & $74343$ \\
      $8$ & $6400$ & $12960$ & $6561$ &  $2.600\,10^{-2}$ & $26083$ & $51843$  & $96963$ \\
      \cline{1-8}
    \end{tabular}
  \end{center}
  \caption{Mesh data and number of degrees of freedom for the VEM with
    polynomial degree $k=2,3,4$ for the family of smoothly remapped
    quadrilateral meshes.}
  \label{tab:Res-Quads-Remapped}
\end{table}

\begin{table}[ht]
  \begin{center}
    \begin{tabular}{|c|ccc|c|ccc|}
      \cline{1-8}
      $\ilev$ & $\nR$ & $\nF$ & $\nV$ & $\hmax$ & $\ndofs_{k=2}$ & $\ndofs_{k=3}$ & $\ndofs_{k=4}$ \\
      \cline{1-8}
      $0$ &   $25$ &    $60$ &   $36$ & $3.311\,10^{-1}$ &   $133$ &   $243$ &   $438$ \\
      $1$ &  $100$ &   $220$ &  $121$ & $1.865\,10^{-1}$ &   $463$ &   $883$ &  $1623$ \\
      $2$ &  $400$ &   $840$ &  $441$ & $9.412\,10^{-2}$ &  $1723$ &  $3363$ &  $6243$ \\
      $3$ &  $900$ &  $1860$ &  $961$ & $6.130\,10^{-2}$ &  $3783$ &  $7443$ & $13863$ \\
      $4$ & $1600$ &  $3280$ & $1681$ & $4.693\,10^{-2}$ &  $6643$ & $13123$ & $24483$ \\
      $5$ & $2500$ &  $5100$ & $2601$ & $3.808\,10^{-2}$ & $10303$ & $20403$ & $38103$ \\
      $6$ & $3600$ &  $7320$ & $3721$ & $3.167\,10^{-2}$ & $14763$ & $29283$ & $54723$ \\
      $7$ & $4900$ &  $9940$ & $5041$ & $2.751\,10^{-2}$ & $20023$ & $39763$ & $74343$ \\
      $8$ & $6400$ & $12960$ & $6561$ & $2.389\,10^{-2}$ & $26083$ & $51843$ & $96963$ \\
      \cline{1-8}
    \end{tabular}
  \end{center}
  \caption{Mesh data and number of degrees of freedom for the VEM with
    polynomial degree $k=2,3,4$ for the family of randomized
    quadrilateral meshes.}
  \label{tab:Res-Quads-Randomized}
\end{table}

\begin{table}[ht]
  \begin{center}
    \begin{tabular}{|c|ccc|c|ccc|}
      \cline{1-8}
      $\ilev$ & $\nR$ & $\nF$ & $\nV$ & $\hmax$ & $\ndofs_{k=2}$ & $\ndofs_{k=3}$ & $\ndofs_{k=4}$ \\
      \cline{1-8}
      $0$ &   $36$ &   $125$ &    $90$ & $3.279\,10^{-1}$ &   $306$ &   $503$ &    $861$ \\
      $1$ &  $121$ &   $400$ &   $280$ & $1.846\,10^{-1}$ &   $961$ &  $1603$ &   $2766$ \\
      $2$ &  $441$ &  $1400$ &   $960$ & $9.686\,10^{-2}$ &  $3321$ &  $5603$ &   $9726$ \\
      $3$ &  $961$ &  $3000$ &  $2040$ & $6.492\,10^{-2}$ &  $7081$ & $12003$ &  $20886$ \\
      $4$ & $1681$ &  $5200$ &  $3520$ & $4.889\,10^{-2}$ & $12241$ & $20803$ &  $36246$ \\
      $5$ & $2601$ &  $8000$ &  $5400$ & $3.914\,10^{-2}$ & $18801$ & $32003$ &  $55806$ \\
      $6$ & $3721$ & $11400$ &  $7680$ & $3.265\,10^{-2}$ & $26761$ & $45603$ &  $79566$ \\
      $7$ & $5041$ & $15400$ & $10360$ & $2.799\,10^{-2}$ & $36121$ & $61603$ & $107526$ \\
      $8$ & $6561$ & $20000$ & $13440$ & $2.451\,10^{-2}$ & $46881$ & $80003$ & $139686$ \\
      \cline{1-8}
    \end{tabular}
  \end{center}
  \caption{Mesh data and number of degrees of freedom for the VEM with
    polynomial degree $k=2,3,4$ for the family of smoothly remapped
    hexagons.}
  \label{tab:Res-Hexa-Remapped}
\end{table}

\begin{table}[ht]
  \begin{center}
    \begin{tabular}{|c|ccc|c|ccc|}
      \cline{1-8}
      $\ilev$ & $\nR$ & $\nF$ & $\nV$ & $\hmax$ & $\ndofs_{k=2}$ & $\ndofs_{k=3}$ & $\ndofs_{k=4}$ \\
      \cline{1-8}
      $0$ &   $25$ &   $120$ &    $96$ & $2.915\,10^{-1}$ &   $313$ &    $483$ &    $798$ \\
      $1$ &  $100$ &   $440$ &   $341$ & $1.458\,10^{-1}$ &  $1123$ &   $1763$ &   $2943$ \\
      $2$ &  $400$ &  $1680$ &  $1281$ & $7.289\,10^{-2}$ &  $4243$ &   $6723$ &  $11283$ \\
      $3$ &  $900$ &  $3720$ &  $2821$ & $4.859\,10^{-2}$ &  $9363$ &  $14883$ &  $25023$ \\
      $4$ & $1600$ &  $6560$ &  $4961$ & $3.644\,10^{-2}$ & $16483$ &  $26243$ &  $44163$ \\
      $5$ & $2500$ & $10200$ &  $7701$ & $2.915\,10^{-2}$ & $25603$ &  $40803$ &  $68703$ \\
      $6$ & $3600$ & $14640$ & $11041$ & $2.430\,10^{-2}$ & $36723$ &  $58563$ &  $98643$ \\
      $7$ & $4900$ & $19880$ & $14981$ & $2.082\,10^{-2}$ & $49843$ &  $79523$ & $133983$ \\
      $8$ & $6400$ & $25920$ & $19521$ & $1.822\,10^{-2}$ & $64963$ & $103683$ & $174723$ \\
      \cline{1-8}
    \end{tabular}
  \end{center}
    \caption{Mesh data and number of degrees of freedom for the VEM
      with polynomial degree $k=2,3,4$ for the family of nonconvex
      octagons.}
    \label{tab:Res-Octa-NonConvex}
\end{table}


\end{document}